\documentclass[10pt, a4paper,
oneside, headinclude,footinclude]{scrartcl}
\usepackage[utf8]{inputenc}

\usepackage[
nochapters, % Turn off chapters since this is an article        
pdfspacing, % Makes use of pdftex’ letter spacing capabilities via the microtype package
dottedtoc % Dotted lines leading to the page numbers in the table of contents
]{classicthesis} % The layout is based on the Classic Thesis style
\usepackage{amsopn}
\usepackage[T1]{fontenc} % Use 8-bit encoding that has 256 glyphs
\usepackage{multirow}
\usepackage[utf8]{inputenc} % Required for including letters with accents
\usepackage{hhline}
\usepackage{graphicx} % Required for including images
\graphicspath{{Figures/}} % Set the default folder for images
\usepackage{indentfirst}
\usepackage{enumitem} % Required for manipulating the whitespace between and within lists
\usepackage{savesym}
\usepackage{mathrsfs}
\usepackage{geometry}
\usepackage{esint}
\usepackage{longtable}
\usepackage{float}
\usepackage{adjustbox}

\usepackage[
    backend=biber,
    style=numeric,
    natbib=false,
    url=true, 
    doi=false,
    eprint=true,
    maxnames=50
]{biblatex}

\usepackage{stmaryrd}

\usepackage{subfig} % Required for creating figures with multiple parts (subfigures)

\usepackage{amsmath, amssymb,amsthm,amsfonts} % For including math equations, theorems, symbols, etc

\usepackage[english,capitalize]{cleveref}

\usepackage{thmtools}
\usepackage{thm-restate}

\theoremstyle{plain}
\newtheorem{teorema}{Theorem}[section]
\newtheorem{proposizione}[teorema]{Proposition}
\newtheorem{lemma}[teorema]{Lemma}

\newtheorem*{theorem*}{Theorem}

\theoremstyle{definition}
\newtheorem{definizione}{Definition}[section]

\theoremstyle{remark}
\newtheorem{osservazione}{Remark}[section]

\newcommand{\Mass}{\mathbb{M}}
\newcommand{\R}{\mathbb{R}}
\newcommand{\N}{\mathbb{N}}
\newcommand{\GG}{\mathbb{G}}
\newcommand{\HH}{\mathbb{H}}
\newcommand{\im}{\mathrm{im}}
\newcommand{\Haus}{\mathscr{H}}
\newcommand{\Leb}{\mathscr{L}}

\newcommand{\F}{\mathscr{F}}
\newcommand{\G}{\mathscr{G}}

\newcommand{\dom}{\mathrm{dom}}
\newcommand{\M}{\mathscr{M}}

\newcommand{\bd}{\partial}

\newcommand{\cc}{\text{Carnot-Carath\'eodory }}
\newcommand{\Lip}{\mathrm{Lip}}

\newcommand{\Tan}{\mathrm{Tan}}
\newcommand{\Int}{\mathrm{Int}}

\newcommand{\supp}{\mathrm{supp}}
\newcommand{\Gr}{\mathrm{Gr}}

\newcommand{\wrt}{w.r.t.\ }
\newcommand{\eps}{\varepsilon}
\newcommand{\dV}{d_V\kern-1pt}
\newcommand{\dW}{d_W\kern-1pt}

\newcommand{\scal}[2]{\langle #1\, ; \, #2\rangle}

\DeclareMathOperator{\largewedge}{\mbox{\large$\wedge$}}
\newcommand{\trait}[3]{\vrule width #1ex height #2ex depth #3ex}
\newcommand{\trace}{\mathchoice%
  {\mathbin{\trait{.12}{1.2}{.03}\trait{.8}{0.09}{0.03}}}
  {\mathbin{\trait{.12}{1.2}{.03}\trait{.8}{0.09}{0.03}}}
  {\mathbin{\hskip.15ex\trait{.09}{.84}{0.02}\trait{.56}{.07}{.02}}\hskip.15ex}
  {\mathbin{\trait{.07}{.6}{.01}\trait{.4}{.06}{.01}}}}
\newenvironment{itemizeb}
{\begin{itemize}\itemsep=2pt}
{\end{itemize}}
\newcounter{const}
\newcommand{\newC}{\refstepcounter{const}\ensuremath{C_{\theconst}}}
\newcommand{\oldC}[1]{\ensuremath{C_{\ref{#1}}}}

\newcounter{eps}

\newcommand{\vertiii}[1]{{\left\vert\kern-0.25ex\left\vert\kern-0.25ex\left\vert #1 
    \right\vert\kern-0.25ex\right\vert\kern-0.25ex\right\vert}}

\hypersetup{
colorlinks=true, breaklinks=true,bookmarksnumbered,
urlcolor=webbrown, linkcolor=RoyalBlue, citecolor=webgreen, % Link colors
pdftitle={}, % PDF title
pdfauthor={\textcopyright}, % PDF Author
pdfsubject={}, % PDF Subject
pdfkeywords={}, % PDF Keywords
pdfcreator={pdfLaTeX}, % PDF Creator
pdfproducer={LaTeX with hyperref and ClassicThesis} % PDF producer
} 

\title{\normalfont\spacedallcaps{On the converse of Pansu's Theorem}} 
\author{\spacedlowsmallcaps{Guido De Philippis\textsuperscript{\dag}, Andrea Marchese\textsuperscript{*}, Andrea Merlo\textsuperscript{**},}\\ \spacedlowsmallcaps{Andrea Pinamonti\textsuperscript{*}, Filip Rindler\textsuperscript{\dag\dag}}}

\date{}

\begin{document}

\renewcommand{\sectionmark}[1]{\markright{\spacedlowsmallcaps{#1}}} 
\lehead{\mbox{\llap{\small\thepage\kern1em\color{halfgray} \vline}\color{halfgray}\hspace{0.5em}\rightmark\hfil}} 
\pagestyle{scrheadings}
\maketitle 
\setcounter{tocdepth}{2}

{\let\thefootnote\relax\footnotetext{{\dag} 
\emph{Courant Institute of Mathematical Sciences}, New York University, 251
Mercer St., New York, NY 10012, USA.
\textit{.}}}
{\let\thefootnote\relax\footnotetext{* \textit{Università di Trento, Via Sommarive 14, 38123 Trento,  Italy.}}}
{\let\thefootnote\relax\footnotetext{** \textit{Universidad del Pa\'is Vasco (UPV/EHU), Barrio Sarriena S/N 48940 Leioa, Spain.}}}
{\let\thefootnote\relax\footnotetext{{\dag\dag} \textit{Mathematics Institute, University of Warwick, Coventry CV4 7AL, United Kingdom.}}}

{\rightskip 1 cm
\leftskip 1 cm
\parindent 0 pt
\footnotesize

	% abstract, keywords and MSC numbers
	%
{\textsc Abstract.}
We provide a suitable generalisation of Pansu's differentiability theorem to general Radon measures on Carnot groups and we show that if Lipschitz maps between Carnot groups are Pansu-differentiable almost everywhere for some Radon measures $\mu$, then $\mu$ must be absolutely continuous with respect to the Haar measure of the group. 
\par
\medskip\noindent
{\textsc Keywords:} Lipschitz functions, Carnot groups, differentiability, 
Pansu's theorem, Normal currents, Smirnov's theorem.
\par
\medskip\noindent
{\textsc MSC (2010):} 
26B05, 49Q15, 26A27, 28A75, 53C17.
\par
}

%\tableofcontents

%%%%%%%%%%%%%%%%%%%%%%%%%%%%%%%%%%%%%%%%%%%%%%%%%%%%%%%%%%%%%%%%%
%
%	INTRODUCTION 
%
%%%%%%%%%%%%%%%%%%%%%%%%%%%%%%%%%%%%%%%%%%%%%%%%%%%%%%%%%%%%%%%%%

\section{Introduction}
\label{s1}

Rademacher's theorem asserts that Lipschitz functions defined on the Euclidean space are differentiable almost everywhere with respect to the Lebesgue measure. Obviously this result fails if the Lebesgue measure is replaced by an arbitrary measure, for instance a Dirac delta. So it is natural to ask whether this is a rigidity property of the Lebesgue measure, see \cite{preissjfa,ACP1,ACP2,AlbMar}. Namely, does there exist a singular measure for which Rademacher's theorem holds? In \cite{DPR}, the first and the last author showed that the answer to the previous question is negative (in two dimensions, this also follows by combining the main result of~\cite{ACP1,ACP2} with~\cite{AlbMar}). Such a result opened the road to a better understanding of the structure of Lipschitz differentiablility spaces, $RCD(K,N)$ spaces, certain types of Sobolev spaces and also some general measures satisfying linear PDE constraints, see \cite{BURQ, MDPR,GP17,KM18,marino2020short,DPDRG18,Rabasa,Bate2}.

In \cite{Cheeger}, Cheeger generalized Rademacher's theorem to the setting of metric spaces endowed with a doubling measure and a Poincar\'e type inequality. This has inspired a lot of research in the area of analysis on metric measure spaces.  The notion of Lipschitz differentiability space has been later axiomatised  by Keith (\cite{Keith}). In \cite{Bate}, Bate characterized Lipschitz differentiability spaces in terms of the existence of a sufficiently rich family of representations of the underlying measure as an integral of Lipschitz curve fragments.

Carnot groups are connected, simply connected, nilpotent Lie groups whose Lie algebra is stratified.
Referring to the next section for more details,  we only mention here  that they are metric measure spaces whose ambient vector space is $\R^d$, the metric allowing movements only along certain horizontal curves, tangent to a given smooth non-involutive
distribution of planes, the so-called first layer of the Lie algebra stratification. One can define a natural notion of differentiablity for functions between Carnot groups and a seminal  theorem of Pansu,  \cite{Pansu}, proves  the analogue
of Rademacher's theorem in this setting, see also \cite{MaPiSp}. In particular,  Carnot groups endowed with the Haar measure are Lipschitz differentiability spaces.

\smallskip
Also in the context of Canot groups, it is then natural to understand how rigid is the Haar measure for the validity of a Rademacher type theorem. 
In this paper we  prove the analogue in the Sub-Riemannian setting of the result proved in \cite{DPR}, namely that the Haar meaesure is indeed (essentially) the unique measure with respect to which is possible to differentiate Lipschitz function almost everywhere.

\begin{teorema}\label{t:main}
Let $\GG$ be a Carnot group and $\mathbb{H}$ be an homogeneous group and let $\mu$ be a Radon measure on $\GG$. If every Lipschitz function $f:\GG\to\HH$ is Pansu-differentiable $\mu$-almost everywhere, see Definition \ref{def:Pansu_diff_sub}, then $\mu$ is absolutely continuous with respect to the Haar measure on $\GG$.
\end{teorema}

In order to prove  Theorem \ref{t:main} we follow  the same strategy of its euclidean counterpart. First we generalize the work of  Alberti and the second author \cite{AlbMar} by associateing to every Radon measure \(\mu\)
on \(\GG\) a decomposability bundle $V(\mu,\cdot)$ which identifies a set of directions along which a  Rademacher-type theorem, adapted to the measure $\mu$ holds true, see Section \ref{s-decomp}. More precisely we have the following theorem,  where we refer to Definition \ref{def:Pansu_diff_sub} for the notion  of differentiability along a homogeneous subgroup.

\begin{teorema}\label{t:differentiability_along_bundle}
Let $\mu$ be a Radon measure on a Carnot group $\GG$. For every homogeneous group $\HH$, every Lipschiz function $f:\GG\to\HH$ is Pansu differentiable at $\mu$-almost every $x\in\GG$ with respect to the homogeneous subgroup $V(\mu,x)$.
\end{teorema}

Once this bundle is obtained, we exploit the work of  Bate \cite{Bate} to show that for a  measure $\mu$ satisfying the assumptions of Theorem \ref{t:main},  $V(\mu,x)=\GG$ for $\mu$-a.e. $x$, see Proposition \ref{prop:decomposizionei}. Finally we show that this forces \(\mu\) to be absolutely continuous with respect to the Haar measure. This last step is obtained by a PDE type argument that extends the result of \cite{DPR} to the hypoelliptc setting.

We note that, although the general strategy follows the one used to prove the euclidean counterpart of Theorem \ref{t:main}, its adaptation to the Carnot setting requires several  non trivial adjustments. In particular, one of the key step in the proof of the converse of  Rademacher theorem is the link between the fact that the decomposability bundle of a measure has full dimension and the existence of a suitable family of normal currents, proved in \cite[Section 6]{AlbMar}. This is indeed a crucial point in order to rely on the results in \cite{DPR}. The key  geometric property used  to show the existence of this family of currents is  the fact that, given a compact set \(K\subset [0,1]\) and a Lipschitz  fragment \(\gamma: K \to \mathbb{R}^n\) with \(\gamma`(t)\) belonging to a cone \(C\) for a.e. \(t \in K\), this mostly coincides, locally a.e., with a full Lipschitz curve \(\tilde{\gamma}: (a,b)\subset[0,1] \to \mathbb{R}^n\), still satisfying \(\tilde{\gamma}`(t)\in C\), for a.e. $t\in(a,b)$, see \cite[(6.13)]{AlbMar}. This property is in general false for Carnot  groups, \cite{unextendablecurve} and \cite{JS2017} and it requires specific assumptions to be true \cite{ZimmermanC1,Speight2016,LDS,PSZ}. We need thus to rely  to a completely different construction which we believe to be of its own interest, see Section \ref{s4}.

The second key point consists  in the extension of  the theory established in \cite{DPR} to the setting of differential operators defined by H\"ormander type vector fields. Indeed the results in \cite{DPR} strongly rely on the notion of wave cone associated with a differential operator which is, loosely speaking, related to the notion of ellipticity. This notion is too strong in this context and it should be relaxed to the notion of hypoellipticty, which however is less "explicit". Luckily for second order operators (which are the only ones needed in this context), the notions can be characterized algebraically and this allows to adapt the proofs in \cite{DPR} to this setting, see Proposition \ref{propofond}. We conclude by noticing that it is an interesting question to extend the full results of \cite{DPR} to a "hypoellitpic wave cone".

	\subsection*{Acknowledgements.}
	The work of G.D.P is  partially supported by the NSF grant DMS 2055686 and by the Simons Foundation. A.Ma. acknowledges partial support from PRIN 2017TEXA3H\_002 "Gradient flows, Optimal Transport and Metric Measure Structures".   
During the writing of this work	  A.Me.~was supported by the Simons Foundation grant 601941, GD., by the Swiss National Science Foundation
(grant 200021-204501 `\emph{Regularity of sub-Riemannian geodesics and
applications}')
and by the European Research Council (ERC Starting Grant 713998 GeoMeG and by the European Union’s Horizon Europe research and innovation programme under the Marie Sk\l odowska-Curie grant agreement no 101065346.
A. Pi. is partially supported by the Indam-GNAMPA project 2022 ``\textit{Problemi al bordo e applicazioni geometriche}, codice CUP\_E55\-F22\-00\-02\-70\-001\,'' and by MIUR and the University of Trento, Italy.
F.R. acknowledges funding from the European Research Council (ERC) under the European Union's Horizon 2020 research and innovation programme, grant agreement No 757254 (SINGULARITY).

%%%%%%%%%%%%%%%%%%%%%%%%%%%%%%%%%%%%%%%%%%%%%%%%%%%%%%%%%%%%%%%%%%
%%
%%	SECTION 2 
%%
%%%%%%%%%%%%%%%%%%%%%%%%%%%%%%%%%%%%%%%%%%%%%%%%%%%%%%%%%%%%%%%%%%

\section*{List of notations}
We add below a list of frequently used notations, together with the page of their first appearance:
\medskip

\begin{longtable}{c p{0.7\textwidth} p{\textwidth}}

$\lvert\cdot\rvert$\label{euclide} & Euclidean norm, & \pageref{euclide}\\

$d_c$ & \cc metric & \pageref{ccmetrix}\\

$\delta_\lambda$ & intrinsic dilations&\pageref{dilatax}\\

$X_i$ & canonical horizontal vector fields &\pageref{campii}\\

$\mathscr{M}(\R^n,\R^n)$& family of vector-valued Radon measures endowed with the topology of weak* topology &\label{Mspace}\pageref{Mspace}\\

$B(x,r)$& ball of centre $x$ and radius $r$ with respect to the metric $d_c$\label{pallasubx}&\pageref{pallasubx}\\

$U(x,r)$ & ball of centre $x$ and radius $r$ with respect to the Euclidean metric \label{pallaeux}&\pageref{pallaeux}\\

$\mathrm{Gr}(\mathbb{G})$ & Grassmannian of homogeneous subgroups of $\mathbb{G}$&\pageref{def:Grassmannian}\\

$\mathrm{Gr}_\mathfrak{C}(\mathbb{G})$ & Grassmannian of Carnot subgroups of $\mathbb{G}$&\pageref{def:Grassmannian}\\

$V(\mu,\cdot)$& decomposability bundle of the Radon measure & \pageref{Dec-bundle}\\

$N(\mu,\cdot)$& auxiliary decomposability bundle of the Radon measure & \pageref{def:auxdecbundle}\\

$\partial T$ & boundary of current $T$  & \pageref{def:boundarycurrent}\\

$d_Vf(x)$ & differential of a Borel map $f$ along the subgroup $V\in \mathrm{Gr}(\mathbb{G})$&\pageref{def:Pansu_diff_sub}\\ 
\end{longtable}

\section{Notation and preliminaries}
\label{s2}

%%%%%%%%%%%%%%%%%%%%%%%%%%%%%%%%%%%%%%%%%%%%%%%

\subsection{Preliminaries on Carnot groups}
\label{s-prelim_carnot}

In this subsection we briefly introduce some notations on Carnot groups that we will extensively use throughout the paper. For a detailed account on Carnot groups we refer to \cite{LD17}.

A Carnot group $\mathbb{G}$ of step $\mathfrak s$  is a connected and simply connected Lie group whose Lie algebra $\mathfrak g$ admits a stratification $\mathfrak g=V_1\, \oplus \, V_2 \, \oplus \dots \oplus \, V_{\mathfrak{s}}$. We say that $V_1\, \oplus \, V_2 \, \oplus \dots \oplus \, V_\mathfrak{s}$ is a {\em stratification} of $\mathfrak g$ if $\mathfrak g = V_1\, \oplus \, V_2 \, \oplus \dots \oplus \, V_\mathfrak{s}$,\label{straix}
$$
[V_1,V_i]=V_{i+1}, \quad \text{for any $i=1,\dots,\mathfrak{s}-1$}, \quad \text{and} \quad [V_1,V_\mathfrak{s}]=\{0\},
$$ 
where $[A,B]:=\mathrm{span}\{[a,b]:a\in A,b\in B\}$. We call $V_1$ the \emph{horizontal layer} of $\mathbb G$. We denote by $n$ the topological dimension of $\mathfrak g$ and by $n_j$ the dimension of $V_j$ for every $j=1,\dots,\mathfrak{s}$. 
Furthermore, we define $\pi_i:\mathfrak{g}\to V_i$ to be the projection maps on the $i$-th strata. 
We will often shorten the notation to $v_i:=\pi_iv$.

The exponential map $\exp :\mathfrak g \to \mathbb{G}$ is a global diffeomorphism from $\mathfrak g$ to $\mathbb{G}$.
Hence, if we choose a basis $\{X_1,\dots , X_n\}$ of $\mathfrak g$,  any $p\in \mathbb{G}$ can be written in a unique way as $p=\exp (p_1X_1+\dots +p_nX_n)$. This means that we can identify $p\in \mathbb{G}$ with the $n$-tuple $(p_1,\dots , p_n)\in \R^n$ and the group $\mathbb{G}$ itself with $\R^n$ endowed with $*$ the group operation determined by the Baker-Campbell-Hausdorff formula. \emph{From now on, we will always assume that $\mathbb{G}=(\R^n,*)$ and, as a consequence, that the exponential map $\exp$ acts as the identity.}

The stratificaton of $\mathfrak{g}$ carries with it a natural family of dilations $\delta_\lambda :\mathfrak{g}\to \mathfrak{g}$, that are Lie algebra automorphisms of $\mathfrak{g}$ and are defined by\label{dilatax}
\begin{equation}
     \delta_\lambda (v_1,\dots , v_\mathfrak{s}):=\begin{cases}
     (\lambda v_1,\lambda^2 v_2,\dots , \lambda^\mathfrak{s} v_\mathfrak{s}), \quad &\text{for any $\lambda>0$},\\
     (-\lvert \lambda\rvert v_1,-\lvert\lambda\rvert^2 v_2,\dots , -\lvert\lambda\rvert^\mathfrak{s} v_\mathfrak{s}), \quad &\text{for any $\lambda\leq 0$}.
     \end{cases}
     \nonumber
\end{equation}
where $v_i\in V_i$. The stratification of the Lie algebra $\mathfrak{g}$  naturally induces a gradation on each of its homogeneous Lie sub-algebras $\mathfrak{h}$, i.e., a sub-algebra that is $\delta_{\lambda}$-invariant for any $\lambda>0$, that is
\begin{equation}
    \mathfrak{h}=V_1\cap \mathfrak{h}\oplus\ldots\oplus V_\mathfrak{s}\cap \mathfrak{h}.
    \label{eq:intr1}
\end{equation}
We say that $\mathfrak h=W_1\oplus\dots\oplus W_{\mathfrak{s}}$ is a {\em gradation} of $\mathfrak h$ if $[W_i,W_j]\subseteq W_{i+j}$ for every $1\leq i,j\leq \mathfrak{s}$, where we mean that $W_\ell:=\{0\}$ for every $\ell > \mathfrak{s}$.
Since the exponential map acts as the identity, the Lie algebra automorphisms $\delta_\lambda$ are also group automorphisms of $\mathbb{G}$.

\begin{definizione}\label{homsub}
A subgroup $\mathbb V$ of $\mathbb{G}$ is said to be \emph{homogeneous} if it is a Lie subgroup of $\mathbb{G}$ that is invariant under the dilations $\delta_\lambda$.
A homogeneous subgroup $\mathbb V\subset \mathbb{G}$ is called {\em horizontal  subgroup} if $\mathbb V\subseteq \exp(V_1)=V_1$.
%A homogeneous subgroup $\mathbb W \subset  \mathbb{G}$ is called a {\em Carnot subgroup} if there exists a sub-algebra $\mathfrak{h}$ of $\mathfrak{g}$ such that \eqref{eq:intr1} defines a stratification of $\mathfrak{h}$ and  $\mathbb{W}=\exp(\mathfrak h)$. %and  it such that the first %layer $V_1\cap\mathfrak h$ of %the grading of $\mathfrak h$ %inherited from the %stratification of $\mathfrak g$ %is the first layer of a %stratification of $\mathfrak h$
\end{definizione}
 The following general fact will play a crucial role later on.

\begin{proposizione}\label{prop:hom1}
Suppose $H$ is a closed subgroup of $\mathbb{G}\cong(\R^n,*)$. Then $H$ can be identified with a vector subspace of $\R^n$. In particular, homogeneous closed subgroups of $\mathbb{G}$ are in bijective correspondence through $\exp$ with the Lie sub-algebras of $\mathfrak{g}$ that are invariant under the dilations $\delta_\lambda$. 
\end{proposizione}

\begin{proof}
Thanks to \cite[Theorem 3.6]{Cartan} we know that $H$ is a Lie subgroup of $\mathbb{G}$. In particular its Lie algebra $\mathfrak{h}$ is a Lie subalgebra of $\mathfrak{g}$. Thanks to the definition of the operation $*$, the exponential map $\mathrm{exp}$ acts as the identity and thus $H$, %by the celebrated Lie's theorem 
can be identified with its Lie algebra in $\mathfrak{g}\cong \R^n$ and thus it can be viewed as a vector subspace of $\R^n$.
\end{proof}

\medskip

\emph{From now on, since as already remarked $\mathrm{exp}$ acts as the identity thanks to the choice of $*$, we will always identify with abuse of notation the elements of $\GG$, with their preimage under $\mathrm{exp}$ in $\mathfrak{g}$.} \emph{In addition, from now on and if not otherwise stated, $\mathbb G$ will be a fixed Carnot group}.

\begin{definizione}[Homogeneous left-invariant distance and norm]
A metric $d:\mathbb{G}\times \mathbb{G}\to \R$ is said to be {\em homogeneous} and {\em left-invariant} if for any $x,y\in \mathbb{G}$ we have, respectively
\begin{itemize}
    \item[(i)] $d(\delta_\lambda x,\delta_\lambda y)=\lambda d(x,y)$ for any $\lambda>0$,
    \item[(ii)] $d(\tau_z x,\tau_z y)=d(x,y)$ for any $z\in \mathbb{G}$.
\end{itemize}
Given a homogeneous left-invariant distance, its associated norm is defined by $\|g\|_{d}:=d(g,0)$, for every $g\in\mathbb G$, where $0$ is the identity element of $\mathbb G$\footnote{In the following we will leave the dependence of the norm on the metric always implicit.}. 
Given a homogeneous left-invariant distance $d$ on $\mathbb G$, for every $x\in \mathbb G$ and every $E\subseteq \mathbb G$ we define $\mathrm{dist}(x,E):=\inf\{d(x,y):y\in E\}$.
\end{definizione}

The specific choice of the metric is not relevant for our purposes thanks to the following result, we refer to \cite[Proposition 5.1.4]{equivmetr} for a proof.

\begin{proposizione}\label{equivbilip}
Assume $d_1,d_2$ are two homogeneous left-invariant metrics on $\mathbb{G}$. Then there exists a constant $C>0$ depending on $d_1$ and $d_2$ such that $C^{-1}d_1(x,y)\leq d_2(x,y)\leq Cd_1(x,y)$ for any $x,y\in\mathbb{G}$.
\end{proposizione}

\begin{lemma}\label{lem:EstimateOnConjugate}
For any left-invariant and homogeneous distance and for any $k>0$ there exists a constant $\newC\label{c:1}:=\oldC{c:1}(k,\mathbb G,d)>1$ such that if $x,y\in B(0,k)$, then
$$
\|y^{-1}* x* y\| \leq \oldC{c:1}\|x\|^{1/\mathfrak{s}}.
$$
\end{lemma}

\begin{proof}
For a proof we refer to \cite[Lemma 3.6]{MR3123745}.
\end{proof}

\begin{osservazione}\label{rk.norm}
Let $d$ be a left-invariant homogeneous distance on $\GG$. It is well known, see for instance \cite[Proposition 5.15.1]{equivmetr}, that for any compact subset $K$ of $\R^n$ there is a constant $C(K,d)>1$ such that:
$$C(K,d)^{-1}\lvert x-y\rvert\leq d(x,y)\leq C(K,d)\lvert x-y\rvert^{1/\mathfrak{s}}\qquad\text{for any }x,y\in K.$$
More precisely the constant $C$ introduced above depends only on $\mathrm{dist}(0,K)+\mathrm{diam}(K)$ and $d$. 
\end{osservazione}

\medskip

For any Lie algebra $\mathfrak{h}$ with gradation $\mathfrak h= W_1\oplus\ldots\oplus W_{\mathfrak{s}}$, we define its \emph{homogeneous dimension} as
$$\text{dim}_{\mathrm{hom}}(\mathfrak{h}):=\sum_{i=1}^{\mathfrak{s}} i\cdot\text{dim}(W_i).$$
Thanks to \eqref{eq:intr1} we infer that, if $\mathfrak{h}$ is a homogeneous Lie sub-algebra of $\mathfrak{g}$, then $$\text{dim}_{\mathrm{hom}}(\mathfrak{h}):=\sum_{i=1}^{\mathfrak{s}} i\cdot\text{dim}(\mathfrak{h}\cap V_i).$$ It is well-known that the Hausdorff dimension, 
% for a definition of Hausdorff dimension see for instance \cite[Definition 4.8]{Mattila1995GeometrySpaces},
of a graded Lie group $\mathbb G$ with respect to a left-invariant homogeneous distance coincides with the homogeneous dimension of its Lie algebra. For a reference for the latter statement, see \cite[Theorem 4.4]{LDNG19}.

\begin{definizione}[Carnot subgroups]\label{definitiongenerazione}
Let $\Lambda\subset [0,\infty)$. Given a collection $\mathscr{F}=\{v_\lambda\in\GG:\lambda\in \Lambda\}$ of elements of $\GG$ we define the homogeneous subgroup $\mathfrak{S}(\mathscr{F})$ of $\mathbb{G}$ generated by $\mathscr{F}$ as:
$$\mathfrak{S}(\mathscr{F}):=\mathrm{cl}\Big(\big\{\delta_{\rho_1}(v_{\lambda_1})*\cdots*\delta_{\rho_N}(v_{\lambda_N}):N\in\N,\,\rho_i\in\R\text{ and }v_i\in \mathscr{F}\text{ for any $i_j\in\Lambda$ and $j\in \{1,\ldots,N\}$}\big\}\Big).$$
We say that a subgroup $V$ of $\GG$ is a \emph{Carnot subgroup} if $V=\mathfrak{S}(V\cap V_1)$.
\end{definizione}

% \begin{lemma}
% A subset $V$ of a Carnot group $\mathbb{G}$ with Lie algebra $\mathfrak{g}$ is a Carnot subgroups if and only if there exists a sub-algebra $\mathfrak{h}$ of $\mathfrak{g}$ such that \eqref{eq:intr1} defines a stratification of $\mathfrak{h}$ and  $V=\exp(\mathfrak h)$
% \end{lemma}
% \begin{proof}

% \end{proof}

\begin{definizione}[Intrinsic Grassmannian on Carnot groups]\label{def:Grassmannian}
Let $\mathcal{Q}:=\text{dim}_{hom}(\mathfrak{g})$\label{homog} and let  $1\leq h\leq Q$. We define $\Gr(h)$ and $\Gr_{\mathfrak{C}}(h)$ to be the family of all homogeneous subgroups $ W$ of $\mathbb{G}$ with Hausdorff dimension $h$ and the family of all Carnot subgroups $ W$ of $\mathbb{G}$ with Hausdorff dimension $h$ respectively.
Finally, we  denote by $\Gr(\mathbb{G})$ and $\Gr_\mathfrak{C} (\GG)$ the sets 
\[
\Gr(\mathbb{G})=\bigcup_{h=1}^Q \Gr(h),\quad \Gr_{\mathfrak{C}}(\mathbb{G})=\bigcup_{h=1}^Q \Gr_{\mathfrak{C}}(h).
\]
Since it will be occasionally used, it will be convenient to denote by $\mathrm{Gr}_{eu}(\GG)$ the Euclidean Grassmannian of the underlying space of $\GG$ endowed with the topology induced by the Hausdorff distance induced by the Euclidean distance. Such topology and that induced by the \cc Hausdorff distance are easily seen to be the same.
\end{definizione}

\begin{proposizione}\label{generato}
Let $V\in \Gr_\mathfrak{C}(\GG)$ and assume $v_1,\ldots,v_N\in V\cap V_1$ are such that $V\cap V_1$ coincides with the linear span of $\{v_1,\ldots,v_N\}$ when seen as vectors of $\R^n$. Then
$\mathfrak{S}(\{v_1,\ldots,v_N\})=V$.
\end{proposizione}

\begin{proof}
The inclusion $\mathfrak{S}(\{v_1,\ldots,v_N\})\subseteq V$ is obvious and thus we just need to prove the converse. Since $\mathfrak{S}(\{v_1,\ldots,v_N\})$ is a closed homogeneous subgroup of $\mathbb{G}$, it is also a vector subspace of $\GG$, see Proposition \ref{prop:hom1}. Therefore, we have $\mathrm{span}\{v_1,\ldots,v_N\}=V\cap V_1\subseteq \mathfrak{S}(\{v_1,\ldots,v_N\})$ and thus: $$V=\mathfrak{S}(V_1\cap V)\subseteq \mathfrak{S}(\{v_1,\ldots,v_N\}),$$
where the first identity follows from the fact that $V$ is a Carnot subgroup of $\GG$.
\end{proof}

For any $p\in \mathbb{G}$, we define the {\em left translation} $\tau _p:\mathbb{G} \to \mathbb{G}$ as\label{tran}
\begin{equation*}
q \mapsto \tau _p q := p* q.
\end{equation*}
As already remarked above, we can suppose without loss of generality that the group operation $*$ is determined by the Campbell-Hausdorff formula. It is well known that $*$ has a polynomial expression in the coordinates, see \cite[Proposition 2.1]{step2}, and more precisely
\begin{equation*}
p*q= p+q+\mathscr{Q}(p,q), \quad \mbox{for all }\, p,q \in  \R^n,
\end{equation*} 
where $\mathscr{Q}=(\mathscr{Q}_1,\dots , \mathscr{Q}_s):\R^n\times \R^n \to V_1\oplus\ldots\oplus V_\mathfrak{s}$, and the $\mathscr{Q}_i$s are vector valued polynomials.
For any $i=1,\ldots \mathfrak{s}$ and any $p,q\in \mathbb{G}$ we have
\begin{itemize}
    \item[(i)]$\mathscr{Q}_i(\delta_\lambda p,\delta_\lambda q)=\lambda^i\mathscr{Q}_i(p,q)$,
    \item[(ii)] $\mathscr{Q}_i(p,q)=-\mathscr{Q}_i(-q,-p)$,
    \item[(iii)] $\mathscr{Q}_1=0$ and $\mathscr{Q}_i(p,q)=\mathscr{Q}_i(p_1,\ldots,p_{i-1},q_1,\ldots,q_{i-1})$.
\end{itemize}
Thus, we can represent the operation $*$ as
\begin{equation}\label{opgr}
p * q= (p_1+q_1,p_2+q_2+\mathscr{Q}_2(p_1,q_1),\dots ,p_s +q_s+\mathscr{Q}_s (p_1,\dots , p_{s-1} ,q_1,\dots ,q_{s-1})). 
\end{equation}

% \begin{definizione}[Hausdorff Measures]\label{def:HausdorffMEasure}
% Throughout the paper we define the $h$-dimensional {\em Hausdorff measure}\label{hausmeas} relative to $d$ as
% $$
% \Haus^h(A):=\sup_{\delta>0}\inf \left\{\sum_{j=1}^{\infty} 2^{-h}\mathrm{diam}(E_j)^h:A \subseteq \bigcup_{j=1}^{\infty} E_j,\, \mathrm{diam}(E)\leq \delta\right\},
% $$
% where $\mathrm{diam}(E) $ is the diameter of $E$ computed with respect to the distance $d$. 
% \end{definizione}

\begin{definizione}
A Borel set $E\subset \mathbb{G}$ is called $1$-rectifiable if there exists a countable family of Lipschitz maps $\gamma_i:K_i\to\GG$, where $K_i$ are compact subsets of $\R$ such that $\mathcal{H}^1(E\setminus \bigcup_{i=1}^{\infty}\gamma_i(K_i))=0$.
A Radon measure $\phi$ on $\mathbb{G}$ is said to be $1$-rectifiable if there exists a $1$-dimensional rectifiable set $E$ such that $\phi\ll f\mathcal{H}^1\trace E$.
\end{definizione}

\subsection{Lipschitz curves and the horizontal distribution and the \cc distance}

A fundamental role in this work will be played by Lipschitz curves and fragments in $\GG$. In this subsection we introduce the \emph{horizontal distribution} of $n_1$-dimensional planes associated to $\R^n$ and we define the \cc di\-stan\-ce. 

\begin{definizione}\label{campii}
Let $\{e_1,\ldots,e_{n_1}\}$ be an orthonormal basis of $V_1$. For every $i=1,\ldots,n_1$ we say that  the left-invariant vector field tangent to $e_i$ at the origin
\begin{equation}
    X_i(x):=\lim_{t\to 0}\frac{x*\delta_t(e_i)-x}{t},
    \label{eq:campi}
\end{equation}
is the $i$-th \emph{horizontal vector field}.
Furthermore, for any $i=1,\ldots,n_1$ we can write the vector field $X_i$ as:
$$X_i(x):=\sum_{j=1}^n \mathfrak{c}_j^i(x)\partial_j,$$
where $\mathfrak{c}_j^i(x)$ are smooth functions since the $\mathscr{Q}_i$s are polynomial functions. We further let  
\begin{equation}
    H\GG(x):=\mathrm{span}(X_1,\ldots,X_{n_1})
    \label{horizontaldistribution}
\end{equation}
The distribution $ H\GG(x)$ of $n_1$-dimensional planes is usually said to be the \emph{horizontal distribution} associated to the group $\GG$. In the following it will be useful to write the coefficients $\mathfrak{c}_j^i$ in the form of the matrix
$$\mathscr{C}(x):=\begin{pmatrix}
\mathfrak{c}_1^1(x) &\dots& \mathfrak{c}^{n_1}_1(x)\\
\vdots&\ddots&\vdots\\
\mathfrak{c}^1_n(x)&\dots&\mathfrak{c}^{n_1}_n(x)
\end{pmatrix}.	$$
\end{definizione}

\begin{osservazione}[Expression for the $\mathfrak{c}_i$'s]\label{rk:expc}
Thanks to the definition of the vector fields $X_i$ in \eqref{eq:campi} and to the coordinate-wise expression of the operation $*$ given in \eqref{opgr}, we infer that
$$X_i(x)=e_i+\frac{\partial\mathscr{Q}}{\partial q_i}(x,0).$$
This shows in particular that the matrix $\mathscr{C}(x)$ can be represented as
$$\mathscr{C}(x):=\begin{pmatrix}\mathrm{id}_{n_1}\\
\hline
\partial_q \mathscr{Q}(x,0)\end{pmatrix}.$$
\end{osservazione}

\begin{definizione}\label{canonicalcoo}
Let $B$ be a bounded Borel subset %{\color{red} FORSE SERVE BOREL (OPPURE BOREL LIMITATO)} 
of the real line. Given a curve $\gamma:B\to \GG$ and a Lebesgue density point $t\in B$, we denote
$$\gamma^\prime(t):=\lim_{\substack{r\to 0 \\ t+r\in B}}\frac{\gamma(t+r)-\gamma(t)}{r},\qquad \text{whenever the right-hand side exists.}$$

Furthermore, given $a<b$ we say that an absolutely continuous curve $\gamma:[a,b]\to \GG$ is \emph{horizontal} if there exists a measurable function $h: [a,b]\to V_1$ such that:
\begin{itemize}
    \item[(i)] $\gamma^\prime(t)=\mathscr{C}(\gamma(t))[h(t)]$ for $\Leb^1$-almost every $t\in [a,b]$,
    \item[(ii)] $\lvert h\rvert\in L^\infty([a,b])$.
\end{itemize}
Following the notation of \cite{tesimonti} we shall refer to $h$ as the \emph{canonical coordinates of }$\gamma$ and if $\lVert h\rVert_\infty\leq 1$ we will say that $\gamma$ is a \emph{sub-unit} path.
Finally, we define the Carnot-Carath\'eodory distance $d_c$ on $\GG$ as:\label{ccmetrix}
\begin{equation}
d_c(x,y):=\inf\{T\geq 0:\text{ there is a sub-unit path }\gamma:[0,T]\to\R^n\text{ such that }\gamma(0)=x\text{ and }\gamma(T)=y\}.\nonumber
\end{equation}
It is well known that $d_c(\cdot,\cdot)$ is a left-invariant homogeneous metric on $\GG$. Finally throughout the paper we will denote by $\lVert \cdot\rVert$ the homogeneous function $x\mapsto d_c(x,0)$ and \emph{from now on and if not otherwise specified, $\GG$ will always be endowed with the distance $d_c$.}
\end{definizione}

\begin{proposizione}\label{cc.geodesic}
The distance $d_c$ is a geodesic distance, i.e. for any $x,y\in\mathbb{G}$ there exists a sub-unit path $\gamma:[0,T]\to \mathbb{G}$ such that $\gamma(0)=x$, $\gamma(T)=y$ and $d_c(x,y)=T$.
\end{proposizione}

\begin{proof}
This follows immediately from Proposition \ref{equivbilip} and \cite[Lemma 3.12]{MFSSC}.
\end{proof}

The following lemma allows us to characterise those Euclidean Lipschitz curves that are also Lipschitz curves when $\R^n$ is endowed with the \cc distance $d_c$ introduced above.

\begin{lemma}\label{lemma.monti1}
Let $B$ be a bounded Borel subset of the real line. If a curve $\gamma:B\to\GG$ is $L$-Lipschitz with respect to the distance $d_c$ on $\GG$, then $\gamma$ is an Euclidean absolutely continuous curve such that: $$\gamma^\prime(t)=\mathscr{C}(\gamma(t))[h(t)]\text{ for }\Leb^1\text{-almost every }t\in B,$$
for some $h\in L^\infty(B,V_1)$ with $\lVert h\rVert_\infty\leq L$.
\end{lemma}

\begin{osservazione}\label{osservazionecoocanoniche}
With abuse of language, for any Lipschitz curve $\gamma:B\to \GG$ we will refer to the function $h$ yielded by Lemma \ref{lemma.monti1} as \emph{the canonical coordinates of $\gamma$}. For the original definition of canonical coordinates, see Definition \ref{canonicalcoo}. 
\end{osservazione}

\begin{proof}
The proof of the lemma follows from \cite[Lemma 1.3.3]{tesimonti} together with an elementary localisation argument. 
\end{proof}

\begin{proposizione}\label{identificazionehaus1}
Let $B$ be a Borel subset of the real line and $\gamma:B\to\mathbb{G}$ be a Lipschitz map.
Then the measures $\Haus^1\trace\im(\gamma)$ and $\Haus^1_{eu}\trace\im(\gamma)$ are mutually absolutely continuous.
	\end{proposizione}

\begin{proof}
Since $\lvert x-y\rvert\leq  d_c(x,y)$ for every $x,y\in\mathbb{G}$ the definition of Hausdorff measure immediately implies that $ \Haus^1_{\mathrm{eu}}\leq\Haus^1$. For the converse, let us note that for any Lipschitz curve $\gamma:B\to \GG$ the area formula \cite[Theorem 4.4]{magnaniarea} implies that for any Borel set $A\subseteq \GG$ we have
$$\Haus^1\llcorner \im(\gamma)(A)=\int_{B\cap A} \lvert D\gamma(t)\rvert dt\leq\frac{1}{\max_{x\in \im\gamma}\lVert \mathscr{C}(x)\rVert}\int_{B\cap A}\lvert \mathscr{C}(\gamma(t))[D\gamma(t)]\rvert dt=\frac{\Haus_{\mathrm{eu}}^1\trace \im(\gamma)(A)}{\max_{x\in \im\gamma}\lVert \mathscr{C}(x)\rVert},$$
where the last identity follows from Lemma \ref{osservazionecoocanoniche}. This concludes the proof. 
\end{proof}

\begin{definizione}[Pansu differentiability]\label{def:Pansu_diff_sub}

We say that a map $f:\GG\to\HH$ is 
Pansu differentiable at the point $x\in\GG$ with respect to
a homogeneous subgroup $V$ of $\GG$ if 
there exists a homogeneous homomorphism $L:V\to\HH$ such 
that
\[
d_{\HH}\big(f(x)^{-1}*f(xh), L(h)\big)=o(\lVert h\rVert_{\GG})
\quad\hbox{for all $h\in V$.}
\]
When it exists, $L$ is called the (Pansu) derivative of $f$ at
$x$ with respect to $V$ and denoted by $\dV f(x)$.
If $V=\GG$ then $\dV f(x)$ is the usual (Pansu) 
derivative, and is simply denoted by $df(x)$.
\end{definizione}

The following lemma can be proved with an immediate adaptation of the argument used to prove \cite[Lemma 2.1.4]{tesimonti} that allows us to characterise the Pansu derivative of Lipschitz curves.

\begin{lemma}\label{lemma.monti2}
Let $B$ a bounded Borel subset of the real line and assume $\gamma:B\to \GG$ is a Lipschitz curve. If $h\in L^\infty(B,V_1)$ is the vector of canonical coordinates of $\gamma$, then for $\Leb^1$-almost every $t\in B$ we have:
    $$D\gamma(t):=\lim_{\substack{s\to 0\\t+s\in B}}\delta_{1/s}(\gamma(t)^{-1}*\gamma(t+s))=(h_1(t),\ldots,h_{n_1}(t),0,\ldots,0).$$
    In particular $D\gamma(t)$ exists for $\Leb^1$-almost every $t\in B$.
\end{lemma}

\begin{proof}
The proof of this lemma follows from \cite[Lemma 2.1.4]{tesimonti} together with an elementary localization argument.
\end{proof}

% {\color{RoyalBlue}
% \begin{proof}
% Let us first prove the result in the case $B$ is a compact set. Arguing as in the proof of Lemma \ref{lemma.monti1} we can extend the curve $\gamma$ to a curve $\tilde{\gamma}$ defined on the segment that is the convex hull of $B$ and thanks to \cite[Lemma 2.1.4]{tesimonti}, we infer that $D\tilde{\gamma}(t)=(h(t),0,\ldots,0)$ for $\Leb^1$-almost every $t$ in the convex hull of $B$. This implies in particular that:
% $$\lim_{\substack{s\to 0\\ t+s\in B}}\delta_{1/s}(\gamma(t)^{-1}*\gamma(t+s))=\lim_{\substack{s\to 0\\ t+s\in B}}\delta_{1/s}(\tilde{\gamma}(t)^{-1}*\tilde{\gamma}(t+s))=(\tilde{h}_1(t),\ldots,\tilde{h}_{n_1}(t),0,\ldots,0),$$
% for $\Leb^1$-almost any $t\in B$, where $\tilde{h}$ is the vector of canonical coordinates for $\tilde{\gamma}$. However, since $\gamma^\prime$ and $\tilde{\gamma}^\prime$ coincide almost everywhere on $B$ and $\mathscr{C}(x)$ is an injective linear map for any $x$, one also has that $\tilde{h}=h$ almost everywhere on $B$. This concludes the proof in the case $B$ is compact. 

% The same approximation argument we employed in Lemma \ref{lemma.monti1} in order to pass from compacts to Borel sets concludes the proof.
% \end{proof}
% }
\begin{definizione}[$C$-curves]
\label{C-curves}
Let $e\in V_1$ be a unit vector and $\sigma\in (0,1)$. We denote by $C(e,\sigma)$ the one-sided, closed, convex cone with axis $e$ and opening $\sigma$ in $V_1$, namely\label{s-cones}
$$C(e,\sigma):=\{x\in V_1:\langle x,e\rangle\geq(1-\sigma^2)|x|\}.$$
Let $B$ be a bounded Borel subset of the real line. A Lipschitz curve $\gamma:B\to \GG$, is said to be a $C(e,\sigma)$-curve if: 
   % \item[(\hypertarget{dir.cono}{i})] $D\gamma(t)\in C$ for $\Leb^1$-almost every $t\in B$,
   \begin{equation}
       \text{ $\pi_1(\gamma(s))-\pi_1(\gamma(t))\in C(e,\sigma)\setminus\{0\}$ for any $t,s\in B$ with $t<s$.}
       \label{conecondition}
   \end{equation}
If the domain of a $C(e,\sigma)$-curve $\gamma$ is a compact interval, we will say that $\gamma$ is a \emph{full} $C(e,\sigma)$-curve.
\end{definizione}

\begin{osservazione}\label{ossnice}
Thanks to Lemmas \ref{lemma.monti1} and \ref{lemma.monti2} if $\gamma:B\to \mathbb{G}$ is a $C$-curve, then for $\Leb^1$-almost every $t\in B$ we have:
$$(\pi_1\circ \gamma)^\prime(t)=\pi_1(\gamma^\prime(t))=\pi_1(\mathscr{C}(x)[h(t)])=D\gamma(t),$$
where $h$ is map of canonical coordinates associated to $\gamma$, see Remark \ref{osservazionecoocanoniche}.
\end{osservazione}

\begin{osservazione}
Note that any $C(e,\sigma)$-curve is injective. Indeed, if we suppose by contradiction that $\gamma(s)=\gamma(t)$ for some $t<s$ we would infer that $\pi_1(\gamma(s))=\pi_1(\gamma(t))$. This however is not possible thanks to \eqref{conecondition}.
\end{osservazione}

\begin{lemma}\label{campovgamma}
Let $\gamma:K\to \mathbb{G}$ be a Lipschitz curve. Then, there exists an $\Haus^1$-measurable map $\mathfrak{v}_\gamma:\im(\gamma)\to \mathbb{G}$ such that $\mathfrak{v}_\gamma(x)\in \{D\gamma(t):t\in\gamma^{-1}(x)\}$ and $\mathfrak{v}_\gamma(x)\neq 0$ for $\Haus^1\llcorner \im(\gamma)$-almost every $x\in\mathbb{G}$.
\end{lemma}

\begin{proof}
Let us denote by $N_1$ the set of those $x\in \im(\gamma)$ such that $\mathrm{Card}(\gamma^{-1}(x))=\infty$ and define
$$\mathscr{N}_2:=\{t\in K:D\gamma(t)\text{ does not exists}\}\qquad\text{and}\qquad\mathscr{N}_3:=\{t\in K\setminus \mathscr{N}_2:D\gamma(t)=0\}.$$
Therefore, defined $N:=N_1\cup \gamma(\mathscr{N}_2\cup \mathscr{N}_3)$ it is immediate to see that Pansu's differentiability theorem and the area formula in  \cite[Theorem 4.4]{magnaniarea}, imply that $\Haus^{1}(N)=0$.
Denote by $\mathscr{T}:\im(\gamma)\setminus N\to \mathbb{G}$ the multimap which assigns to each $x\in\im(\gamma)$ the set $\{D\gamma(t):t\in \gamma^{-1}(x)\}$. 
Note that $\mathscr{T}$ the multimap is closed-valued as it takes values in the family of non-empty finite sets. Finally, since the map $t\mapsto D\gamma(t)$ is Borel and for any Borel set $B\subseteq \GG$ we have $\{x:\mathscr{T}(x)\cap B\neq \emptyset\}=\gamma(\{t:D\gamma(t)\in B\})$, we see that the Lipschitzianity\footnote{Lipschitz curves send sets of $\Leb^1$-null sets to $\Haus^1$-null sets.} of $\gamma$ implies that $\{x:\mathscr{T}(x)\cap B\neq \emptyset\}$ is $\mathcal{H}^1$-measurable for any Borel set $B$. This in particular proves that $\mathscr{T}$ is a measurable multimap and hence \cite[Theorem 5.2.1]{Sriva} implies the existence of the vector field $\mathfrak{v}_\gamma$ by setting $\mathfrak{v}_\gamma=0$ on $\mathbb{G}\setminus (\im(\gamma)\setminus N)$.
\end{proof}

\begin{definizione}[Tangents to curves]\label{tg:curvesoso}
Let $B\subset\mathbb{R}$ be a Borel set and $\gamma:B\to \GG$ a Lipschitz map. Then, for any $x\in \gamma(K)$ where $\mathfrak{v}_\gamma(x)\neq 0$, we define $\Tan(\gamma(K),x)$ as the $1$-parameter subgroup generated by $\mathfrak{v}_\gamma(x)$ and as $\{0\}$ otherwise.
Note that thanks to  Lemma \ref{lemma.monti2}  we have that $\Tan(\gamma(K),x)$ is a $1$-dimensional homogeneous subgroup for $\Haus^1\llcorner \gamma(K)$-almost every $x\in\GG$.
Furthermore, since $\gamma$ is in particular an Euclidean Lipschitz curve, we denote by $\Tan_{eu}(\gamma,x)$ the usual Euclidean tangent to $\gamma$. Note that by Lemmas \ref{lemma.monti1}, \ref{lemma.monti2} and Proposition \ref{identificazionehaus1} we have:
$$\Tan_{eu}(\gamma,x)=\mathscr{C}(x)[\Tan(\gamma,x)]\qquad \text{for $\Haus^1$-almost every $x\in\GG$}.$$
\end{definizione}

\subsection{Euclidean and horizontal currents.}

We recall here the basic notions and terminology from
the theory of Euclidean currents.
A \emph{$k$-dimensional current} (or $k$-current) 
in $\R^n$ is a continuous linear functional on the space of smooth and 
compactly supported differential $k$-forms on $\R^n$, endowed with the topology of test functions. 

The boundary of a $k$-current $\mathbf{T}$ is the $(k-1)$-current 
$\bd \mathbf{T}$ defined by $\scal{\bd \mathbf{T}}{\omega} := \scal{\mathbf{T}}{d\omega}$\label{def:boundarycurrent}
for every smooth and compactly supported 
$(k-1)$-form $\omega$ on $\R^n$, and where $d\omega$ denotes the exterior derivative of $\omega$.
The \emph{mass} of $\mathbf{T}$, denoted by 
$\Mass(\mathbf{T})$, is the supremum of $\scal{\mathbf{T}}{\omega}$ over
all forms $\omega$ such that $|\omega|\le 1$
everywhere. 
A current $\mathbf{T}$ is called \emph{normal} if both $\mathbf{T}$ 
and $\bd \mathbf{T}$ have finite mass.

\medskip

By Riesz theorem a current $\mathbf{T}$ with finite mass
can be represented as a finite measure with 
values in the space $\largewedge_k(\R^n)$
of $k$-vectors in $\R^n$,
and therefore it can be written in the form
$\mathbf{T}=\tau\mu$ where $\mu$ is a finite positive measure
and $\tau$ is a $k$-vector field such that 
$\int |\tau|d\mu < +\infty$.
In particular the action of $\mathbf{T}$ on a form
$\omega$ is given by
\[
\scal{\mathbf{T}}{\omega} 
= \int_{\R^n} \scal{\tau(x)}{\omega(x)} \, d\mu(x)
\, ,
\]
and the mass $\Mass(\mathbf{T})$ is the total mass of $\mathbf{T}$ as a
measure, that is, $\Mass(\mathbf{T})=\int |\tau| d\mu$. Note that $0$-dimensional currents with locally finite mass are signed Radon measures and the mass coincides with the total variation.

In the following, whenever we write a current $\mathbf{T}$ as $\mathbf{T}=\tau\mu$ 
we tacitly assume that $\tau(x)\ne 0$ for $\mu$-almost every ~$x$;
in this case we say that $\mu$ is a measure
\emph{associated} to the current $\mathbf{T}$.%

Moreover, if $\mathbf{T}$ is a $k$-current with finite mass 
and $\mu$ is an arbitrary measure, we can
write $\mathbf{T}$ as $\mathbf{T}=\tau\mu+\boldsymbol\nu$ where $\tau$
is a $k$-vector field in $L^1(\mu)$, 
called the Radon-Nikod\'ym density of $\mathbf{T}$ 
\wrt $\mu$, and $\boldsymbol\nu$ is a measure 
with values in $k$-vectors which is singular 
with respect to $\mu$.

\smallskip

Given a Carnot group $\GG$, in the previous subsection we have already remarked how we can identify $\GG$ with $\R^n$, the underlying vector subspace of its Lie algebra, endowed with the operation given by the Baker-Campbell-Hausdorff formula. The $1$-dimensional currents of finite mass in $\R^n$ that are of particular importance for this paper and for the geometry of $\GG$ are those that are \emph{tangent to the horizontal distribution of $G$} or simply horizontal. 

\begin{definizione}[Horizontal $1$-dimensional currents of finite mass]\label{defin:hor:curr}
Let $\GG=(\R^n,*)$.
A $1$-dimensional current of finite mass $\mathbf{T}=\tau\mu$ on $\R^n$ is said to be $\GG$-\emph{horizontal}, or simply horizontal, if for $\mu$-almost every $x\in\R^n$ we have $\tau(x)\in H\GG(x)$.
\end{definizione}

The following definition is a central concept throughout the paper, which is the integration of a family of measures.

\begin{definizione}[Integration of measures]
\label{s-measint}
Let $(I,dt)$ be a ($\sigma$-)finite measure space and for every $t\in I$ 
let $\mu_t$ be a real- or vector-valued measure on $\R^n\cong \GG$ such that:
\begin{itemizeb}
\item[(a)]
for every Borel set $E$ in $\GG$ the function $t\mapsto \mu_t(E)$ 
is measurable; 
\item[(b)]
$\int_I \Mass(\mu_t) \, dt <+\infty$.
\end{itemizeb}
Then we denote by $\int_I \mu_t\, dt$ the measure on 
$\GG$ defined by
\[
{\textstyle \big[ \int_I \mu_t\, dt \big]}(E)
:= \int_I \mu_t(E) \, dt
\quad\text{for every Borel set $E$ in $\GG$.}
\]
Note that the assumption (a) is equivalent to say that $t\mapsto \mu_t$ is a measurable map from $I$ to the space of finite measures on $\mathbb{G}$ endowed with the weak* topology.
\end{definizione}

We now introduce some notation that will be used throughout the paper.

\begin{definizione}\label{correnticurve}
Let $B$ be a Borel subset of $\R$ and $\gamma:B\to \GG$ be a Lipschitz curve. We denote by $\llbracket \gamma\rrbracket$ the current of finite mass that acts on compactly supported smooth $1$-forms $\omega$ as:
$$\scal{\llbracket \gamma\rrbracket}{\omega}:=\int_B \scal{\gamma'(t)}{\omega(\gamma(t))}dt.$$
% Note that if $B$ is a finite union of compact intervals, then $\llbracket \gamma\rrbracket$ is a normal current.

In the following it will be also useful to write $\llbracket \gamma\rrbracket=\tau_\gamma\rho\Haus^1\trace \im(\gamma)$, where $\rho$ is a suitable non-negative function in $L^1(\Haus^1\llcorner \im(\gamma))$ 
and $\tau_\gamma(x)$ is a unitary Borel vector field that coincides with $\mathscr{C}(x)[\mathfrak{v}_\gamma(x)]$, up to a real (non-zero) multiple, $\rho\Haus^1\trace \im(\gamma)$-almost everywhere. For a definition of the vector field $\mathfrak{v}_\gamma$, see Lemma \ref{campovgamma}.
% and where  $\Haus^1$, thanks to Proposition \ref{identificazionehaus1}, can be thought as either the Hausdorff measure induced by $d_c$ or by the Euclidean distance as they coincide on Lipschitz fragments in $\GG$.
\end{definizione}

With this notation at hand we can introduce the following result, essentially due to Smirnov, see \cite{Smirnov}.

\begin{teorema}\label{smirnov}
Let $\GG$ be a Carnot group and let $\mathbf{T}=\tau \mu$ be a $1$-dimensional normal and horizontal current with $\partial \mathbf{T}=0$ and with $\lvert\tau(x)\rvert=1$ for $\mu$-almost every $x\in \GG$. Then, there exists a family of vector-valued measures $t\mapsto \boldsymbol\mu_t$ satisfying the hypothesis (a) and (b) of Definition \ref{s-measint} such that
\begin{itemize}
    \item[(i)] for almost every $t\in I$ there exists a Lipschitz curve $\gamma_t:[0,1]\to\GG$ for which $\boldsymbol\mu_t=\llbracket \gamma_t\rrbracket$ and 
    $$\scal{\mathbf{T}}{\omega}=\int_I \scal{\llbracket \gamma_t\rrbracket}{\omega}\,dt=\int_I\int\rho_t\scal{\tau_{\gamma_t}}{\omega} d\Haus^1\trace \im(\gamma_t)\, dt,$$
    for every smooth and compactly supported $1$-form $\omega$;
    \item[(ii)] the following identity holds
    $$\Mass(\mathbf{T})=\int_I \Mass(\llbracket \gamma_t\rrbracket)\,dt=\int_I\lVert \rho_t\rVert_{L^1(\Haus^1\llcorner\im(\gamma_t))} \,dt,$$ 
    and in particular  $\tau(x)=\tau_{\gamma_t}(x)$
   for $\Haus^1$-almost every $x\in\im(\gamma_t)$ and for almost every $t\in I$;
   \item[(iii)] the measure $\mu$ can be written as $\mu=\int_I \rho_t\Haus^1\trace \mathrm{im}(\gamma_t)dt$.
\end{itemize}
Further, one can also rewrite $\mathbf{T}$ as 
\begin{equation}\label{e:smirnov1}
\scal{\mathbf{T}}{\omega}=\int_\R\int_I\scal{\tau_{\gamma_t}}{\omega} d\Haus^1\trace (\im(\gamma_t)\cap E_{t,s})\, dtds,    
\end{equation}
where $E_{t,s}$ is a suitable Borel subset of $\im(\gamma_t)$ and the map $(t,s)\mapsto\Haus^1\trace (\im(\gamma_t)\cap E_{t,s})$ satisfies the hypothesis (a) and (b) of Definition \ref{s-measint}. In addition 
\begin{equation}\label{e:smirnov2}
\Mass(\mathbf{T})=\int_\R\int_I\Haus^1(\im(\gamma_t)\cap E_{t,s}) \,dtds\qquad \text{and}\qquad \mu=\int_\R\int_I \Haus^1\llcorner (\im(\gamma_t)\cap E_{t,s})\,dtds ,
\end{equation}
with $\tau_{\gamma_t}(x)=\tau(x)$ for $\Haus^1$-almost every $x\in\im(\gamma_t)\cap E_{t,s}$ and for almost every $(s,t)\in \R\times I$. 
\end{teorema}

\begin{proof}
Thanks to \cite[Theorem 3.1]{PaoliniStepanov}, there exists a family of vector-valued measures $t\mapsto \mu_t$ satisfying the hypothesis (a) and (b) of Definition \ref{s-measint} such that for $\mathcal{L}^1$-almost every $t\in [0,\Mass(T)]$ there exists a Lipschitz curve $\gamma_t:[0,1]\to \mathbb{G}$ such that $\boldsymbol\mu_t=\llbracket \gamma_t\rrbracket$ and
\begin{equation}
    \begin{split}
        \textbf{T}=\int_0^{\Mass( \textbf{T})}\llbracket \gamma_t\rrbracket dt\qquad\text{and}\qquad\Mass( \textbf{T})=\int_0^{\Mass( \textbf{T})}\Mass(\llbracket \gamma_t\rrbracket) dt=\int_0^{\Mass( \textbf{T})}\int_0^1\lvert\gamma_t^\prime(s)\rvert ds\, dt.
    \end{split}
\end{equation}
% Finally, it also holds that $\Mass(\llbracket \gamma_t\rrbracket)=\int_0^1\lvert\gamma_t^\prime(s)\rvert ds=1$ for $\Leb^1$-almost every $t\in[0,\Mass[T]]$. 
The proof of the first part of the result can be obtained from the above discussion with the same argument used for \cite[Theorem 5.5]{AlbMar}. The only variation on \cite[Theorem 5.5]{AlbMar} is how to prove that the curves $\gamma_t$ used to decompose the current are Lipschitz, where the codomain is endowed with the \cc metric. This can be obtained as follows. 
The argument in \cite[Theorem 5.5]{AlbMar} implies that for $\Haus^1$-almost every $x\in\im(\gamma_t)$ and almost every $t\in I$ we have 
$$H\mathbb{G}(x)\ni\tau(x)=\tau_{\gamma_t}(x),$$
which coincides with $\mathfrak{v}_{\gamma_t}(x)$ for $\Haus^1$-almost every $x\in \im(\gamma_t)$ and almost every $t\in I$, see for instance Definition \ref{correnticurve}.
This implies thanks to \cite[Proposition 1.3.3]{tesimonti} that for $\Haus^1\trace\im(\gamma_t)$-almost every $x\in \R^n$ and almost every $t$ the curve $\gamma_t$ is horizontal and thus Lipschitz if seen as curve $\gamma_t:[0,1]\to \mathbb{G}$.

The second part of the statement can be obtained defining $E_{t,s}:=\{x\in\mathrm{im}(\gamma_t):\rho_t(x)\geq s\}$ by applying the Cavalieri formula $\rho_t\mathcal{H}^1\llcorner \im(\gamma_t)=\int_0^\infty\mathcal{H}^1\llcorner (\im(\gamma_t)\cap \{\rho_t\geq \lambda\}) d\lambda$.
\end{proof}

\begin{osservazione}
Thanks to \cite[Remark 2.7 (iii)]{AlbMar} one can also rewrite \eqref{e:smirnov1} and \eqref{e:smirnov2} as integrals with respect to a single real variable as follows
\begin{equation}
    \scal{\mathbf{T}}{\omega}=\int_{\R}\scal{\tau_{\gamma_t}}{\omega} d\Haus^1\trace (\im(\gamma_t)\cap E_{t})\, dt\qquad \Mass(\mathbf{T})=\int_\R\Haus^1(\im(\gamma_t)\cap E_{t}) \,dt\qquad \mu=\int_\R \Haus^1\llcorner (\im(\gamma_t)\cap E_{t})\,dt.
    \nonumber
\end{equation}
\end{osservazione}

\begin{osservazione}\label{remark:rappresentazione}
Let $\mathbf{T}$ be a horizontal normal current such that $\partial \mathbf{T}=0$. Then for any smooth  compactly supported $1$-form we have 
\begin{equation}
    \langle\partial \mathbf{T}; \omega\rangle=\int_{\R^n}\langle \tau(x);d\omega(x)\rangle d\mu(x)=\int\langle \tilde{\tau}(x),d_H \omega(x)\rangle_{\R^{n_1}} d\mu(x),\footnote{In the following we will drop the subscript $\R^{n_1}$ in the scalar product in the first layer if not otherwise specified.}
    \label{rappre1}
\end{equation}
where $\langle \cdot,\cdot\rangle_{\R^{n_1}}$ is the scalar product on $\R^{n_1}$ and
\begin{equation}
\begin{split}
     \tilde  \tau(x):=\sum_{i=1}^{n_1}\tau_i(x)\partial_i\qquad \text{and}\qquad
     d_H\omega(x):=\sum_{i=1}^{n_1} X_i\omega(x)dx_i,
     \label{rappre2}
\end{split}
\end{equation}
where $\tau(x)=\sum_{i=1}^{n_1}\tau_i(x)X_i(x)$. It will be convenient in the following to view horizontal finite mass $1$-dimensional currents as Radon measures $\mathbf{T}\in\mathcal{M}(\mathbb{G},\R^{n_1})$ which acts by duality on vector-valued smooth function $\omega\in \mathcal{C}^\infty(\mathbb{G},\R^{n_1})$.
\end{osservazione}

\section{The decomposability bundle}
\label{s-decomp}

In this section we introduce an intrinsic notion of decomposability bundle to the setting of Carnot groups and we prove some of its elementary properties. 

First of all, we need to introduce a notion of intrinsic Grassmannian.

\begin{proposizione}[{\cite[Proposition 2.3]{MerloAntonelli}}]\label{prop:CompGrassmannian}
Fix $1\leq h\leq Q$. For any $ W_1, W_2\in \Gr(\GG)$ let
$$
d_{\mathbb G}(W_1, W_2):=d_{\Haus,\mathbb G}( W_1\cap B(0,1), W_2\cap B(0,1)),
$$
where $d_{\Haus,\mathbb G}$ is the Hausdorff distance of sets induced by the some homogenous left invariant distance $d$ on $\mathbb{G}$. Then, $d_\mathbb{G}$ is a metric on $\Gr(h)$. Moreover $(\Gr(h),d_{\mathbb G})$ is a compact metric space for any $h\in\{1,\ldots,Q\}$ and thus $(\Gr(\GG),d_{\mathbb G})$ is a compact metric space as well.
\end{proposizione}

\begin{proposizione}[{\cite[Proposition 2.10]{MerloAntonelli}}]\label{prop:projmap}
Let $\mathcal{L}(\GG,\GG)$ be the set of linear maps from the vector space underlying $\GG$ into itself, endowed with the operator norm $\rho$. Then, the following are equivalent
\begin{itemize}
    \item[(i)]$V:\GG\to \Gr_{\mathrm{eu}}(\GG)$ is a Borel map;\footnote{Note that we are allowing $V$ to take values in the Euclidean Grassmannian.}
    \item[(ii)]the projection map $\pi_V:\GG\to\mathcal{L}(\GG,\GG)$, defined as
$\pi_V(x):=\Pi_{V(x)}$ where $\Pi_{V(x)}$ is the Euclidean orthogonal projection onto $V(x)$, is Borel;
\item[(iii)]the projection map $\pi_{V^\perp}:\GG\to\mathcal{L}(\GG,\GG)$  defined as
$\pi_{V^\perp}(x):=\Pi_{V(x)^\perp}$ where $\Pi_{V^\perp(x)}$\footnote{From here on, if not otherwise specified the symbol $V^\perp$ will denote the \emph{Euclidean} orthogonal of $V$ in $\mathbb{G}\cong\R^n$.} is the Euclidean orthogonal projection onto $V^\perp(x)$,
is Borel.
\end{itemize}
Finally the Borelianity of $\pi_V$ is also equivalent to saying that  for any fixed $v,w\in \GG$, seen as vectors of coordinates, the map $x\mapsto \langle v,\pi_V(x)[w]\rangle$ is Borel.
\end{proposizione}

\begin{proof}
The proof of the proposition is omitted. It can be achieved by proving that the map $\Psi$ associating an element of the Grassmannian $V\in \Gr_{\mathrm{eu}}(\mathbb{G})$ to its Euclidean orthogonal projection $\Pi_V$ is an homeomorphism. Actually what can be shown is that $\Psi$ is bi-H\"older.
\end{proof}

\begin{proposizione}\label{prop:Borelintersez}
For any couple of Borel distributions of homogeneous subgroups $V,W:\GG\to \Gr_{\mathrm{eu}}(\GG)$, the intersection map $V\cap W(x):= V(x)\cap W(x)$ is Borel. 
\end{proposizione}

\begin{proof}
Since for any $\mathbb{V}, \mathbb{W}\in \mathrm{Gr}(\GG)$ we have $\mathbb{V}\cap \mathbb{W}=(\mathbb{V}^\perp+\mathbb{W}^\perp)^\perp$, Proposition \ref{prop:projmap} implies that in order to prove the claim it suffices to show the Borelianity of the sum (as vector subspaces of $\GG\cong \R^n$) of the maps $V^\perp$ and $W^\perp$. Let $e_1,\ldots,e_n$ be a basis of $\mathbb{R}^n$, that as recalled above is underlying vector space of $\mathbb{G}$, and define
$$\zeta_i(x):=\pi_{V^\perp(x)}[e_i]\qquad \mathrm{and} \qquad \zeta_{i+n}(x):=\pi_{W^\perp(x)}[e_i]\qquad \text{for any $i=1,\ldots,n$}.$$
The vector fields $\zeta_1,\ldots,\zeta_{2n}:\mathbb{G}\to\mathbb{G}$ 
are Borel thanks to Proposition \ref{prop:projmap}. In addition by construction we know that the vector fields $\{\zeta_1,\ldots, \zeta_n\}$ span $V(x)^\perp$ while the $\{\zeta_{n+1},\ldots,\zeta_{2n}\}$ span $W(x)^\perp$ for any $x\in\mathbb{G}$. This implies in particular that  $\{\zeta_1(x),\ldots,\zeta_{2n}(x)\}$ span $V^\perp(x)+ W^\perp(x)$ for any $x\in \mathbb{G}$. 

Now, from the $\zeta_i$s we construct some other vector fields $\omega_1,\ldots,\omega_{2n}$ that still span $V^\perp+W^\perp$ defined in the following inductive way. As a first step we define  $$\omega_1(x):=\begin{cases}\zeta_1(x)/\lvert \zeta_1(x)\rvert& \mathrm{if }\,\zeta_1(x)\neq 0\mathrm{,}\\0& \mathrm{otherwise.}\end{cases}$$
Notice that the vector field $\omega_1(x)$ is trivially seen to be Borel.
As a second step, suppose that we already defined the vector fields $\omega_1,\ldots,\omega_{k-1}$. Define 
$\tilde{\omega}_k(x):= \zeta_k(x)-\sum_{i=1}^{k-1}\langle \zeta_k(x),\omega_i(x)\rangle\omega_i(x)$ and note that $\tilde{\omega}_k$ is a Borel vector field. Finally, we define $\omega_k$ as:
$$\omega_k(x)=\begin{cases}\tilde{\omega}_k(x)/\lvert \tilde{\omega}_k(x)\rvert& \mathrm{if }\,\tilde{\omega}_k(x)\neq 0\mathrm{,}\\0& \mathrm{otherwise.}\end{cases}$$
Notice that $\omega_k$ is a Borel vector field as well.
Thanks its very definition, we see that for any fixed $x\in \mathbb{G}$ there are only $\mathrm{dim}(V^\perp(x)+W^\perp(x))$ non-null elements of $\{\omega_1(x),\ldots,\omega_{2n}(x)\}$ and those that are not null form an orthonormal basis of $V^\perp(x)+W^\perp(x)$.
In particular, for any $x\in\mathbb{G}$ the map $\pi_{V(x)^\perp+W^\perp(x)}$ is easily seen to be represented as:
$$\pi_{V(x)^\perp+W^\perp(x)}=\sum_{i=1}^{2n}\omega_i(x)\otimes\omega_i(x),$$
which is a Borel matrix field. This concludes the proof by Proposition \ref{prop:projmap}.
\end{proof}

\begin{lemma}
\label{e-minmap}
Let $\GG$ be a Carnot group. Let $\mu$ be a measure on $\GG$ and 
let $\G$ be a family of Borel maps from 
$\GG$ to $\Gr(\GG)$ which is closed under
countable intersection, in the sense
that for every countable family $\{V_i\} \subset\G$
the map $V$ defined by $V(x):=\cap_i V_i(x)$ 
for every $x\in\GG$ belongs to $\G$.

Then $\G$ admits an element $V$ which is $\mu$-minimal,  
in the sense that every other $V'\in\G$ satisfies 
$V(x) \subset V'(x)$ for $\mu$-almost every~$x$.
Moreover this $\mu$-minimal element is unique modulo 
equivalence $\mu$-almost everywhere.
\end{lemma} 

\begin{proof}
The proof of the lemma is identical to its Euclidean counterpart, see \cite[Lemma 2.4]{AlbMar}.
\end{proof}

\begin{definizione}
Let $\mu$ be a measure on $\GG$, let $\F$ be a family of 
Borel vector fields on $\GG$ and let $\G$ be the class 
of all Borel maps $V:\GG\to \Gr(\GG)$ such that for 
every $\tau\in \F$ there holds
\[
\mathfrak{S}(\tau(x)) \subseteq  V(x)
\quad\text{for $\mu$-almost every~$x$,}
\]
where $\mathfrak{S}(\tau(x))$ denotes the     homogeneous subgroup generated by $\tau(x)$, see Definition \ref{definitiongenerazione}.
Since $\G$ is closed under 
countable intersection, see Proposition \ref{prop:Borelintersez}, 
by Lemma~\ref{e-minmap} it admits 
a $\mu$-minimal element which is unique 
modulo equivalence $\mu$-almost everywhere.
We call \emph{any} of these minimal elements the
\emph{$\mu$-essential span} of $\F$.
\end{definizione}

\begin{lemma}\label{lemmamisurabilitatangenti}
Suppose that $\{\mu_t\}_{t\in I}$ is a family of Radon measures satisfying the hypothesis  (a) and (b) of Definition \ref{s-measint} and for which for almost every $t$ there exists a Lipschitz curve $\gamma_t$ such that $\mu_t=\Haus^1\trace\im(\gamma_t)$. Then, the map $(x,t)\mapsto \Tan(\gamma_t,x)\in \mathrm{Gr}(\mathbb{G})$ is Borel.
\end{lemma}

\begin{proof}
The proof of the lemma can be achieved adapting the argument used to prove \cite[Lemma 6.10]{AlbMar}. With such argument one proves that the map $(x,t)\mapsto \nu_{t,x}$ is Borel, where $\nu_{t,x}$ is the $1$-dimensional tangent measure of $\mu_t$ at $x$, see \cite{antonellimerloexistencedensity,antonellimerlomarstrandmattila,MerloAntonelli,Kirchheimarea}. For almost every $t\in I$ and $\Haus^1\trace \im(\gamma_t)$-almost every $x\in\GG$,
the tangent measure $\nu_{t,x}$ is seen easily to be $\Haus^1\trace\mathfrak{S}(\mathfrak{v}_{\gamma_t}(x))$, where $\mathfrak{v}_{\gamma_t}(x)$ is the vector field introduced in Lemma \ref{campovgamma}. This concludes the proof. 
\end{proof}

\begin{definizione}[Decomposability bundle]
\label{Dec-bundle}
Let $\GG$ be a Carnot group. Given a measure $\mu$ on $\GG$ we denote by
$\F_{\mu}$ the class of all families of measures
$\{\mu_t : t\in I\}$ where $I$ is a measured space
endowed with a probability measure $dt$ and:
\begin{itemizeb}
\item[(\hypertarget{a}{a})]
each $\mu_t$ is the restriction of $\Haus^1$ to a $1$-Lipschitz curve $\gamma_t$ in $\GG$;
\item[(b)]
the map $t\mapsto \mu_t$ satisfies the assumptions
(a) and (b) in \S\ref{s-measint};
\item[(c)]
the measure $\int_I \mu_t \, dt$ is 
absolutely continuous with respect to $\mu$.
\end{itemizeb}
We denote by $\G_{\mu}$ the class of all Borel maps
$V:\GG\to \Gr(\GG)$ such that for every 
$\{\mu_t=\Haus^1\trace \im(\gamma_t): t\in I\}\in\F_{\mu}$ 
there holds
\begin{equation}\label{test}
\mathfrak{S}(\mathfrak{v}_{\gamma_t}(x))=\Tan(\gamma_t,x) \subset V(x)
\quad\text{for $\mu_t$-almost every~$x$ and almost every~$t\in I$.}
 \end{equation}
Since $\G_{\mu}$ is closed under countable intersection, by Lemma~\ref{e-minmap} it admits a $\mu$-minimal element.
% , namely the $\mu$-essential span of the family $\F_{\mu}$.
We call \emph{any} of these minimal elements the
\emph{decomposability bundle of $\mu$}, 
and denote it by $x\mapsto V(\mu,x)$.
\end{definizione}

\begin{osservazione}\label{rkfbeit}
If we substitute to (\hypertarget{a}{a}) the assumption
\begin{itemize}
    \item[(a*)]each $\mu_t$ is an absolutely continuous with respect to the restriction of $\Haus^1$ to a Lipschitz curve $\gamma_t$ in $\GG$,
\end{itemize}
the definition of decomposability bundle does not change. If we denote with $V^*(\mu,\cdot)$ the decomposability bundle that arises from the assumptions (a*), (b) and (c), it is clear that 
$V^*(\mu,x)\subseteq V(\mu,x)\text{ for $\mu$-almost every $x\in\GG$.}$
Therefore, we only need to prove the converse inclusion, i.e. that for any family of measures $\mu_t$ satisfying (a*), (b) and (c) we have that
\begin{equation}
\mathfrak{S}(\mathfrak{v}_{\gamma_t}(x))=\Tan(\gamma_t,x) \subset V(\mu,x)
\quad\text{for $\mu_t$-almost every~$x$ and almost every~$t\in I$.}
\label{eq:finale2}
\end{equation}

In order to see this, let $\gamma: K\to\GG$ be a Lipschitz curve and suppose $\mu$ is a finite measure on $\GG$ such that $\mu\ll\Haus^1\trace\im(\gamma)$. The Radon-Nikodym's decomposition theorem implies that there exists a $\rho\in L^1(\Haus^1\trace \im(\gamma))$ such that $\mu=\rho\Haus^1\trace\im(\gamma)$. Let $A\subseteq \GG$ be any Borel set and note that the map $t\mapsto \Haus^1(A\cap \{x:\rho(x)\geq t\})$ monotone. Hence, the measures $\nu_t:=\Haus^1(A\cap \{x:\rho(x)\geq t\})$ satisfy the hypothesis (a) and (b) of Definition \ref{s-measint} and thus their integral $\tilde\mu:=\int_0^\infty \nu_t e^{-t}dt$ is well defined. It is an easy task to check that the measures $\mu$ and $\tilde{\mu}$ are mutually absolutely continuous.

Thus, let $\{\mu_t\}_{t\in I}$ be a family of measures satisfying the hypothesis (a*), (b) and (c). For any $t\in I$ and any $s\in[0,\infty)$ we denote by $\nu_{s,t}$ the measure $\nu_{s,t}:=\Haus^1\trace \im(\gamma_t)\cap \{x:\rho_t(x)\geq s\}$. It is immediate to see that the measures $(s,t)\mapsto \nu_{s,t}$ satisfy item (a) and a standard argument shows that they also satisfy item (b). In addition, the above discussions proves that the measures $\int\int \nu_{s,t}e^{-s}dtds$ and $\int\mu_tdt$ are mutually absolutely continuous
and thus the $\nu_{s,t}$ satisfy also (c). By definition of $V(\mu,x)$ this implies that 
\begin{equation}
\mathfrak{S}(\mathfrak{v}_{\gamma_t}(x))=\Tan(\gamma_t,x) \subset V(\mu,x)
\quad\text{for $\nu_{s,t}$-almost every~$x$ and almost every~$(s,t)\in I$.}
\label{eq:finale1}
\end{equation}
On the other hand, again by the above discussion we know that $\int \nu_{s,t}e^{-s}ds$ and $\mu_t$ are mutually absolutely continuous, and thus from \eqref{eq:finale1} we infer that \eqref{eq:finale2} holds. This shows that $V^*(\mu,\cdot)=V(\mu,\cdot)$.
\end{osservazione}

\begin{lemma}
\label{prop:V=carnot}
Let $\mu$ be a Radon measure on $\mathbb{G}$. Then, $V(\mu,x)\in \mathrm{Gr}_\mathfrak{C}(\GG)$ for $\mu$-almost every $x\in\GG$. In other words $V(\mu,x)$ coincides with the closed subgroup of $\GG$ generated by $V_1\cap V(\mu,x)$ for $\mu$-almost every $x\in \GG$.
\end{lemma} 

\begin{proof}
Since $\Tan(\gamma_t,x)\subseteq V_1$ for $\mu_t$-almost every $x$, by definition of $V(\mu, x)$ we have that:
$$\Tan(\gamma_t,x)\subseteq V(\mu, x)\cap V_1 \quad\text{for $\mu_t$-almost every~$x$ and almost every~$t\in I$.}$$
Furthermore since $\mathfrak{S}(V_1\cap V(\mu, x))$, the homogeneous subgroup generated by $V_1\cap V(\mu, x)$ is contained in $V(\mu, x)$ for every $x\in \GG$, we just need to show that the map $x\mapsto \mathfrak{S}(V_1\cap V(\mu, x))$ is Borel and thus it is a competitor in the definition of $V(\mu, x)$.

The map $x\mapsto \HH V(x):=V_1\cap V(\mu,x)$ is Borel measurable thanks to Proposition \ref{prop:Borelintersez} and hence, since every element $W$ of $\Gr_\mathfrak{C} (\GG)$ is uniquely determined by $W\cap V_1$  and $\mathfrak{S}(V_1\cap V(\mu, x)) \in \Gr_\mathfrak{C} (\GG)$, we infer that for any closed set $C\subseteq \Gr (\GG)$ we have:
\begin{equation}
\begin{split}
        &\mathfrak{S}(V_1\cap V(\mu, x))^{-1}(C)=\mathfrak{S}(V_1\cap V(\mu, x))^{-1}(C\cap \Gr_\mathfrak{C} (\GG))\\
        =&\HH V^{-1}(\{V_1\cap W\in \Gr(V_1):W\in C\cap \Gr_\mathfrak{C}(\GG)\})
        =\HH V^{-1}(\{V_1\cap W\in \Gr(V_1):W\in C\}).
\end{split}
\label{eq:nummm20}
\end{equation}
Since $C$ is closed, the set $\{V_1\cap W\in \Gr(V_1):W\in C\}$ is easily proved to be closed. Finally, thanks to the Borelianity of $\HH V$ and \eqref{eq:nummm20} we thus infer that $\mathfrak{S}(V_1\cap V(\mu, x))^{-1}(C)$ is Borel as well and the proof of the proposition is concluded.
\end{proof}

\begin{definizione}
Let us fix a Radon measure $\mu$ on $\GG$. For any element $F\in \F_{\mu}$ we consider the family of all Borel maps $V:\GG\to\Gr_\mathfrak{C}(\GG)$ for which \eqref{test} holds. Since this class by Proposition \ref{prop:Borelintersez} is closed by countable intersection, by Lemma \ref{e-minmap} it admits a $\mu$-minimal element $\mathfrak{V}(\mu,F,\cdot)$ that is unique modulo equivalence $\mu$-almost everywhere.

Furthermore, since the elements of $F$ are Euclidean $1$-rectifiable measures, $F$ is also a decomposition of the measure $\mu$ in the Euclidean sense, see \cite[\S 2.6]{AlbMar}. Consider the family of all Borel maps $V:\GG\to\Gr_{eu}(\GG)$ for which \cite[Equation (2.2)]{AlbMar} holds and note that since this class is closed by countable intersection, by Lemma \ref{e-minmap} it admits a $\mu$-minimal element denoted by $\mathfrak{V}_{eu}(\mu,F,\cdot)$.
\end{definizione}

\begin{proposizione}\label{prop:id:fib}
For any Radon measure $\mu$ on $\GG$ and any $F=\{\Haus^1\trace \im(\gamma_t):t\in I\}\in\mathscr{F}_{\mu}$ for $\mu$-almost every $x\in \GG$ we have:
$$\mathfrak{V}_{eu}(\mu,F,x)=\mathscr{C}(x)[\mathfrak{V}(\mu,F,x)\cap V_1].$$
\end{proposizione}

\begin{proof}
Let us prove that $\mathfrak{V}_{eu}(\mu,F,x)\subseteq \mathscr{C}(x)[\mathfrak{V}(\mu,F,x)\cap V_1]$ for $\mu$-almost every $x\in\mathbb{G}$. By definition of $\mathfrak{V}(\mu,F,x)$ for almost every $t\in I$ and $\Haus^1\trace \im(\gamma_t)$-almost every $x\in\mathbb{G}$ we have:
\begin{equation}
    \Tan_{\mathrm{eu}}(\gamma_t,x)=\mathscr{C}(x)[\Tan(\gamma_t,x)]\subseteq  \mathscr{C}(x)[\mathfrak{V}(\mu,F,x)\cap V_1].
    \label{eq:incl1un}
\end{equation}
 The last part of Proposition \ref{prop:projmap} implies that the distribution of planes $\mathscr{C}(\cdot)[\mathfrak{V}(\mu,F,\cdot)\cap V_1]$ is Borel measurable and thanks to \eqref{eq:incl1un} it is a competitor in the definition of $\mathfrak{V}_{eu}(\mu,F,x)$. This shows in particular that $\mathfrak{V}_{eu}(\mu,F,x)\subseteq \mathscr{C}(x)[\mathfrak{V}(\mu,F,x)\cap V_1]$.
Concerning the converse inclusion,  by definition of $\mathfrak{V}_{eu}(\mu,F,x)$ we know that for almost every $t\in I$ and $\Haus^1\trace \im(\gamma_t)$-almost every $x\in\mathbb{G}$ we have
$$\Tan(\gamma_t,x)=\pi_1(\Tan_{\mathrm{eu}}(\gamma_t,x))\subseteq  \pi_1(\mathfrak{V}_{eu}(\mu,F,x)).$$
Let us note that the distribution of planes $\mathfrak{S}(\pi_1(\mathfrak{V}_{eu}(\mu,F,x)))$ is seen to be Borel with the same argument used to prove Lemma \ref{prop:V=carnot}. Thus, it is a competitor in the definition of $\mathfrak{V}(\mu,F,x)$  we infer that
$$\mathfrak{V}(\mu,F,x)\subseteq \mathfrak{S}(\pi_1(\mathfrak{V}_{eu}(\mu,F,x)))\qquad\text{for $\mu$-almost every }x\in \GG.$$
Since $\mathfrak{V}(\mu,F,x)$ is uniquely determined by its intersection with $V_1$, and the resulting distribution is still Borel measurable by Proposition \ref{prop:Borelintersez}, we infer that
$$\mathfrak{V}(\mu,F,x)\cap V_1\subseteq \mathfrak{S}(\pi_1(\mathfrak{V}_{eu}(\mu,F,x)))\cap V_1=\pi_1(\mathfrak{V}_{eu}(\mu,F,x))\qquad\text{for $\mu$-almost every }x\in \GG.$$
This concludes the proof of the proposition.
\end{proof}

\begin{proposizione}\label{s-generator}
Let $\mu$ be a Radon measure on $\GG$, then
\begin{itemizeb}
\item[(i)]
for every $F \in \F_{\mu}$ there holds
$\mathfrak{V}(\mu,F,x) \subseteq V(\mu,x)$ for $\mu$-almost every~$x$;
\item[(ii)]
there exists $G \in \F_{\mu}$
such that $\mathfrak{V}(\mu, G, x) = V(\mu, x)$ 
for $\mu$-almost every~$x$.
\end{itemizeb}
\end{proposizione}

\begin{proof}
The proof of the proposition is identical to its Euclidean counterpart, see \cite[Proposition 2.8]{AlbMar}. \end{proof}

\begin{proposizione}\label{s-basic}
Let $\mu$, $\mu'$ be measures on $\GG$.
Then, the following statements hold:
\begin{itemizeb}
\item[(i)]
{\textrm{[strong locality principle]}}
if $\mu' \ll \mu$ then $V(\mu', x) = V(\mu, x)$ 
for $\mu'$-almost every~$x$. More generally, 
if $1_E \, \mu' \ll\mu$ for some Borel set $E\subset\GG$,
then $V(\mu', x) = V(\mu, x)$ for $\mu'$-almost every~$x\in E$,
\item[(ii)] there exists a $G=\{\mu_t : t\in I\} \in \F_{\mu}$ such that for $\mu$-almost every~$x$ we have $\mathfrak{V}(\mu, G, x) = V(\mu, x)$ and:
$$\mu\trace\{x:V(\mu,x)\neq\{0\}\}\ll\int_I\mu_t dt.$$
\end{itemizeb}
\end{proposizione}

\begin{proof}
The proof of (i) is identical to its Euclidean counterpart, see \cite[Proposition 2.9]{AlbMar}.
In order to prove (ii) let $G\in\F_{\mu}$ be the family of measures given by Proposition \ref{s-generator} (ii). Consider the Radon Nikodym decompostion of $\mu':=\mu\trace \{x:V(\mu,x)\neq\{0\}\}$ with respect to $\nu:=\int_I\mu_t dt$ and let $E$ be a Borel set such that $\mu'\trace E\ll\nu$ and $\nu(\GG\setminus E)=0$. Observe that the choice of $E$ implies that $\mu_t(\GG\setminus E)=0$ for almost every $t\in I$. We need to prove that $\mu'(\GG\setminus E)=0$. Assume by contradiction that this is not the case, and observe that by point (i) the family $G':=\{\mu_t\trace(\GG\setminus E)\}\in\F_{\mu\trace(\GG\setminus E)}$ satisfies $$\{0\}=V(\mu'\trace(\GG\setminus E), G', x)=V(\mu'\trace(\GG\setminus E), x),$$
for $\mu'$-almost every $x\in \GG\setminus E$. This contradicts the fact that $V(\mu,x)\neq\{0\}$ for $\mu'$-almost every $x$.
\end{proof}

\begin{proposizione}\label{normalcurrdec}
Let $\mu$ be a Radon measure on $\GG$ and let $\mathbf{T}$ be an horizontal $1$-dimensional normal current, see Definition \ref{defin:hor:curr}. Then
\begin{equation}
    \mathfrak{S}(\pi_1(\tau(x)))\subseteq V(\mu,x)\qquad \text{for $\mu$-almost every $x\in\GG$},
    \label{eq.incl.tangent}
\end{equation}
where $\tau$ is the Radon-Nikodym derivative of $\mathbf{T}$ with respect to $\mu$.
\end{proposizione}

\begin{proof}
Let $\mathbf{T}=\tau'\mu'$ with $\lvert \tau^\prime(x)\rvert=1$ for $\mu$-almost every $x\in \GG$ and let $\{\boldsymbol\mu_t\}_{t\in I}$ be the family of rectifiable measure yielded by Theorem \ref{smirnov}. Thanks to Theorem \ref{smirnov}, we know that for any $t\in I$ we have $\boldsymbol\mu_t=\llbracket \gamma_t\rrbracket$ and for almost every $t\in I$ and $\Haus^1$-almost every $x\in \im(\gamma_t)$
\begin{equation}
    \tau'(x)=\tau_{\gamma_t}(x)=\lambda\mathscr{C}(x)[\mathfrak{v}_{\gamma_t}(x)]\qquad\text{for some $\lambda\neq 0$},
    \label{eq:tag}
\end{equation}
for $\Haus^1\trace \im(\gamma_t)$ almost every $x$ and almost every $t\in I$.
The last identity follows from the last lines of Definition \ref{correnticurve}. Projecting on $V_1$ identity \eqref{eq:tag}, we deduce that $\pi_1(\tau'(x))=\lambda\mathfrak{v}_{\gamma_t}(x)$ and hence
$$\mathfrak{S}(\pi_1(\tau'(x)))=\mathfrak{S}(\mathfrak{v}_{\gamma_t}(x))=\Tan(\gamma_t,x)\qquad\text{for almost every $t\in I$ and $\Haus^1$-almost every $x\in \im(\gamma_t)$}.$$
However, since $\mu'=\int_I\lVert\boldsymbol\mu_t\rVert dt$, this implies by definition of decomposability bundle that 
\begin{equation}
    \mathfrak{S}(\pi_1(\tau'(x)))\subseteq V(\mu',x)\qquad \text{for $\mu'$-almost every $x\in\GG$.}
    \label{num:dir:inc}
\end{equation}
Write $\mathbf{T}=\tau'\mu'=\tau\mu+\boldsymbol\mu_s'$, where $\mu$ and $\boldsymbol\mu_s'$ are mutually singular and where here the vector field $\tau$ is the Radon-Nikodym derivative of $\textbf{T}$ with respect to $\mu$ as in the statement of the proposition. Let $A$ be a Borel set, yielded by the Hahn decomposition theorem such that $\mu(A)=0$ and $\boldsymbol\mu_s^\prime(A^c)=0$ and denote by $\bar\tau$ a $\mu$-measurable representative $\tau$. Finally, let $E\subseteq A^c$ be the $\mu$-measurable where $\bar \tau\neq 0$. It is immediate to see that $\tau'\mu'\trace E=\tau\mu=(\tau/\lvert\tau\rvert)\lvert \tau\rvert\mu$ and hence by the uniqueness of the polar decomposition, we infer that $\tau'(x)=\tau(x)/\lvert\tau(x)\rvert$ on $\mu$-almost every $x\in E$.
 
Since $\mathbb{1}_E\mu\ll \mu'$, Proposition \ref{s-basic}(i) implies that $V(\mu,x)=V(\mu',x)$ for $\mu$-almost every $x\in E$. This, together with \eqref{num:dir:inc} implies that
$$\mathfrak{S}(\pi_1(\tau(x)))\subseteq V(\mu,x)\qquad \text{for $\mu$-almost every $x\in E$.}$$
This concludes the proof.
% Since however $\tau=0$ for $\mu$-almost every $x\in A^c$, the inclusion \eqref{eq.incl.tangent} is proved.
\end{proof}

\section{Integrals of Lipschitz fragments are pieces of horizontal normal currents}
\label{s4}

This section is devoted to the proof of Proposition \ref{p:normal}.
Proposition \ref{p:normal} shows that any vector-valued measure $\boldsymbol\mu$
which can be represented by integration of natural vector-valued measures associated to Lipschitz fragments can be closed to a horizontal normal current by adding to $\boldsymbol\mu$ another integral of Lipschitz fragments $\boldsymbol\sigma$ whose total variation can be taken singular with respect to any given Radon measure $\eta$. 
 The strategy of the proof partially follows that of \cite[Theorem 1.1]{AlbMar2}, but here the necessity to construct a \emph{horizontal} normal current introduces substantial additional difficulties. 
In our applications we are interested to use these integrals of curves to study some fixed Radon measure $\Phi$ which is absolutely continuous with respect to the total variation of $\boldsymbol\mu$. In particular, if we let $\sigma\perp \Phi+\lVert\boldsymbol\mu\rVert$  we still have $\Phi\ll \lVert\boldsymbol\sigma+\boldsymbol\mu\rVert$, which is only possible if we know that $\boldsymbol\sigma$ and $\boldsymbol\mu$ do not have cancellations on a set of positive $\Phi$-measure,  and still, when studying the local structure of the measure $\Phi$ the choice of $\boldsymbol\sigma$  is essentially negligible.

\begin{proposizione}\label{p:normal}
Let $(I, dt)$ be a $\sigma$-finite measure space, $\eta$ be a positive Radon measure and $t\mapsto \boldsymbol\mu_t$ be a family of vector-valued measures satisfying the hypothesis (a) and (b) of Definition \ref{s-measint} and such that for almost every  $t\in I$ there exists a 1-Lipschitz curve $\gamma_t:K_t\to\GG$ defined on a compact set $K_t$ of $\R$ such that $\boldsymbol\mu_t=\tau_t\Haus^1\trace \im(\gamma_t)$, where $\tau_t$ is a unit vector field such that $\tau_t\in \Tan_\mathrm{eu}(\im(\gamma_t),x)$ for $\Haus^1$-almost every $x\in \im(\gamma_t)$. Further, we let
\begin{equation}\label{e:measure_structure}
\boldsymbol\mu:=\int_I \boldsymbol\mu_t\,dt.
\end{equation}
Then, for every $\varepsilon_0>0$ there exists a normal 1-current $\mathbf{T}$ on $\GG\cong \R^n$ such that $\partial \mathbf{T}=0$, $\Mass(\mathbf{T})\leq 2\Mass(\boldsymbol\mu)+\varepsilon_0$ and $\mathbf{T}=\boldsymbol\mu+\boldsymbol\sigma$, where $\boldsymbol\sigma$ and $\eta$ are mutually singular and $\boldsymbol{\sigma}$ is an integration of horizontal Lipschitz fragments as in \eqref{e:measure_structure}.
\end{proposizione}

\begin{proposizione}\label{estensionetame}
Let $K$ be a compact subset of $\R$ and $\gamma:K\to\mathbb{G}$ be a $1$-Lipschitz curve. Then, for any $\varepsilon>0$ there exist finitely many closed intervals $\{I_j\}_{j=1,\ldots, N}$ and a $1$-Lipschitz curve $\tilde \gamma:\cup_{j=1}^NI_j\to \GG$ such that:
\begin{itemize}
    \item[(i)] $K\subseteq \bigcup\{I_j:j=1,\ldots, N\}$ and $\mathcal{L}^1(\bigcup\{I_j:j=1,\ldots, N\}\setminus K)<\varepsilon\mathcal{L}^1(K)$,
    \item[(ii)] $\tilde \gamma\lvert_K=\gamma$.
\end{itemize}
\end{proposizione}

\begin{proof}
Since $K$ is compact, we can find countably many disjoint open intervals $\{(a_j,b_j)\}_{j\in\N}$ such that:
$$K\cup \bigcup\{(a_j,b_j):j\in\N\}=[\min K,\max K].$$
Let $\varepsilon>0$ arbitrary and choose $N=N(\varepsilon)\in\N$ in such a way that $\sum_{j>N}(b_j-a_j)<\varepsilon$. Then:
$$[\min K,\max K]\setminus \bigcup_{j\leq N}(a_j,b_j)=K\cup \bigcup_{j>N}(a_j,b_j),$$
and in particular $K\cup \bigcup_{j>N}(a_j,b_j)$ is a finite union of closed intervals. Denote now by $\eta_j:[0,d(\gamma(a_j),\gamma(b_j))]\to \GG$ be the geodesic joining $\gamma(a_j)$ and $\gamma(b_j)$ and note that if we let:
\begin{equation}
    \tilde\gamma(t):=\begin{cases}
    \gamma(t) & \text{ if $t\in K$},\\
    \eta_j\Big(\frac{d(\gamma(a_j),\gamma(b_j))(t-a_j)}{b_j-a_j}\Big)& \text{ if $t\in (a_j,b_j)$},
    \end{cases}
\end{equation}
then $\tilde \gamma$ satisfies (ii). 

In order to conclude the proof, we just need to check that $\tilde{\gamma}$ is a $1$-Lipschitz curve. We check here only the most complicated case in which $s_i\in (a_{j_i},b_{j_i})$ for some $j_1\neq j_2$. Assume without loss of generality that $b_{j_1}<a_{j_2}$ and note that:
\begin{equation}
    \begin{split}
        d(\tilde \gamma(s_2),\tilde \gamma(s_1)) \leq& d( \tilde \gamma (b_{j_1}),\tilde \gamma(s_1))+d( \tilde \gamma (b_{j_1}),\tilde \gamma(a_{j_2}))+d(\tilde \gamma(a_{j_2}),\tilde \gamma(s_2))\\
        \leq & \frac{d(\gamma(a_{j_1}),\gamma(b_{j_1}))}{b_{j_1}-a_{j_1}}(b_{j_1}-s_1)+(a_{j_2}-b_{j_1})+\frac{d(\gamma(a_{j_2}),\gamma(b_{j_2}))}{b_{j_2}-a_{j_2}}(s_2-a_{j_2})\leq(s_2-s_1).
        \nonumber
    \end{split}
\end{equation}
This concludes the proof.
\end{proof}

\begin{definizione}[Strong graph metric]\label{stronghausmetric}
For any couple of Lipschitz curves $\gamma_i:\dom(\gamma_i)\to \mathbb{G}$ with $i=1,2$, we let their \emph{strong Hausdorff distance} to be
$$d_{\mathfrak s\Haus}(\gamma_1,\gamma_2):=d_{\Haus,\mathrm{eu}}(\mathrm{gr}(\gamma_1),\mathrm{gr}(\gamma_2))+\lVert\mathbb{1}_{\dom(\gamma_1)}-\mathbb{1}_{\dom(\gamma_2)}\rVert_{L^1},$$
where $\mathrm{gr}(\gamma_i):=\{(t,\gamma_i(t))\in\R^{n+1}:t\in\dom(\gamma_i)\}$ for $i=1,2$, and where $d_{\Haus,\mathrm{eu}}$ is the Hausdorff distance induced by the Euclidean metric on $\R^{n+1}$. 
\end{definizione}

\begin{osservazione}\label{rk:equivhaus}
Note that since the Euclidean metric on $\GG\cong\R^n$ and any left-invariant homogeneous distance on $\GG$ are locally H\"older equivalent, the topologies respectively induced by their Hausdorff distances on $\mathcal{X}_N$  on compact sets are equivalent as well. 
% This is the reason for which the above definition, for our purposes, is given with the Hausdorff distance induced by the Euclidean distance and not with the one induced by the \cc metric.
\end{osservazione}

\begin{definizione}\label{distanceoncurvesmaps}
Let $N\in \N$. In the following we denote by $\mathcal{X}_N$ the family of the curves $\gamma:\dom(\gamma)\to\GG$ where $\gamma$ is a $1$-Lipschitz curve and where $\mathrm{dom}(\gamma)$ is the domain of $\gamma$, which is a union of at most $N$ disjoint compact intervals. 
We endow $\mathcal{X}_N$ with the topology induced by the Hausdorff distance $d_{\mathfrak s\Haus}$ between the corresponding graphs, see Definition \ref{stronghausmetric}.
% see for instance \cite[Definition 2.1]{Bate}.
In addition, we let $\mathcal{X}:=\cup_{N\in\N} \mathcal{X}_N$ be the topological union
\footnote{The topological union is defined as follows: first take the disjoint union of the product topological spaces $\mathcal{X}_N\times\{N\}$ and quotient with the equivalence relation given by the inclusions $\mathcal{X}_N\subseteq \mathcal{X}_M$ provided $N\leq M$, see \cite{MR1039321}.} of the $\mathcal{X}_N$.
\end{definizione}

\begin{proposizione}\label{p:borel_selection}
Let $(I, dt)$ and $\{\boldsymbol\mu_t\}_{t\in I}$  be as in Proposition \ref{p:normal}. Then, for every $\varepsilon>0$ there exists an $N\in\N$, a finite measure space $(\bar{I},d\bar{t})$ and a Borel map $\eta_\varepsilon:\tilde{I}\to \mathscr{M}(\R^n,\R^n)$ such that 
  \begin{itemize}
    \item[(i)] for almost every $t\in \tilde{I}$ we have $\eta_\varepsilon(t)=\llbracket \Gamma(t)\rrbracket$ where the map $\Gamma:\tilde{I}\to \mathcal{X}_N$ is a Borel map with respect to the metric $d_{\mathfrak{s}\Haus}$ introduced in Definition \ref{distanceoncurvesmaps};
    \item[(ii)] $\Mass[\int_I \boldsymbol{\mu}_tdt-\int_{\tilde{I}} \eta_\varepsilon(t)dt]<\varepsilon$.
\end{itemize}
\end{proposizione}

 \begin{proof}
 Throughout the proof, we identify without further comment $\mathbb{G}$ with its underlying vector space $\R^n$ as Lipschitz curves in $\GG$, as it has been previously remarked, are Lipschitz in $\mathbb{G}\cong\R^n$ endowed with the Euclidean metric. 
Without loss of generality, we can assume that the $\boldsymbol\mu_t$'s are supported on the closed ball $B(0,R)$ for some $R>0$ and that $I$ is $\R$ and $dt$ is the Lebesgue measure, see \cite[Remark 2.7 (iii)]{AlbMar}.  Thanks to the fact that the masses of the $\boldsymbol{\mu}_t$ are assumed to be summable, i.e. $\int_\R\Mass(\boldsymbol{\mu}_t)dt<\infty$, for every $\varepsilon>0$ there exists a compact interval $\tilde{I}$ such that $\int_{\R\setminus \tilde{I}}\Mass(\boldsymbol\mu_t)dt<\varepsilon/3$.

For any $\lambda>0$ in the following we denote with $\mathcal{BL}_\lambda$ the family of $(1+\lambda)$-bi-Lipschitz fragments endowed with $d_{\mathfrak{s}\Haus}$,\footnote{It is immediate to see that with such metric $\mathcal{BL}_\lambda$ is a complete metric space.} the metric of strong Hausdorff convergence of graphs, see Definition \ref{stronghausmetric}.

\smallskip

\textbf{Step 1.} In this first step, we show that for any fixed $\lambda>0$ and $\varepsilon>0$ there exists a finite measure space  $(\tilde{I}_\varepsilon,dt)$ and a Borel map
$\tilde\Xi:\tilde{I}_\varepsilon\to \M(B(0,R),\R^n)$ such that
\begin{itemize}
    \item[($\alpha$)] for almost every $t\in\tilde{I}_\varepsilon$ such that $\tilde\Xi(t)=\llbracket \Gamma(t)\rrbracket$, where $\Gamma:\tilde{I}_\varepsilon\to \mathcal{BL}_\lambda$ is Borel;
    \item[($\beta$)] $\Mass[\int_I\boldsymbol\mu_tdt-\int_{\tilde{I}_\varepsilon}\tilde\Xi(s)ds]\leq2\varepsilon/3$.
\end{itemize}
Let us fix an $N\in\N$ and let us define the map $\Psi:\tilde{I}\times \mathcal{BL}_\lambda^N\times \mathscr{M}(B(0,R),\R^n)\to \R\times \R$ as 
$$\Psi_{\lambda,N}(t,\gamma_1,\ldots,\gamma_N,\boldsymbol\nu):=(\Mass(\boldsymbol{\mu}_t-\boldsymbol\nu),\Mass(\boldsymbol\nu-\sum_{j=1}^N\llbracket \gamma_j\rrbracket)),$$ 
and we claim that $\Psi_{\lambda,N}$ is Borel. This is actually equivalent to prove that the maps 
$$\Psi_{1,\lambda,N}:(t,\boldsymbol\nu)\mapsto \Mass(\boldsymbol{\mu}_t-\boldsymbol\nu)\qquad\text{and}\qquad \Psi_{2,\lambda,N}:(\gamma_1,\ldots,\gamma_N,\boldsymbol\nu)\mapsto \Mass(\boldsymbol\nu-\sum_{j=1}^N\llbracket \gamma_j\rrbracket),$$
are Borel. 
Let us note that
\begin{itemize}
    \item[(a)] by assumption  $t\mapsto\boldsymbol\mu_t$ is a Borel map from $\tilde I$ to the space $\mathscr{M}(B(0,R),\R^n)$ endowed with the weak$^\ast$ topology, see Definition \ref{s-measint};
    \item[(b)] $\mathbb{E}:(\boldsymbol\nu_0,\boldsymbol\nu_1,\ldots\boldsymbol\nu_N)\mapsto \mathbb{M}(\boldsymbol\nu_0-\sum_{j=1}^N\boldsymbol\nu_j)$ is Borel since it is the supremum over  countably many maps $\phi\in \mathcal{C}_c(\R^n,\R^n)$ with $|\phi|\leq 1$ of the maps $(\boldsymbol\nu_0,\boldsymbol\nu_1,\ldots\boldsymbol\nu_N)\mapsto \scal{\boldsymbol\nu_0}{\phi}-\sum_{j=1}^N\scal{\boldsymbol\nu_j}{\phi}$ which are continuous with respect to the weak* topology. With the same argument one can also prove that the map $(\mu,\nu)\mapsto \mathbb{M}[\mu-\nu]$ is Borel as well;
    \item[(c)] the map $\mathbb{F}:\mathcal{BL}_\lambda\to \mathscr{M}(B(0,R),\R^n)$ defined as $\mathbb{F}(\gamma):=\llbracket \gamma\rrbracket$ is Borel. 
\end{itemize}
Items (a), (b) and (c) imply that the map $\Psi_{\lambda,N}=(\mathbb{M}(\mu-\nu),\mathbb{E}[\nu, \mathbb{F}[\gamma_1],\ldots,\mathbb{F}[\gamma_N]])$ is Borel, and hence $\Psi^{-1}_{\lambda,N}([0,\varepsilon]\times\{0\})$ is a Borel set. Furthermore, it is not difficult to prove using \cite[Lemma 4, Theorem 7]{Kirchheimarea} that for almost every $t\in\tilde{I}$ there exists an $N_{t,\varepsilon}\in\N$ and finitely many $\{\gamma_{t,j}\}_{j=1,\ldots,N_{t,\varepsilon}}\subseteq  \mathcal{BL}_\lambda$ such that $$\Mass[\boldsymbol\mu_t-\sum_{j=1}^{N_{t,\varepsilon}}\llbracket \gamma_{t,j}\rrbracket]\leq \varepsilon/3\Leb^1(\tilde{I}),\qquad \im(\gamma_{t,j_1})\cap \im(\gamma_{t,j_2})=\emptyset\qquad \text{and} \qquad\sum_{j=1}^{N_{t,\varepsilon}}\Mass[\llbracket \gamma_{t,j}\rrbracket]\leq \Mass[\boldsymbol\mu_t].$$
  This in particular shows that if we denote by $\pi_1$ the projection on $\tilde I$, we have that 
  $$\Leb^1(\tilde I\setminus \pi_1(\bigcup_{N\in\N} \Psi^{-1}_{\lambda,N}([0,\varepsilon/3\Leb^1(\tilde{I})]\times\{0\})))=0.\,\footnote{Recall that projection of Borel sets are Suslin and thus universally measurable.}$$
This shows in particular that there exists an $N\in\N$ such that $\Leb^1(\tilde{I}\setminus \pi_1(\Psi^{-1}_{\lambda,N}([0,\varepsilon/3\Leb^1(\tilde{I})]\times\{0\})))\leq \varepsilon/3$. So, defined $\tilde{I}_\varepsilon:=\pi_1(\Psi^{-1}_{\lambda,N}([0,\varepsilon/3\Leb^1(\tilde{I})]\times\{0\}))$, we have by \cite[Theorem 5.5.2]{Sriva} that there exists a universally measurable map $\Xi:\tilde{I}_\varepsilon\to \mathcal{BL}_\lambda^N\times \mathscr{M}(B(0,R),\R^n)$ defined as $\Xi(t):=(\gamma_{1,t},\ldots,\gamma_{N,t},\boldsymbol\nu_t)$ such that
\begin{itemize}
    \item[(a')] $\Mass[\boldsymbol\nu_t-\sum_{j=1}^N\llbracket \gamma_{j,t}\rrbracket]\leq \varepsilon/3\Leb^1(\tilde{I})$ and $\Mass[\boldsymbol\mu_t-\boldsymbol\nu_t]=0$;
    \item[(b')] for every $j=1,\ldots,N$ we have $\Xi_j:\tilde{I}_\varepsilon\to \mathscr{M}(B(0,R),\R^n)$ defined as $\Xi_j(t):=\llbracket\gamma_{j,t}\rrbracket$ is Borel.
\end{itemize}
It is immediate to see that items (a') and (b') together with \cite[Remark 2.7 (iii)]{AlbMar} imply items ($\alpha$) and ($\beta$).

\smallskip

\textbf{Step 2.} Let us move to the proof of the main statement. 
Let us define the map $\Psi:\tilde{I}_\varepsilon\times \mathcal{X}\times \mathscr{M}(B(0,R),\R^n)\to \R\times \R$  as 
$$\Psi(t,\gamma,\boldsymbol\nu):=\Mass(\tilde\Xi(t)-\mathbb{F}(\gamma)),$$ 
where $\mathbb{F}:\mathcal{X}\to \M(B(0,R),\R^n)$ is the map $\mathbb{F}(\gamma):=\llbracket\gamma\rrbracket$. Arguing as above, one can prove that $\Psi$ is Borel. In addition, thanks to Proposition \ref{estensionetame}, we know that for $t\in \tilde{I}_\varepsilon$ there exists an $M_{t,\varepsilon}\in\N$ and a $\gamma\in \mathcal{X}_{M_{t,\varepsilon}}$ such that
$$\Mass[\mathbb{F}(\gamma)-\tilde\Xi(t)]\leq \varepsilon/6\Leb^1(\tilde{I}_\varepsilon).$$
Hence, arguing as above it is possible to show that there exists an $M_{t,\varepsilon}\in\N$ and a universally measurable map $\eta_\varepsilon:\tilde{I}_\varepsilon\to \M(\R^n,\R^n)$ such that 
\begin{itemize}
    \item[($\alpha'$)] for almost every $t\in\tilde{I}_\varepsilon$ such that $\tilde\Xi(t)=\llbracket \Gamma(t)\rrbracket$, where $\Gamma:\tilde{I}_\varepsilon\to \mathcal{X}_{M_{t,\varepsilon}}$ is Borel;
    \item[($\beta'$)] $\Mass[\int_I\boldsymbol\mu_tdt-\int_{\tilde{I}_\varepsilon}\tilde\Xi(s)ds]\leq\varepsilon$.
\end{itemize}
This concludes the proof of the proposition.

\end{proof}

\begin{proposizione}\label{prop:singcurves}
Let $K$ be a compact subset of the real line of positive $\mathcal{L}^1$-measure, assume that $\gamma:K\to\mathbb{G}$ is a Lipschitz curve and let $\eta$ be a positive and finite Radon measure on $\GG$. Then, for every $\delta>0$ there exists a set of full measure of vectors $v\in B(0,\delta)$ such that $\eta$ and $\mathcal{H}^1\trace (v*\im(\gamma))$ are mutually singular.
\end{proposizione}

\begin{proof} Without loss of generality we can assume that $\eta$ is finite by restricting $\eta$ to a ball that compactly contains $\cup_{v\in B(0,1)}v*\im (\gamma)$. We can further assume that $\eta$ is mutually singular with respect to $\Leb^n$. Indeed, if we write $\eta=\eta_a+\eta_s$ where $\eta_a\ll\Leb^n$ and $\eta_s\perp \Leb^n$, for every Lipschitz curve $\gamma$ we have that any $\lambda\ll \Haus^1\trace\im(\gamma)$ and $\eta_a$ are mutually singular. For the rest of the proof we will also consider $\delta$ to be fixed. 

\smallskip

\textbf{Step 1.} Let $u(v,x)$ be a non-negative Borel function on $B(0,\delta)\times \im(\gamma)$. Then, for any Borel set $A$, the map $$v\mapsto \int_Au(v,v^{-1}*x)d\Haus^1\trace v*\im(\gamma)(x)=:G_{A,u}(v)\qquad\text{is Borel.} $$
First of all, we prove that given a fixed Borel set $A$ for any Borel set $B$ we have that the map $v\mapsto \Haus^1(v*\im(\gamma)\cap v*A\cap B)=\Haus^1\trace\im(\gamma)\cap A(v^{-1}*B)$ is Borel. This however is an immediate consequence of the fact that the map $v\mapsto \Haus^1\trace v*(\im(\gamma)\cap A)$ is continuous in the space of Radon measures endowed with the weak* topology, see Definition \ref{s-measint}. 

Let us assume now that $u$ is of the form $u(v,x)=a\mathbb{1}_{B_1}(v)\mathbb{1}_{B_2}(x)$ where $a>0$ and $B_1$ and $B_2$ are Borel subsets of $B(0,\delta)$ and $\GG$ respectively and that $A$ is some fixed Borel set. Then
$$G_{B,u}(v)=a\mathbb{1}_{B_1}(v)\int_A\mathbb{1}_{B_2}(v^{-1}*x)d\Haus^1\trace v*\im(\gamma)(x)=a\mathbb{1}_{B_1}(v)\Haus^1(v*\im(\gamma)\cap v*B_2\cap A).$$
This shows by the discussion above that in this case $G_{A,u}$ is Borel measurable.

If now $u$ is supposed to be a general non-negative measurable function, we can find an increasing sequence of simple function $g_i$ such that for every $i\in\N$ there is an $N_{i}\in\N$ and $a_j>0$ Borel subsets $B_{1,j}$ and $B_{2,j}$ of $B(0,\delta)$ and $\im(\gamma)$ respectively for $j=1,\ldots,N$ such that 
$$g_i(v,x)=\sum_{j=1}^{N_i} a_j\mathbb{1}_{B_{1,j}}(v)\mathbb{1}_{B_{2,j}}(x),$$
and the functions $g_i$ pointwise converge to $u$. Let further $A$ be a fixed Borel set in $\GG$. It is immediate to see that the functions $G_{A,g_i}$ can be rewritten as follows
\begin{equation}
\begin{split}
     &\qquad\qquad\qquad\qquad G_{A,g_i}(v)=\int_A \sum_{j=1}^Na_j\mathbb{1}_{B_{1,j}}(v)\mathbb{1}_{B_{2,j}}(v^{-1}*x)d\Haus^1\trace v*\im(\gamma)(x)\\
    =&\int \sum_{j=1}^Na_j\mathbb{1}_{B_{1,j}}(v)\mathbb{1}_{B_{2,j}}(x)\mathbb{1}_{A}(v*x)d\Haus^1\trace \im(\gamma)(x)=\int\mathbb{1}_{A}(v*x) g_i(v,x)d\Haus^1\trace \im(\gamma)(x).
    \nonumber
\end{split}
\end{equation}
Since the sequence of the $g_i$'s is increasing, by the monotone convergence theorem this shows that for any $v\in B(0,\delta)$ we have the pointwise convergence $G_{A,u}(v)=\lim_{i\to\infty}G_{A,g_i}(v)$. However, thanks to the above discussion and the linearity of the integral and the previous paragraph, we know that the functions $G_{A,g_i}$ are Borel and since pointwise countable limit of Borel functions is Borel, we conclude the proof of the claim that $G_{A,u}$ itself is Borel.

\smallskip

\textbf{Step 2.} Suppose $u(v,x):B(0,\delta)\times \im(\gamma)\to \R$ is a non-negative Borel function in $L^1(\Leb^n\trace B(0,\delta)\otimes \Haus^1\trace\im(\gamma))$. We claim that the Radon measure $\mu_{\gamma,u}$ defined by
\begin{equation}
    \mu_{\gamma,u}(A):=\int_{B(0,\delta)}\int_A u(v,v^{-1}*x) d\Haus^1 \trace v*\im(\gamma)(x) dv \qquad\text{for any Borel set }A,
    \label{intabscontcurv}
\end{equation}
is absolutely continuous with respect to the Lebesgue measure. Note that the measure $\mu_{\gamma,u}$ is well defined thanks to Step 1.
Let us fix a Borel set $A\subseteq \mathbb{G}$ and note that since the curve $v*\gamma$ has the same Lipschitz constant of $\gamma$, we infer that 
\begin{equation}
    \begin{split}
        \int_A u(v,v^{-1}*x)d\Haus^1\trace v*\im(\gamma)(x)=&\int_K u(v,\gamma(t))\mathbb{1}_A(v*\gamma(t))\lvert D\gamma(t)\rvert d\Leb^1(t)\\
        \leq &\mathrm{Lip}(\gamma)\int_Ku(v,\gamma(t))\mathbb{1}_A(v*\gamma(t))d\Leb^1(t),
        \label{eq.gamma.const}
    \end{split}
\end{equation}
where above we used the area formula \cite[Theorem 4.4]{magnaniarea} and where we recall that $D\gamma(t)$ above denotes the Pansu derivative of $\gamma$ at $t$ and that $D(v*\gamma(t))=D(\gamma(t))$.
Thanks to \eqref{eq.gamma.const} we infer that
\begin{equation}
    \begin{split}
        \int_{B(0,\delta)}\int_Ku(v,\gamma(t))&\mathbb{1}_A(v*\gamma(t))d\Leb^1(t)d\Leb^n(v)\\
        \geq&\mathrm{Lip}(\gamma)^{-1}\int_{B(0,\delta)}\int_A u(v,v^{-1}*x)d\Haus^1\trace v*\im(\gamma)(x)d\Leb^n(v)
        =\mathrm{Lip}(\gamma)^{-1}\mu_{\gamma,u}(A).
        \label{eq:first fubini}
    \end{split}
\end{equation}
On the other hand, since we are dealing with positive measurable functions, Tonelli's theorem implies
\begin{equation}
\begin{split}
    \mathrm{Lip}(\gamma)^{-1}\mu_{\gamma,u}(A)\overset{\eqref{eq:first fubini}}{\leq} \int_{B(0,\delta)}\int_Ku(v,\gamma(t))&\mathbb{1}_A(v*\gamma(t))d\Leb^1(t)d\Leb^n(v)\\ =\int_K\int_{B(0,\delta)}u(v,\gamma(t))\mathbb{1}_A(v*\gamma(t))d\Leb^n(v)d\Leb^1(t)
    =&\int_K\int_{B(0,\delta)*\gamma(t)}u(w*\gamma(t)^{-1},\gamma(t))\mathbb{1}_A(w)d\Leb^n(w)d\Leb^1(t),
\end{split}
\end{equation}
where the last identity above comes from the fact that Carnot groups are unimodular. Summing up, we have shown that
$$\mu_{\gamma,u}(A)\leq \mathrm{Lip}(\gamma)\int_K\int_{B(0,\delta)*\gamma(t)\cap A}u(w*\gamma(t)^{-1},\gamma(t))d\Leb^n(w)d\Leb^1(t).$$
From the above inequality is immediate to see that if $A$ is Lebesgue-null then it is $\mu_{\gamma,u}$-null and this concludes the proof of the claim. With the choice $u\equiv 1$, this implies in particular that
\begin{itemize}
    \item[(\hypertarget{I}{I})] for any $w$ and for $\Leb^n$-almost every $v\in B(0,1)$ we have     $\Haus^1(v*\im(\gamma)\cap w*\im(\gamma))=0$.
\end{itemize}

\smallskip

\textbf{Step 3.} Thanks to Radon-Nikodym's decomposition theorem we can write
\begin{equation}
    \eta\trace v*\im(\gamma)=\rho_v\Haus^1\trace v*\im(\gamma)+\eta_v,
    \label{vRN}
\end{equation}
for any $v\in B(0,\delta)$, where $\rho_v\in L^1(\Haus^1\trace v*\im(\gamma))$ and $\eta_v$ is supported on $v*\im(\gamma)$ and singular with respect to $\Haus^1\trace v*\im(\gamma)$. Note that since $\eta$ is finite, then for any $v\in B(0,\delta)$ we have
\begin{equation}
    \eta(\GG)\geq \eta( v*\im(\gamma))\geq \int \rho_v(y)d\Haus^1\trace v*\im(\gamma)(y).
    \label{eq:bdneta}
\end{equation}
We now show that the function $(x,v)\mapsto \rho_v(v*x)$ belongs to $L^1(\Leb^n\otimes\Haus^1\trace \im(\gamma))$.
Thanks to the Euclidean Besicovitch's differentiation theorem we know that for every $v\in B(0,\delta)$ and $\Haus^1\trace \im(\gamma)$-almost every $x\in \GG$ we have
$$\rho_v(v*x)=\lim_{r\to 0}\frac{\eta(U(v*x,r))}{\Haus^1\trace v*\im(\gamma)(U(v*x,r))}=\lim_{\substack{i\to \infty\\i\in\N}}\frac{\eta(U(v*x,i^{-1}))}{\Haus^1\trace v*\im(\gamma)(U(v*x,i^{-1}))}=:\tilde{\rho}(v,x),$$
we remark that $U(y,s)$ above denote the Euclidean ball of centre $y$ and radius $s$.

For any $i\in\N$ the functions $(x,v)\mapsto \eta(U(v*x,i^{-1}))$ and $(x,v)\mapsto \Haus^1\trace v*\im(\gamma)(U(x,i^{-1}))$ are upper semicontinuous \emph{in both variables} $(v,x)$ thanks to Fatou's lemma. Therefore, the ratio $\eta(U(x,i^{-1}))/\Haus^1\trace v*\im(\gamma)(U(x,i^{-1}))$ is Borel measurable jointly in the variables $(x,v)$. Since the function $\tilde{\rho}(v,x)$ is the pointwise limit of countable many Borel functions, this shows that for any $v\in B(0,\delta)$ and for $\Haus^1\trace \im(\gamma)$-almost every $x\in \GG$, the function $\rho_v(v*x)$ coincides with the Borel function $\tilde{\rho}(v,x)$ for $\Leb^n\trace B(0,\delta)\otimes \Haus^1\trace\im(\gamma)$-almost every $(v,x)\in B(0,\delta)\times \im(\gamma)$. 

Let us prove that $ \rho_v(v*x)$ is in $L^1(\Leb^n\otimes\Haus^1\trace \im(\gamma))$. In order to see this, just note that
\begin{equation}
    \begin{split}
        &\int \rho_v(v*x)d\Leb^n\trace B(0,\delta)\otimes\Haus^1\trace \im(\gamma)(v,x)=\int \tilde{\rho}(v,x)d\Leb^n\trace B(0,\delta)\otimes\Haus^1\trace \im(\gamma)(v,x)\leq \Leb^n(B(0,\delta))\eta(\GG),
        \nonumber
    \end{split}
\end{equation}
where the last inequality comes from Tonelli's theorem and \eqref{eq:bdneta}.

Let $B$ be a Borel set and note that putting together all the pieces of information collected up to now, we infer that for $\Leb^n$-almost every $v\in B(0,\delta)$ we have
$$\nu_v(B):=\int_B\rho_v(y)d\Haus^1\trace v*\im(\gamma)(y)=\int_B\rho_v(v*x)d\Haus^1\trace \im(\gamma)(x)=\int_B\tilde{\rho}(v,x)d\Haus^1\trace \im(\gamma)(x).$$
The above identity thus proves that the map $v\mapsto \nu_v(B)$ coincides with a Borel map, see Step 1, on a set of full $\Leb^n\trace B(0,\delta)$-measure.

\smallskip

\textbf{Step 4.} An immediate consequence of the measurability of the function $v\mapsto \nu_v(\GG)$ is that for any $\varepsilon>0$ the set $F_\varepsilon$ of those $v\in B(0,\delta)$ for which $\nu_v(\GG)\geq \varepsilon$ is $\Leb^n$ measurable. Let us assume by contradiction that
\begin{itemize}
    \item[(\hypertarget{A}{A})] \textbf{there is $\varepsilon>0$ and a Borel set $F\subseteq F_\varepsilon$ of positive
   $\Leb^n$-measure.}
\end{itemize}
Since the map $v\mapsto \nu_v(B)$ is $\Leb^n$-measurable for any Borel set $B$, it is well defined the measure that acts as
  $$\daleth(B):=\int_{F}  \nu_v(B)d\Leb^n(v)\qquad \text{ on every Borel set $B$}.$$
It is immediate to see that $\daleth$ is finite and that $\daleth\ll\eta$. Therefore, thanks to Radon-Nykodym theorem, there exists a $\rho\in L^1(\eta)$ such that $\daleth=\rho\eta$. However, since it is immediately seen that $\daleth\leq \mu_{\gamma,\tilde{\rho}}$, by Step 2 we conclude that
$$\rho\eta\leq \mu_{\gamma,\tilde{\rho}}\ll \Leb^n,$$
which is in contradiction with the assumption that $\eta$ was mutually singular with respect to $\Leb^n$. This shows in particular that $\nu_v(\mathbb{G})=0$ for almost every $v\in B(0,\delta)$ and hence $\Haus^1\trace v*\im(\gamma)$ is singular with respect to $\eta$ for $\Leb^n$-almost every $v\in B(0,\delta)$.
\end{proof}

\begin{proof}[Proof of Proposition \ref{p:normal}]
We divide the proof in several steps. Throughout the proof we fix $0<\varepsilon<\Mass[\int_I\boldsymbol\mu_tdt]/10$. Without loss of generality, we can assume that the $\boldsymbol\mu_t$'s are supported on the closed ball $B(0,R)$ for some $R>0$ and that $I$ is $\R$ and $dt$ is the Lebesgue measure, see \cite[Remark 2.7 (iii)]{AlbMar}.  Thanks to the assumption that the masses of the $\boldsymbol{\mu}_t$ are summable, i.e. $\int_\R\Mass(\boldsymbol{\mu}_t)dt<\infty$, for every $\varepsilon>0$ there exists a compact interval $\tilde{I}$ such that $\int_{\R\setminus \tilde{I}}\Mass(\boldsymbol\mu_t)dt<\varepsilon/3$.

\medskip

{\textbf{Step 1}}
 Since the family of measures $\{\boldsymbol\mu_t\}_{t\in\tilde{I}}$  satisfies the hypothesis (a) and (b) of Definition \ref{s-measint} and hence Proposition \ref{p:borel_selection} and \cite[Remark 2.7 (iii)]{AlbMar} imply that there exists a compact interval $\tilde{I}$ and
a Borel map $\eta_\varepsilon:\tilde{I}\to \mathscr{M}(\R^n,\R^n)$ such that 
  \begin{itemize}
    \item for almost every $ t\in \tilde{I}$ we have $\eta_\varepsilon(\tilde{I})=\llbracket \Gamma( t)\rrbracket$ where the map $\Gamma:\tilde{I}\to \mathcal{X}$ is a Borel map with respect to the metric $d_{\mathfrak{s}\Haus}$ introduced in Definition \ref{stronghausmetric};
    \item $\Mass[\int_{\bar{I}} \boldsymbol{\mu}_{\bar t}d\bar{t}-\int_{\bar{I}} \eta_\varepsilon(\bar t)d\bar t]<\varepsilon/3$.
\end{itemize}
%  and thanks to the assumption that the masses of the $\boldsymbol{\mu}_t$ are summable, i.e. $\int_\R\Mass(\boldsymbol{\mu}_t)dt<\infty$, for every $\varepsilon>0$ there exists a compact interval $\tilde{I}$ and a map $\eta_\varepsilon:\tilde{I}\to \mathcal{X}$, where $\mathcal{X}$ was introduced in Definition \ref{distanceoncurvesmaps}, such that denoting $\bar{\boldsymbol{\mu}}_t:=\llbracket \eta_{\varepsilon}(t)\rrbracket$ we obtain
% a family of vector-valued measures $\{\bar{\boldsymbol{\mu}}_t\}_{t\in \tilde{I}}$ for which the map $t\mapsto \bar{\boldsymbol{\mu}}_t$ satisfies the hypothesis (a) and (b) of Definition \ref{s-measint} and for which  denoting
% \begin{equation}\label{e_mutilde}
% \tilde {\boldsymbol{\mu}}:=\int_{\tilde{I}} \bar{\boldsymbol{\mu}}_t\,dt,    
% \end{equation}
% it holds
% $\mathbb{M}(\boldsymbol\mu-\tilde{\boldsymbol{\mu}})<\varepsilon$. 
Thanks to Lusin's theorem and by the absolute continuity of the integral, guaranteed by property (b) of Definition \ref{s-measint},
we can find a closed subset $J\subset \bar{I}$ and an $N\in\N$ such that $\Leb^1(\bar{I}\setminus J)\leq \varepsilon/6$ and
\begin{itemize}
    \item the maps $\eta_\varepsilon:J\to \mathscr{M}(\R^n,\R^n)$, $\Gamma:J\to \mathcal{X}$ and $t\mapsto \Mass[\Gamma(t)]$ are continuous with respect to the metric $d$ inducting the weak* topology on $\mathscr{M}(\R^n,\R^n)$ and  $d_{\mathfrak{s}\Haus}$ respectively, see Definition \ref{distanceoncurvesmaps} and Remark \ref{rk:equivhaus}. 
    \item $\Gamma(J)\subseteq \mathcal{X}_N\cap \{\gamma\in \mathcal{X}:\dom(\gamma)\subseteq [0,N]\}$.
    \item denoting $\bar{\boldsymbol{\mu}}_t:=\eta_\varepsilon(t)$ and 
\begin{equation}\label{e_mubar}
\bar{\boldsymbol{\mu}}:=\int_J \bar{\boldsymbol{\mu}}_t\,dt,    
\end{equation}
the vector-valued measure $\boldsymbol\mu-\bar{\boldsymbol{\mu}}$ is an integral of (horizontal) Lipschitz curves in the sense of \eqref{e:measure_structure} and
it holds
\begin{equation}\label{e:massdifference}
\mathbb{M}(\boldsymbol\mu-\bar{\boldsymbol{\mu}})<5\varepsilon/6.
\end{equation}
\end{itemize}

{\textbf{Step 2}}
Throughout the rest of the proof, we let $\bar\gamma_t:=\Gamma(t)$ and we claim that  there exits $\delta=\delta(\varepsilon)$ such that for any $t\in J$ and any Borel set $A\subset J\cap(t-\delta,t+\delta)$ there exists $t_0\in A$ such that denoting $\mathbf{N}_A:=|A|\llbracket\bar\gamma_{t_0}\rrbracket$
and
$\mathbf{T}_A:=\int_A\llbracket\bar\gamma_s\rrbracket ds$
the following holds:
\begin{equation}\label{e:mass}
\Mass(\mathbf{N}_A)\leq\Mass(\mathbf{T}_A)
\end{equation}
and moreover
\begin{equation}\label{flat2}
\mathbf{N}_A-\mathbf{T}_A=\mathbf{R}+\partial \mathbf{S}, \qquad 
\mbox{with } \mathbb{M}(\mathbf{R})+\mathbb{M}(\mathbf{S})<\varepsilon|A|
\end{equation}
where $\mathbf{R}$ is an integration of Lipschitz curves defined on compact intervals and with values in $\GG$ and $\mathbf{S}$ is a $2$-dimensional (Euclidean) normal current.

\smallskip

In order to prove that \eqref{e:mass} and \eqref{flat2} hold, it suffices to show that, for every $0<d<1$ and for any $s,t\in J$ such that $d_{\mathfrak{s}\Haus}(\bar\gamma_s,\bar\gamma_t)\leq d$, we have
\begin{equation}\label{e:basic_flat}
\llbracket\bar\gamma_t\rrbracket-\llbracket
\bar\gamma_s\rrbracket=\mathbf{R}_1+\partial \mathbf{S}_1,\qquad 
\mbox{with } \mathbb{M}(\mathbf{R}_1) + \mathbb{M}(\mathbf{S}_1) \leq C(N)d^{\frac{1}{\mathfrak{s}}}    \end{equation}
where $\mathbf{R}_1$ is a finite sum of Lipschitz curves defined on compact intervals and with values in $\GG$, $\mathbf{S}$ is a $2$-dimensional (Euclidean) integral current and $C(N)$ is a constant depending only on $N$. 

Let us prove that \eqref{e:basic_flat} implies \eqref{e:mass} and \eqref{flat2}.
Thanks to the uniform continuity on $J$ of the map $t\mapsto\bar\gamma_t$ there is a $\delta=\delta(\varepsilon)>0$ such that for every $s,t_0\in J\cap (t-\delta,t+\delta) $ it holds $d_{\Haus}(\bar\gamma_s,\bar\gamma_{t_0})\leq (C(N)^{-1}\varepsilon)^\mathfrak{s}$. Therefore, for any $t\in J$ and any Borel subset $A$ of $J\cap (t-\delta,t+\delta)$ \eqref{e:mass} holds whenever $t_0\in A$ is such that $$\Mass (\llbracket\bar\gamma_{t_0}\rrbracket)\leq \fint_A\Mass(\llbracket\bar\gamma_{s}\rrbracket)ds.$$
On the other hand, \eqref{flat2} follows immediately integrating \eqref{e:basic_flat} on $A$.

\smallskip

Let us move to the proof of \eqref{e:basic_flat}. Observe that for every $z\in K_{s,t}:=\dom(\bar\gamma_s)\cap \dom(\bar\gamma_t)$ it holds 
\begin{equation}\label{e:distanza}
|\bar\gamma_s(z)-\bar\gamma_t(z)|\leq 2d,
\end{equation}
and consider the map 
$g:K_{s,t}\times[0,d]\to\GG$ given by
$$g(\sigma,\tau)=\left(1-\frac{\tau}{d}\right)\bar\gamma_s(\sigma)+\frac{\tau}{d}\bar{\gamma}_t(\sigma).$$
Let us observe that thanks to our choice of $J$ and since $s,t\in J$ we have that  $K_{s,t}$ has the form 
$K_{s,t}=\bigcup_{i=1}^{M}[a_i,b_i]$,
for some number $M\leq N^2$. Moreover it is easy to check that $g$ is $3$-Lipschitz, because it is $1$-Lipschitz in the variable $\sigma$, being a convex combination of $1$-Lipschitz maps, and it is $2$-Lipschitz in the variable $\tau$, due to \eqref{e:distanza}. 
Let us further denote
\begin{equation}
    \begin{split}
        &\qquad\qquad\qquad\qquad\qquad \mathbf{S}_0:=g_{\sharp}\llbracket K_{s,t}\times[0,d]\rrbracket:=g_\#\int e_1\wedge e_2 d\Leb^2\trace K_{s,t}\times [0,d],\\
        &\mathbf{R}_0:=\sum_{i=1}^{M}g(b_i,\cdot)_\sharp\llbracket[0,d]\rrbracket-\sum_{i=1}^{M}g(a_i,\cdot)_\sharp\llbracket[0,d]\rrbracket + \llbracket\bar\gamma_t\trace (\dom(\bar{\gamma}_t)\setminus K_{s,t})\rrbracket-\llbracket\bar\gamma_s\trace (\dom(\bar{\gamma}_s)\setminus K_{s,t})\rrbracket,
    \end{split}
\end{equation}
where $g_\#$ denotes the pushforward of currents through the maps $g$, see \cite[\S 7.4.2]{MR2427002}. The first identity follows from the fact that pushforward and boundary are two commuting operators, i.e.  $\partial g_\#T=g_\#\partial T$ for any $2$-dimensional  normal current $T$.
Note that since $\llbracket K_{s,t}\times [0,d]\rrbracket$ is a $2$-dimensional normal current in $\R^2$, the $2$-current $S_0$ is well defined thanks to \cite[Proposition 5.17]{AlbMar}.
Notice also that by definition of $\mathbf S_0$ and $\mathbf R_0$ we have  \begin{equation}\label{flatconzero}
    \llbracket\bar\gamma_t\rrbracket-\llbracket
\bar\gamma_s\rrbracket=\mathbf{R}_0+\partial \mathbf{S}_0.
\end{equation}
The proof of \eqref{e:basic_flat} is not still completed since $\mathbf{R}_0$ is not a sum of \emph{horizontal} Lipschitz curves.

Therefore, for every $i=1,\ldots, M$ let us denote $\psi_{a_i}$ a geodesic joining $g(a_i,0)$ to $g(a_i,d)$ and $\psi_{b_i}$ a geodesic joining $g(b_i,0)$ to $g(b_i,d)$ and define $\mathbf{R}^i:=\llbracket\psi_{a_i}\rrbracket-\llbracket\psi_{b_i}\rrbracket$. Since we are working inside the fixed compact set $B(0,R)$, there exists a constant $C(R,\GG)$ such that:
\begin{equation}\label{massaRi}
\Mass(\mathbf{R}^i)\leq C(R,\GG) d^{\frac{1}{\mathfrak{s}}},
\end{equation}
and this is a consequence of Remark \ref{rk.norm}.
Let $S_{a_i}$ and $S_{b_i}$ be 2-dimensional currents with boundary $g(a_i,\cdot)_\sharp\llbracket[0,d]\rrbracket-\llbracket\psi_{a_i}\rrbracket$ and $g(b_i,\cdot)_\sharp\llbracket[0,d]\rrbracket-\llbracket\psi_{b_i}\rrbracket$ respectively and
\begin{equation}
\Mass(\mathbf{S}_{a_i})+\Mass(\mathbf{S}_{b_i})\leq Cd^{\frac{1}{\mathfrak{s}}}.
    \label{massaSi}
\end{equation}

\begin{figure}[H]
\centering
   \adjustbox{trim={.05\width} {.175\height} {0.05\width} {.12\height},clip}{\includegraphics[scale=0.41, 
    ]{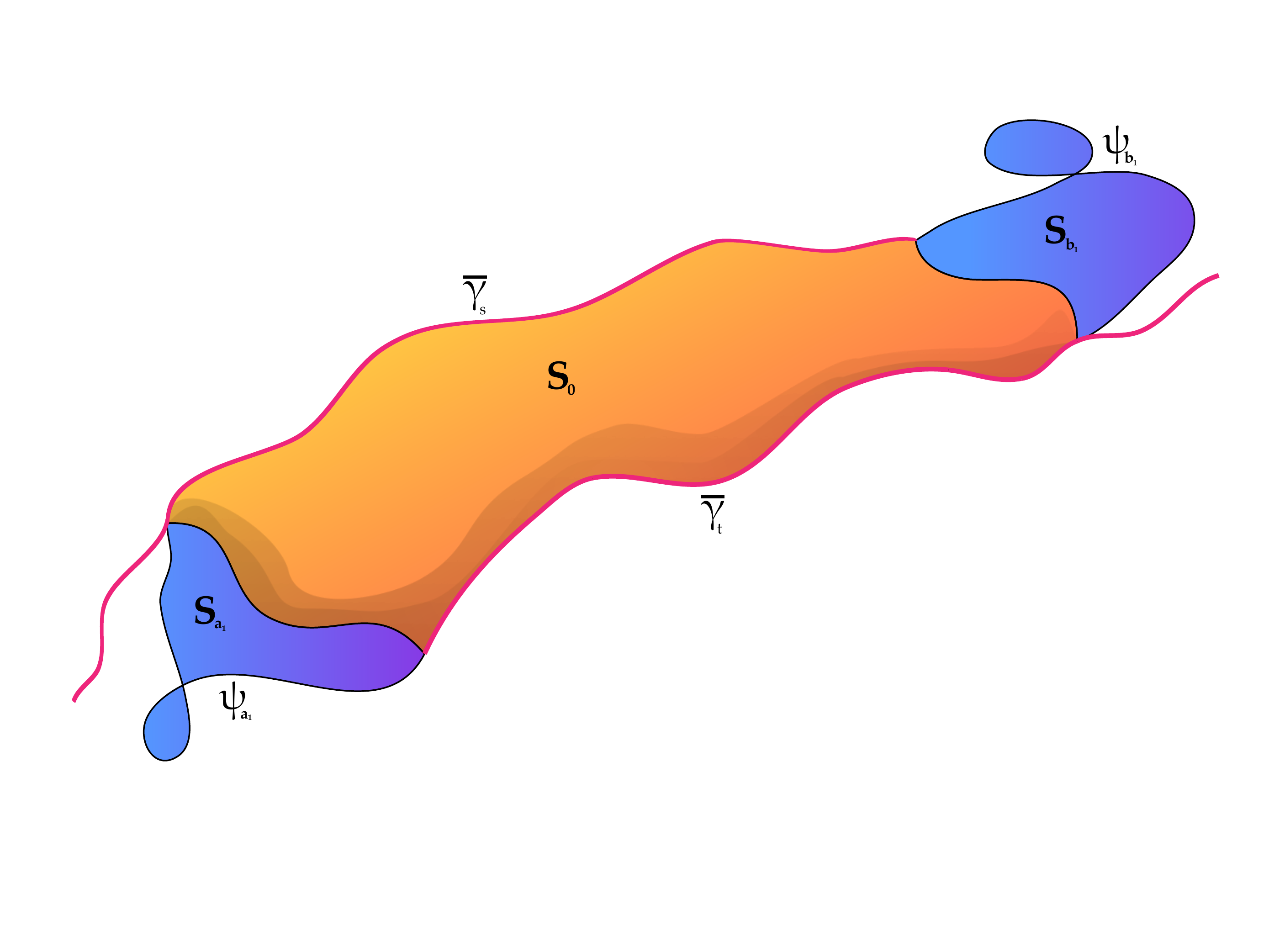}} 
\caption{The image shows the filling $2$-dimensional surfaces with horizontal boundaries attaching the curves $\bar\gamma_s$ to $\bar\gamma_t$.}
\end{figure}

The choice of such $\mathbf{S}_i$, with the above control on their mass, is achievable thanks to the classical cone construction, see for instance \cite[(26.26)]{Simon1983LecturesTheory}.  We further define $\mathbf{S}^i:=\mathbf{S}_{a_i}-\mathbf{S}_{b_i}$ and 
\begin{equation}\label{e:defR1S1}
    \mathbf{S}_1:=\mathbf{S}_0+\sum_{i=1}^{M}\mathbf{S}^i\qquad \text{and}\qquad  \mathbf{R}_1:=\mathbf{R}_0+\sum_{i=1}^{M}\mathbf{R}^i.
\end{equation}
It follows from \eqref{flatconzero} that $\llbracket\bar\gamma_t\rrbracket-\llbracket
\bar\gamma_s\rrbracket=\mathbf{R}_1+\partial \mathbf{S}_1$
and by construction $\mathbf{R}_1$ is a finite sum of Lipschitz curves defined on compact intervals and with values in $\GG$.
Thanks to the proof of \cite[Proposition 5.17]{AlbMar} we have that $$\Mass(\mathbf{S}_0)\leq \mathrm{Lip}(g)^2\Mass\llbracket K_{s,t}\times[0,d]\rrbracket\leq 9Nd,$$
and thus $\Mass(\mathbf{S}_1)\leq C(N)d^{\frac{1}{\mathfrak s}}$.
Moreover 
$$\Mass(\mathbf{R}_1)\leq\Mass(\llbracket\bar\gamma_t\trace (\dom(\bar{\gamma}_t)\setminus K_{s,t}))\rrbracket+\Mass(\llbracket\bar\gamma_t\trace (\dom(\bar{\gamma}_s)\setminus K_{s,t}))\rrbracket+\sum_{i=1}^{M}\Mass(\mathbf{R}_{a_i})+\sum_{i=1}^{M}\Mass(\mathbf{R}_{b_i})\leq C(N)d^{\frac{1}{\mathfrak s}},$$
where we used \eqref{massaRi}, \eqref{massaSi} and the fact that $\bar\gamma_s$ and $\gamma_t$ are 1-Lipschitz and $(\dom(\bar{\gamma}_t)\setminus K_{s,t})$ and $(\dom(\bar{\gamma}_s)\setminus K_{s,t})$ have length bounded by $Cd$ due to the fact that we have chosen $d_{\mathfrak{s}\Haus}(\bar\gamma_s,\bar\gamma_t)\leq d$. This concludes the proof of \eqref{e:basic_flat} and hence of the Step 2.

\medskip

{\textbf{Step 3}}
Let us take $\delta$ as in Step 2 and finitely many disjoint Borel sets $A_j$ of diameter less than $\delta$ such that $\bigcup_j A_j=J$. Denote $\mathbf N_{A_j}$ the corresponding currents given in Step 2 and $\mathbf{T}_1:=\sum_j \mathbf N_{A_j}$. Notice that by \eqref{e:massdifference} and \eqref{e:mass} it holds
\begin{equation}\label{e:massdifference2}
\Mass(\mathbf{T}_1)\leq\Mass(\boldsymbol\mu)+2\varepsilon  
\end{equation}
and by \eqref{flat2} 
\begin{equation}\label{flat3}
\boldsymbol\mu-\mathbf{T}_1=\mathbf{R}_2+\partial \mathbf{S}_2, \qquad 
\mbox{with } \mathbb{M}( \mathbf R_2)+\mathbb{M}(\mathbf S_2)<C\varepsilon,
\end{equation}
where $\mathbf R_2$ is an integration of Lipschitz fragments as in \eqref{e:measure_structure} and $ 
\mathbf S_2$ is a $2$-dimensional (Euclidean) normal current.

 We claim that there exists a normal 1-current $\mathbf Z$, which is a finite sum of Lipschitz curves defined on compact intervals and with values in $\GG$, such that $\|\mathbf Z\|$ and $\eta$ are mutually singular, 
 \begin{equation}\label{masszeta}
   \mathbb{M}(\mathbf Z)\leq\mathbb{M}(\mathbf T_1)  
 \end{equation}
 and moreover 
 \begin{equation}\label{flatzeta}
   \mathbf T_1-\mathbf Z=\mathbf R_3+\partial \mathbf S_3, \qquad \mbox{with } \mathbb{M}(\mathbf R_3)+\mathbb{M}(\mathbf S_3)<C\varepsilon,
 \end{equation}
 and $\mathbf R_3$ is a finite sum of Lipschitz curves defined on compact intervals and with values in $\GG$.

\smallskip

Since $\mathbf T_1$ is a finite sum of Lipschitz curves defined on compact intervals, it is sufficient to prove the claimed conclusion for $\mathbf T_1=\llbracket\gamma\rrbracket$ being a single curve with $\dom(\gamma)=[0,1]$. To this purpose we apply Proposition \ref{prop:singcurves} with $\delta=\epsilon^{\mathfrak{s}}$ and we take $Z$ to be the current $\mathbf Z:=\llbracket v*\gamma\rrbracket$ for any $v\in B(0,\delta)$ as in the conclusion of the proposition.
Defining $\mathbf R_3$ and $\mathbf S_3$ as in the previous step (see \eqref{e:defR1S1}) we get \eqref{flatzeta} with the same argument employed above.

Putting together \eqref{e:massdifference2}, \eqref{flat3}, \eqref{masszeta}, \eqref{flatzeta} we get currents $\bar {\mathbf R}:=\mathbf R_2+\mathbf R_3$ and $\bar{\mathbf S}:=\mathbf S_2+\mathbf S_3$ such that:
 \begin{equation}\label{massmu}
   \mathbb{M}(\mathbf Z)\leq\mathbb{M}(\boldsymbol\mu)+2\varepsilon,  
 \end{equation}
 and
 \begin{equation}\label{flatmu}
   \boldsymbol\mu-\mathbf Z=\bar {\mathbf R}+\partial \bar{\mathbf S}, \qquad \mbox{with } \mathbb{M}(\bar{\mathbf R})+\mathbb{M}(\bar{\mathbf S})<C\varepsilon,
 \end{equation}
where $\bar {\mathbf R}$ is an integration of Lipschitz fragments as in \eqref{e:measure_structure} and $\bar{\mathbf S}$ is a $2$-dimensional (Euclidean) normal current.

\medskip

{\textbf{Step 4}}
We obtain the current $\mathbf{T}$ by iterating on the previous steps, as follows.

Applying Step 1, 2 and 3 we obtain currents $\mathbf Z_1$, $\bar{ \mathbf R}_1$, $\bar{\mathbf{S}}_1$ such that \eqref{massmu} and \eqref{flatmu} hold with $\varepsilon=\varepsilon_0/4$. In particular, $Z_1$ and $\eta$ are mutually singular. 

By construction of $\mathbf R_1$, we can apply step 1,2 and 3 to $\boldsymbol\mu_1:=\mathbf R_1$ with $\varepsilon:=\varepsilon_0/8$ thus obtaining currents $\mathbf Z_2$, $\bar{\mathbf{R}}_2$, $\bar{\mathbf{S}}_2$ such that:
 \begin{equation}\label{massmu2}
   \mathbb{M}(Z_2)\leq\mathbb{M}(\boldsymbol\mu_1)+\varepsilon_0/4\overset{\eqref{massmu}}{\leq} \Mass(\boldsymbol\mu)+3\epsilon_0/4,  
 \end{equation}
 and
 \begin{equation}\label{flatmu2}
  (\boldsymbol\mu-\mathbf Z_1-\partial \bar{\mathbf{S}}_1)-\mathbf Z_2 =\boldsymbol\mu_1-\mathbf Z_2=\bar{\mathbf{R}}_2+\partial \bar{\mathbf{S}}_2, \qquad \mbox{with } \mathbb{M}(\bar{\mathbf{R}}_2)+\mathbb{M}(\bar{\mathbf{S}}_2)<C\varepsilon_0/8,
 \end{equation}
One can rewrite \eqref{flatmu2} in the following more appealing way:
 \begin{equation}\nonumber
  \boldsymbol\mu-(\mathbf Z_1+\mathbf Z_2)=\bar{\mathbf{R}}_2+\partial (\bar{\mathbf{S}}_1+\bar{\mathbf{S}}_2), \qquad \mbox{with } \mathbb{M}(\bar{\mathbf{R}}_2)<C\varepsilon_0/8\text{ and }\mathbb{M}(\bar{\mathbf{S}}_1+\bar{\mathbf{S}}_2)<3C\varepsilon_0/8,
 \end{equation}
and where $\bar{\mathbf{R}}_2$ is still an integration of horizontal fragments. Note also that $\mathbf Z_1+\mathbf Z_2$ and $\eta$ are mutually singular.

Iterating the procedure, since $(\Mass(\bar {\mathbf S}_i))_i$ and $(\Mass(\mathbf Z_i))_i$ are summable sequences and $\Mass(\bar{\mathbf R}_i)$ is infinitesimal, we have that the currents
$$\boldsymbol\sigma=:-\sum_{i=1}^\infty \mathbf Z_i\qquad \text{and}\qquad  \mathbf S:=\sum_{i=1}^\infty \mathbf S_i,$$
are well defined. We observe that $\boldsymbol\sigma$ and $\eta$ are mutually singular, that
$$\Mass( \boldsymbol \sigma)\leq \Mass(\boldsymbol\mu)+\varepsilon_0,$$
that $\sigma$ is an integration of horizontal fragments and that defined
$$\mathbf{T}:=\boldsymbol\mu+\boldsymbol \sigma=\mathbf S,$$
we have that $\partial \mathbf{T}=0$ and $\Mass(\mathbf{T})\leq 2\Mass(\boldsymbol\mu)+\varepsilon_0$.
\end{proof}

\section{Auxiliary decomposability bundle}

In this section we relate the decomposability bundle to the existence of suitable horizontal normal currents, which is the key step in the proof of the main theorem; compare with \cite[Section 6]{AlbMar}.

\begin{definizione}[Auxiliary decomposability bundle]\label{def:auxdecbundle}
Let $\mu$ be a Radon measure on $\GG$. For any $x\in\supp(\mu)$ we denote with $N(\mu,x)$ the set of all vectors of $v\in \mathscr{C}(x)[V_1]$ for which there exists a horizontal $1$-dimensional normal current $\mathbf{T}$ with $\partial \mathbf{T}=0$ such that:
$$\limsup_{r\to 0}\frac{\Mass( (\mathbf{T}-v\mu)\trace U(x,r))}{\mu(U(x,r))}=0,$$
where we recall that $U(x,r)$ denotes the closed Euclidean ball centred at $x$ of radius $r$.
\end{definizione}

\begin{lemma}\label{Nemisurabile}
For any Radon measure $\mu$, the map $x\mapsto N(\mu,x)$ is universally measurable.
\end{lemma}

\begin{proof}
The proof can be achieved following the argument used to prove \cite[Lemma 6.9]{AlbMar}. The only difference is that here the vector $v$ is forced to lie in the smooth distribution of $n_1$-dimensional planes $\mathscr{C}(x)[V_1]$.
\end{proof}

\begin{proposizione}\label{normaleimpliesinN}
For any Radon measure $\mu$ on $\GG$ and any $1$-dimensional horizontal normal current $\mathbf{T}$, if we denote by $\tau$ the Radon-Nikodym derivative of $\mathbf{T}$ with respect to $\mu$, we have
$$\tau(x)\in N(\mu,x)\qquad\text{ for any $\mu$-almost every $x\in \GG$.}$$
\end{proposizione}

    \begin{proof}
    Indeed, let $\mathbf{T}=\tau\mu+\boldsymbol\sigma$, where $\boldsymbol\sigma$ and $\mu$ are mutually singular. Then by Lebesgue-Besicovich differentiation theorem we have
    \begin{equation}
    \begin{split}
         \limsup_{r\to 0}\frac{\Mass[(\mathbf{T}-\tau(x)\mu)\trace U(x,r)]}{\mu(U(x,r))}\leq \limsup_{r\to 0}\fint_{U(x,r)}\lvert \tau(y)-\tau(x)\rvert d\mu(y)+\limsup_{r\to 0}\frac{\boldsymbol\sigma(U(x,r))}{\mu(U(x,r))}
        =0,
        \nonumber
    \end{split}
    \end{equation}
    for $\mu$-almost every $x\in \GG$, which in turn implies that $\tau(x)\in N(\mu,x)$.
     \end{proof}

\begin{lemma}\label{lemmaapproxtau}
Let $\mu$ be a Radon measure on $\GG$ and suppose $\tau$ is an $L^1(\mu)$ vector field such that $\tau(x)\in N(\mu,x)$ for $\mu$-almost every $x\in\GG$. Then, for any $\varepsilon_0>0$ there exists an horizontal normal current $\mathbf{T}$ on $\GG$ such that:
\begin{itemize}
    \item[(i)] $\lVert \tilde\tau-\tau\rVert_{L^1(\mu)}\leq \lVert\tau\rVert_{L^1(\mu)}/2$ where $\tilde\tau$ is the Radon-Nikodym derivative of $\mathbf{T}$ with respect to $\mu$;
    \item[(ii)] $\partial \mathbf{T}=0$ and $\Mass(\mathbf{T})\leq 2(1+2\varepsilon_0) \lVert \tau\rVert_{L^1(\mu)}$.
\end{itemize}
\end{lemma}
    
\begin{proof}Let us prove the statement when $\mu$ has compact support. 
If $\tau(x)=0$ for $\mu$-almost every $x\in\GG$ there is nothing to prove and hence we may assume that $\tau$ is non-trivial. Let $\varepsilon>0$ to be chosen later and thanks to Lebesgue's differentiation theorem, see for instance \cite[Corollary 2.9.9]{Federer1996GeometricTheory}, for $\mu$-almost every $x\in\GG$ there exists an $r_0(x)>0$ such that for any $0<s<r_0(x)$ we have
\begin{equation}
    \fint_{U(x,s)}\lvert \tau(y)-\tau(x)\rvert d\mu(y)\leq \varepsilon.
    \nonumber
\end{equation}
Therefore, thanks to Besicovitch's covering theorem and the fact that $\tau(x)\in N(\mu,x)$ for $\mu$-almost every $x\in\GG$,
we can find countably many closed and disjoint Euclidean balls $\{U(x_i,r_i)\}_{i\in\N}$ such that:
\begin{enumerate}
    \item $\mu(\mathbb{G}\setminus \bigcup_iU(x_i,r_i))=0$ and $\mu(\partial U(x_i,r_i))=0$,
    \item for any $i\in\N$ we have
    $\fint_{U(x_i,r_i)}\lvert \tau(y)-\tau(x_i)\rvert d\mu(y)\leq \varepsilon$ and we can find a $1$-dimensional horizontal normal current $\mathbf T_i=\tau_i\mu_i$ such that $\partial \mathbf T_i=0$ and
    \begin{equation}
            \Mass\big( (\mathbf T_i-\tau(x_i)\mu)\trace U(x_i,r_i)\big)\leq \varepsilon\mu(U(x_i,r_i)).
        \label{e:mass:auxiliarybundle}
    \end{equation}
\end{enumerate}
    Thanks to Theorem \ref{smirnov}, for any $i\in\N$ we can find a family of vector-valued measures $t\mapsto \boldsymbol\eta_t^i$ satisfying the hypothesis (a) and (b) of Definition \ref{s-measint} and such that for any $i\in\N$ and for almost every $t\in I$ there exists a Lipschitz curve $\gamma_t^i:[0,1]\to\GG$ such that $\boldsymbol\eta_t^i=\llbracket \gamma_t^i\rrbracket=\tau_{\gamma_t^i}\rho_{t,i}\Haus^1\trace \im(\gamma_t^i)$ and
\begin{align*}
    \scal{\mathbf T_i}{\omega}=&\int_I \scal{\llbracket \gamma_t^i\rrbracket}{\omega}\,dt\qquad \text{for every smooth and compactly supported $1$-form $\omega$},
\end{align*}
In addition $\tau_i=\tau_{\gamma^i_t}$ for $\rho_{t,i}\Haus^1\trace \im(\gamma^i_t)$-almost every $x\in \R^n$ and almost every $t\in I$.

Let $N=N(\varepsilon)$ be so big that $\mu(\mathbb{G}\setminus \bigcup_{i=1}^NU(x_i,r_i))\leq\varepsilon\Mass(\mu) $ define on $[0,N]$ and let
$$\boldsymbol\eta_t:=\boldsymbol\eta_{\{t\}}^{\lfloor t\rfloor+1}\trace B(x_{\lfloor t\rfloor},r_{\lfloor t\rfloor}) \qquad \text{for any $t\in[0,N]$}.$$
It is immediate to see that the map $t\mapsto \boldsymbol\eta_t$ is Borel.
Denoted $\boldsymbol\nu:=\int_0^N \boldsymbol\eta_t dt$
we see that $\boldsymbol\nu\llcorner B(x_i,r_i)=\mathbf T_i\llcorner B(x_i,r_i)$ for any $i=1,\ldots,N$ and that the hypothesis of Proposition \ref{p:normal} are satisfied by the family of measures $\boldsymbol \eta_t$. 
This implies that we can find a $1$-dimensional horizontal normal current $\mathbf{T}$ such that $\partial \mathbf{T}=0$, $\mathbb{M}(\mathbf{T})\leq (2+\varepsilon)\sum_{i=1}^N\Mass(\mathbf T_i\trace B(x_i,r_i))$ and $\mathbf{T}=\boldsymbol\nu+\boldsymbol\sigma$ where $\lVert \boldsymbol\nu\rVert+\mu$ and $\lVert \boldsymbol\sigma\rVert$ are mutually singular.
Thanks to the choice of $\boldsymbol\sigma$ and to Radon-Nykodym's decomposition theorem we can write $\mathbf{T}$ as
\begin{equation}
    \mathbf{T}=\sum_{i\in\N}\frac{d  \mathbf T_i}{d\mu}\mu\trace U(x_i,r_i)+\sum_{i\in\N}\boldsymbol\sigma_i+\boldsymbol\sigma,
    \label{expressionforT}
\end{equation}
where the $\boldsymbol\sigma_i$'s are vector valued measures supported on $U(x_i,r_i)$ singular with respect to $\mu\trace U(x_i,r_i)$. Hence, if we write $\mathbf{T}$ as $\mathbf{T}=\tilde{\tau}\mu+\boldsymbol{\tilde{\sigma}}$, where $\mu$ and (the total variation of) $\boldsymbol{\tilde{\sigma}}$ are mutually singular, then 
$$\tilde\tau(y)=\frac{d \mathbf T_i}{d\mu}(y)\qquad\text{for $\mu$-almost every $y\in U(x_i,r_i)$}.$$
Hence
\begin{equation}
    \begin{split}
        \int \lvert \tilde{\tau}(y)-\tau(y)\rvert d\mu(y)=&\sum_{i=1}^\infty\int_{U(x_i,r_i)} \lvert \tilde{\tau}(y)-\tau(y)\rvert d\mu(y) \leq \sum_{i=1}^\infty\int_{U(x_i,r_i)}\lvert \tilde{\tau}(y)-\tau(x_i)\rvert d\mu(y)+\varepsilon\Mass(\mu)\\
        =&\sum_{i=1}^\infty\int_{U(x_i,r_i)}\Big\lvert \frac{d\mathbf T_i}{d\mu}(y)-\tau(x_i)\Big\rvert d\mu(y)+\varepsilon\Mass(\mu).
        \label{estimateL1}
    \end{split}
\end{equation}
Inequality \eqref{e:mass:auxiliarybundle} can be rewritten thanks to \eqref{expressionforT} as
\begin{equation}
  \varepsilon\mu(U(x_i,r_i))\geq  \Mass\big( (\mathbf T_i-\tau(x_i)\mu)\trace U(x_i,r_i)\big)=\int_{U(x_i,r_i)}\Big\lvert \frac{d\mathbf T_i}{d\mu}(y)-\tau(x_i)\Big\rvert d\mu(z)+\Mass(\boldsymbol\sigma_i)+\Mass(\boldsymbol\sigma\trace U(x_i,r_i)).
  \label{e:estimatemassinUi}
\end{equation}
Putting together \eqref{estimateL1} and \eqref{e:estimatemassinUi} we conclude that
$$\int \lvert\tilde{\tau}(y)-\tau(y)\rvert d\mu(y)\leq 2\varepsilon \Mass(\mu).$$
In addition, \eqref{e:mass:auxiliarybundle} also implies that
\begin{equation}
    \begin{split}
        \Mass(\mathbf{T})\leq& (2+\varepsilon)\sum_{i\in\N}\Mass(\mathbf T_i\trace U(x_i,r_i))\leq(2+\varepsilon)\sum_{i\in\N}(\lvert \tau(x_i)\rvert+\varepsilon)\mu(U(x_i,r_i))\\
        \leq& (2+\varepsilon)\Big(\sum_{i\in\N}\int_{U(x_i,r_i)}\lvert \tau(y)\rvert d\mu(y)+ 2\varepsilon\Mass(\mu)\Big) =(2+\varepsilon)(\lVert \tau\rVert_{L^1(\mu)}+2\varepsilon\Mass(\mu)).
    \end{split}
\end{equation}
The claim of the proposition follows by choosing $\varepsilon\leq \lVert \tau\rVert_{L^1(\mu)}\varepsilon_0/4\Mass(\mu)$.
\end{proof}

\begin{proposizione}\label{costruzionecampi}
Let $\mu$ be a Radon measure on $\GG$ and suppose $\tau$ is an $L^1(\mu)$ vector field such that $\tau(x)\in N(\mu,x)$ for $\mu$-almost every $x\in\GG$. Then there exists an horizontal normal current $\mathbf{T}$ on $\GG$ such that:
\begin{itemize}
    \item[(i)] the Radon-Nykodym derivative of $\mathbf{T}$ with respect to $\mu$ coincides $\mu$-almost everywhere with $\tau$, that is $\mathbf{T}=\tau\mu+\mapsto\sigma$ where $\mapsto \sigma$ and $\mu$ are mutually singular;
    \item[(ii)] $\partial \mathbf{T}=0$ and $\Mass(\mathbf{T})\leq 4\lVert \tau\rVert_{L^1(\mu)}$.
\end{itemize}
\end{proposizione}

\begin{proof}
The proof of the proposition follows verbatim that of \cite[Proposition 6.3]{AlbMar} where we replace \cite[Lemma 6.11]{AlbMar} with Lemma \ref{lemmaapproxtau}.
\end{proof}

\begin{teorema}\label{theoAu=Dec}
Let $\mu$ be a Radon measure on $\GG$. Then, for $\mu$-almost every $x\in \GG$ we have $V_1\cap V(\mu,x)=\pi_1(N(\mu,x))$. 
\end{teorema}

\begin{proof}
Let us first prove the inclusion $\pi_1(N(\mu,x))\subseteq V_1\cap V(\mu,x)$. Assume by contradiction that the inclusion does not hold on a set of positive $\mu$-measure. Then, by \cite[Theorem 5.2.1]{Sriva} we can find a bounded Borel vector field $\tau:\GG\cong\R^n\to V_1$ such that $\pi_1(\tau(x))\in \pi_1(N(\mu,x))\setminus V_1\cap V(\mu,x)$ on a set of positive $\mu$-measure. Note that here the measurability of $N(\mu,x)$ provided by Lemma \ref{Nemisurabile} is required.
Thanks to Proposition \ref{costruzionecampi} there exists an horizontal $1$-dimensional normal current $\mathbf{T}$ such that $\mathbf{T}=\tau\mu+\boldsymbol\sigma$ where $\boldsymbol\sigma$ and $\mu$ are mutually singular. Thanks to Proposition \ref{normalcurrdec}, we know that $\mathfrak{S}(\pi_1(\tau(x)))\subseteq  V(\mu,x)\cap V_1$ for $\mu$-almost every $x\in\GG$ which is in contradiction with the choice of $\tau$.

Let us prove the converse inclusion. First of all we modify the map $x\mapsto N(\mu,x)$ in a $\mu$-negligible set to make it Borel.
By definition of the decomposability bundle, it suffices to prove that the map $x\mapsto N(\mu,x)$ belongs to the class $\mathscr{G}_\mu$, see Definition \ref{Dec-bundle}, or in other words, given a measure a family of measures $\mu_t:=\Haus^1\trace \im(\gamma_t)$ such that $\{\mu_t\}_{t\in I}\in \mathscr{F}_\mu$, we must show that
\begin{equation}
    \Tan(\gamma_t,x)=\mathfrak{S}(\mathfrak{v}_{\gamma_t}(x))=\mathfrak{S}(\tau_{\gamma_t}(x))\subseteq \pi_1(N(\mu,x))\qquad \text{for $\mu_t$-almost any $x\in\GG$ and almost every $t\in I$,}
    \label{eq:tan:inc}
\end{equation}
where we recall that $\Tan(\gamma_t,x)$ was introduced in Definition \ref{tg:curvesoso} and $\mathfrak{v}_\gamma$ in Lemma \ref{campovgamma}.
Let us assume by contradiction that $N(\mu,\cdot)\not\in \mathscr{G}_\mu$ and thus that there is a family of measures $\{\mu_t\}_{t\in I}\in\mathscr{F}_\mu$ as above such that \eqref{eq:tan:inc} fails. 

Denoted with $\delta_t$ the Dirac delta at $t$, it can be easily seen that  $t\mapsto \delta_t\otimes \mu_t$ is a Borel function. This can be shown for instance by proving that $\Lambda>0$ and any Borel set $A\subseteq \mathbb{G}$, the function $t\mapsto \delta_t\otimes \mu_t((-\infty,\Lambda]\times A)$ is Borel.
Further, denote by $\mu''$ the measure on $I\times\GG$ 
given by $\mu'' := \int_I (\delta_t \otimes \mu_t) \, dt$, that is well defined thanks to the above discussion and Definition \ref{s-measint} and by $F$ to 
be the set of all $(t,x)\in I\times\R^n$ 
such that $\mathfrak{v}_{\gamma}(x)$,  exists non-zero and is not contained 
in $N(\mu,x)$. Thanks to Lemma \ref{lemmamisurabilitatangenti}, the set $F$ is Borel and hence, the fact that \eqref{eq:tan:inc} does not hold 
is equivalent to saying that $F$ has positive $\mu''$-measure.

Now, let $\Theta$ be a family of cones $C=C(e,\alpha)\subseteq V_1$ with $e$ ranging in a given countable dense subset of
the unit sphere in $V_1$ and $\alpha$ ranging in a given 
countable dense subset of $(0,1)$. For for 
every $C\in\Theta$ we denote by $F_C$ the subset of 
$(t,x) \in F$ such that $\mathfrak{v}_{\gamma_t}(x)$ and $\pi_1(N(\mu,x))$
are separated by $C$, that is,
$\mathfrak{v}_{\gamma_t}(x) \in \Int(C)$ and $\pi_1(N(\mu,x)) \cap C=\{0\}$. Note that the set $F_C$ is Borel since $F$ and of the maps $(t,x)\mapsto\mathfrak{v}_{\gamma_t}(x)$ 
and $x\mapsto N(\mu,x)$ are Borel measurable.
Then the sets $F_C$ with $C\in\Theta$ form a countable
cover of $F$, and since $\mu''(F)>0$ there exists one $\tilde C\in\Theta$ such that $\mu''(F_{\tilde{C}})>0$. We then let $\tilde{I}$ be the projection on $I$ of the set $F_{\tilde{C}}$ and let $\mathcal{G}_{\tilde{C},t}:=\{x\in\GG:(t,x)\in F_{\tilde{C}}\}$, which is Borel for every $t$. It is immediate to see that $\{\mu_t\trace \mathcal{G}_{\tilde{C},t}\}_{t\in\tilde{I}}$ belongs to $\mathscr{F}_\mu$.

The map of vector valued measures $ t\mapsto \mathfrak{v}_{\gamma_t}\mu_t\trace\mathcal{G}_{\tilde{C},t}$
is Borel. 
This is due to Proposition \ref{campovgamma} and the fact that the map $t\mapsto \mu_t\trace\mathcal{G}_{\tilde{C},t}$ is Borel. This implies by Proposition \ref{p:normal} that we can find a $1$-dimensional horizontal normal current $\mathbf{T}$ such that 
$$\mathbf{T}=\int_{\tilde{I}}\mathscr{C}(x)[\mathfrak{v}_{\gamma_t}]\mu_t\trace \mathcal{G}_{\tilde{C},t} dt+\boldsymbol\sigma,$$
where $\boldsymbol\sigma$ is mutually singular with respect to $\mu$. The vector valued measure $\int_{\tilde{I}}\mathscr{C}(x)[\mathfrak{v}_{\gamma_t}]\mu_t\trace \mathcal{G}_{\tilde{C},t} dt$ is not trivial since the vector fields $\mathfrak v_{\gamma_t}\in \tilde C$ for  $\mu_t\trace \mathcal{G}_{\tilde{C},t}$-almost every $x$ and almost every $t\in \tilde{I}$ and thus there cannot be cancellations. 
It is immediate to see that  $\lVert \int_{\tilde{I}}\mathscr{C}(x)[\mathfrak{v}_{\gamma_t}]\mu_t\trace \mathcal{G}_{\tilde{C},t}dt\rVert \ll \mu$. Hence $\lVert \int_{\tilde{I}}\mathscr{C}(x)[\mathfrak{v}_{\gamma_t}]\mu_t\trace \mathcal{G}_{\tilde{C},t}dt\rVert=\rho \mu$ where $\rho\in L^1(\mu)$ and hence on $\{\rho>0\}$ the measures $\mu$ and $\lVert \int_{\tilde{I}}\mathscr{C}(x)[\mathfrak{v}_{\gamma_t}]\mu_t\trace \mathcal{G}_{\tilde{C},t}dt\rVert$ are mutually absolutely continuous.  

Writing $\mathbf{T}=\bar\tau\nu$, since 
\begin{equation}
\mathfrak{v}_{\gamma_t}(x)\in \mathrm{int}(\tilde{C})\qquad\qquad\text{ for any $t\in\tilde{I}$ and $\mu_t$-almost every $x\in \mathcal{G}_{\tilde{C},t}$},
    \label{eq:taucono}
\end{equation}
we infer that $\bar\tau(x)\in\mathscr{C}(x)[ \mathrm{int}(\tilde C)]$ for $\int_{\tilde{I}}\mu_t\trace \mathcal{G}_{\tilde{C},t} dt$-almost every $x\in \GG$.  This implies in particular that 
$$\pi_1(\bar\tau(x))\in  \mathrm{int}(\tilde C)\qquad\text{on a set of $\mu$-positive measure.}$$
However, this is in contradiction with
Proposition \ref{normaleimpliesinN} which implies that $\pi_1( \bar\tau(x))\in \pi_1(N(\mu,x))$ for $\mu$-almost every $x\in \GG$.
\end{proof}

\section{Differentiability along the decomposability bundle}
\label{s3n}
This section is devoted to the proof of Theorem \ref{t:differentiability_along_bundle}. \emph{Throughout the rest of this section and if not otherwise specified, $\mathbb{H}$  will always be a fixed Carnot group endowed with a homogeneous and left invariant distance.}

%\begin{teorema}\label{t:differentiability_along_bundle}
%Let $\mu$ be a Radon measure on $\GG$. For every Carnot group $\HH$, every Lipschiz function $f:\GG\to\HH$ is Pansu differentiable at $\mu$-almost every $x\in\GG$ with respect to the homogeneous subgroup $V(\mu,x)$.
%\end{teorema}
    
\subsection{Construction of vector fields of universal differentiability.}
    
This subsection is devoted to the proof of the following

\begin{lemma}\label{LEMMAvector fieldS}
Let $\mu$ be a Radon measure on $\GG$. Then, there are $n_1$ Borel maps $\zeta_1,\dots,\zeta_{n_1}:\GG\to\GG$ such that:
\begin{itemize}
    \item[(i)]$V(\mu,x)=\mathfrak{S}(\{\zeta_1(x),\dots,\zeta_{n_1}(x)\})$ for $\mu$-almost every $x\in K$,
    \item[(ii)]every $f\in\mathrm{Lip}(\GG,\HH)$ has derivative along $\zeta_i(x)$ for every $i=1,\ldots,n_1$ and for $\mu$-almost every $x\in \GG$.
\end{itemize} 
\end{lemma}

\begin{proof}
For any $i=1,\ldots,n_1$ define
$$\zeta_i(x):=\begin{cases}\frac{\pi_{V(\mu,\cdot)}(x)[e_i]}{\lvert\pi_{V(\mu,\cdot)}(x)[e_i]\rvert}& \text{if } \pi_{V(\mu,\cdot)}(x)[e_i]\neq 0,\\
0& \text{otherwise}.\end{cases}\qquad\text{and}\qquad\xi_i(x):=\mathscr{C}(x)[\zeta_i(x)].$$
where the $e_i$s are the vectors of Definition \ref{campii} and the map $\pi_{V(\mu,\cdot)}$ is the projection map associated to $V(\mu,\cdot)$ yielded by Proposition \ref{prop:projmap}. Note that $\mu$-almost every $x\in \GG$ the vectors  $\zeta_i(x)$ are contained $V(\mu,\cdot)\cap V_1$ and since $e_1,\ldots,e_{n_1}$ are othronormal, the Borel vector fields $\zeta_i$
 span $V(\mu,\cdot)\cap V_1$ at $\mu$-almost every $x\in\GG$. Furthermore, by Remark \ref{rk:expc} for $\mu$-almost every $x\in\GG$ on the one hand we have that the vector fields $\xi_1,\ldots,\xi_{n_1}$ span the vector space $\mathscr{C}(x)[V(\mu,x)\cap V_1]$  and we also have that the identity $\pi_1[\xi_i(x)]=\zeta_i(x)$ holds at every $x\in\GG$.

 For any $i=1,\ldots,n_1$ let $T_i=\tau_i\mu_i$ be the $1$-dimensional horizontal normal currents yielded by Proposition \ref{costruzionecampi} with the choice $\tau=\zeta_i$. Thanks to Theorem \ref{smirnov} for any $i$ we can find a family of vector-valued measures $t\mapsto \mu_t^i$ satisfying the hypothesis (a) and (b) of Definition \ref{s-measint} and such that $T_i$ can be written as $T_i=\int_I \mu_t^i\, dt$. Thanks to Theorem \ref{smirnov} we infer that for any $i$ and every $t$ there exists a Lipschitz curve $\gamma_t^i:[0,1]\to \GG$ such that $\mu_t^i=\llbracket\gamma_t^i\rrbracket$ and $\tau_{\gamma_t^i}=\tau_i$ for $\rho_{t,i}\Haus^1\trace \im(\gamma_t^i)$-almost every $x\in \GG$. It is elementary to see that every Lipschitz map $f\in\mathrm{Lip}(\GG,\HH)$ is Pansu differentiable along $\pi_1(\tau_i)$ for $\Haus^1\trace\im(\gamma_t^i)$-almost every $x\in \R^n$ and almost every $t\in I$. This implies in particular that every $f\in\mathrm{Lip}(\GG,\HH)$ is Pansu differentiable along $\pi_1(\tau_i)$ for $\int \rho_{t,i}\Haus^1\trace \im(\gamma_t^i)$-almost every $x\in \R^n$. However, since by Theorem \ref{smirnov} we have that $\mu_i=\int\rho_{t,i} \Haus^1\trace \im(\gamma_t^i)$ and that $\pi_1(\tau_i)=\pi_1(\xi_i)=\zeta_i$ for $\mu$-almost every $x\in \R^n$ we conclude that every $f\in\mathrm{Lip}(\GG,\HH) $ is Pansu differentiable along $\zeta_i$ for $\mu$-almost every $x\in\R^n$. Thanks to Proposition \ref{generato} and the fact that $\zeta_1,\ldots\zeta_{n_1}$ span $V(\mu,\cdot)\cap V_1$ $\mu$-almost everywhere, the proof of the lemma is achieved.
\end{proof}

\subsection{Partial and total derivatives}

In this section we relate the existence of partial derivatives to the Pansu differentiability along the  decomposability bundle. Since the group operation is not commutative, we cannot follow the proof of the Euclidean counterpart, see \cite[Section 3]{AlbMar}

\begin{definizione}
A function $f:\GG\to \HH$ is said to have \emph{derivative} $Df(x,\zeta)$ \emph{at} $x$ \emph{along} $\zeta\in\GG$ if the following limit exists:
$$Df(x,\zeta)=\lim_{r\to 0}\delta_{1/r}(f(x)^{-1}*f(x\delta_r(\zeta)))\in\GG.$$
Furthermore, $f$ is said to be \emph{differentiable at the point} $x\in\mathbb{G}$ \emph{along} $\zeta\in\GG$, if $Df(x,\zeta)$ and $Df(x,\zeta^{-1})$ exist and $Df(x,\zeta)^{-1}=Df(x,\zeta^{-1})$. 
\end{definizione}

\begin{osservazione}\label{oss:hom}
Note that if $Df(x,\zeta)$ exists, then for any $\lambda>0$ the derivative $D(f,\delta_\lambda(\zeta))$ of $f$ along $\delta_{\lambda}(\zeta)$ at $x$ exists and:
$$Df(x,\delta_\lambda(\zeta))=\delta_\lambda Df(x,\zeta).$$
Indeed:
$$\lim_{t\to 0}\delta_{1/r}(f(x)^{-1}f(x\delta_{\lambda r}(\zeta)))=\delta_{\lambda}\bigg(\lim_{r\to 0}\delta_{1/\lambda r}(f(x)^{-1}f(x\delta_{\lambda r}(\zeta)))\bigg)=\delta_\lambda(Df(x,\zeta)).$$
\end{osservazione}

\begin{proposizione}\label{prop:diff_impl_direct}
Let $\mu$ be a Radon measure on $\GG$ and $\zeta:\GG\to\GG$ be a Borel map such that any Lipschitz map $f:\GG\to \HH$ is differentiable $\mu$-almost everywhere along $\zeta(\cdot)$. Finally let $B$ be any $\mu$-positive Borel subset of $\supp(\mu)$. 
Then, for $\mu$-almost every $x\in B$, there exists a $t(x)>0$ and a map $x(\cdot): (-t(x),t(x))\to B$ such that:
\begin{equation}
\lim_{t\to 0}\frac{d_c(x(t),x*\delta_t(\zeta(x)))}{t}=0.
\label{eq:mer2}
\end{equation}
\end{proposizione}

\begin{proof}
Since in any Carnot group there is an isometrically embedded copy of $\R$ if we prove the claim for $\HH=\R$, the result follows in full generality. The first step of the proof is to show that the function
$g(x):=\inf\{r>0:\mu(B(x,r)\cap B)>0\}$ is a non-negative $1$-Lipschitz function. Let $x,y\in \GG$ and note that $B(y,r)\subseteq B(x,r+d(x,y))$. Therefore, for any $\varepsilon>0$ we have:
\begin{equation}
\mu(B(x,g(y)+d(x,y)+\varepsilon)\cap B)\geq \mu(B(y,g(y)+\varepsilon)\cap B)>0.
\label{eq:mer1}
\end{equation}
Inequality \eqref{eq:mer1} implies that $g(x)\leq g(y)+d(x,y)$ and thus, interchaning $x$ and $y$, the claim is proved.

Suppose by contradiction that there is a $\mu$-positive compact set $K\subseteq B$ for which \eqref{eq:mer2} fails $\mu$-almost everywhere on $K$. This means that for $\mu$-almost every $x\in K$ there is an infinitesimal sequence $s_i(x)$ and a $\lambda(x)>0$ such that:
\begin{equation}
\mathrm{dist}(x*\delta_{s_i(x)}(\zeta(x)),B)\geq \lambda(x) s_i(x)\qquad \text{for any }i\in\mathbb{N}.
\label{eq:mer5}
\end{equation}

First of all, let us show that $g(z)=0$ for $\mu$-almost every $z\in K$. If this is not the case, we can find a $\mu$-positive subcompact $\tilde K$ such that $g>0$ on $\tilde K$. Then, by definition of $g$, for any $z\in \tilde K$ we can find a $r(z)>0$ such that $\mu(\mathrm{int} \big(B(x,r(x)\big)\cap B)=0$. However, since $\tilde K$ is compact, we can find a finitely many $z_i\in K$, with $i=1,\ldots,N$ such that 
$$\tilde{K}\subseteq \bigcup_{i=1}^N B(z_i,r_i).$$
This implies that
$$\mu(\tilde{K})=\mu(\tilde{K}\cap B)\leq \sum_{i=1}^N\mu(B(z_i,r(z_i))\cap B)=0,$$
which is in contradiction with the fact that $\tilde{K}$ had positive measure. This concludes the proof of the fact that $g(z)=0$ for $\mu$-almost every $z\in K$.
In order to discuss why \eqref{eq:mer5} is false, we shall fix a $z\in K$ where $g(z)=0$ and note that \eqref{eq:mer5} implies that
\begin{equation}
\limsup_{r\to 0}\frac{\mathrm{dist}\big(z*\delta_{r}(\zeta(z)),B\big)}{\lvert r\rvert}\geq \eta(z),
\label{eq:mer3}
\end{equation}
for some positive $\eta(z)>0$ depending on $z$. We can also assume without loss of generality that $g$ is differentiable along $\zeta(z)$ at $z$. This implies that 
$$Dg(z,\zeta(z))=\lim_{r\to 0}\frac{g(z*\delta_r(\zeta(z)))-g(z)}{\lvert r\rvert}=\limsup_{r\to 0}\frac{g(z*\delta_r(\zeta(z)))}{\lvert r\rvert}\geq \eta(z).$$ This, together with the fact that $g$ is non-negative implies that $g$ cannot be differentiable along $\zeta(z)$ at $z$, since the identity $Dg(z,\zeta(z))=-Dg(z,\zeta(z)^{-1})$ cannot be satisfied even if both $Dg(z,\zeta(z))$ and $Dg(z,\zeta(z)^{-1})$ existed. The Borel regularity of the measure $\mu$ yields the desired conclusion.
\end{proof}

\begin{proposizione}\label{composition}
Suppose $\mu$ is a Radon measure on $\GG$ and assume $\zeta_1,\zeta_2:\GG\to \GG$ are two Borel vector fields such that every $f\in\mathrm{Lip}(\GG,\HH)$ is differentiable along both $\zeta_1(x)$ and $\zeta_2(x)$ for $\mu$-almost every $x\in \GG$. Then, $\mu$-almost everywhere, every $f$ is differentiable along $\zeta_{i_1}(x)^{\beta_1}\zeta_{i_2}(x)^{\beta_2}$, where $i_j\in\{1,2\}$ and $\beta_j\in\{\pm1\}$ as $j=1,2$. Furthermore, we have:
\begin{equation}
    Df(x,\zeta_{i_1}(x)^{\beta_1}\zeta_{i_2}(x)^{\beta_2})=Df(x,\zeta_{i_1}(x))^{\beta_1}Df(x,\zeta_{i_2}(x))^{\beta_2}.
    \label{hom:condition}
\end{equation}
\end{proposizione}

\begin{proof}Without loss of generality we can assume that the measure $\mu$ is supported on a compact set $K$. Therefore, thanks to
Severini-Egoroff's theorem and Lusin's theorem we can find a compact set $K_1$ such that:
\begin{itemize}
    \item[(i)]$\mu(K\setminus K_1)\leq \varepsilon\mu(K)$,
    \item[(\hypertarget{unif}{ii})] the incremental ratios
    $Rf(x,\zeta_i(x);t):=\delta_{1/t}(f(x)^{-1}f(x*\zeta_i(x)))$
    converge uniformly to $Df(x,\zeta_i)$ on $K_1$ as $t$ goes to $0$ for $i=1,2$,
    \item[(\hypertarget{continuity}{iii})]the maps $\zeta_i(\cdot)$ and $Df(x,\zeta_i(x)^\beta)$ are continuous on $K_1$ for any $i=1,2$ and $\beta\in\{\pm1\}$.
\end{itemize}
Let $\beta_1,\beta_2\in\{\pm 1\}$ and $i_1,i_2\in\{1,2\}$ and note that:
\begin{equation}
Rf(x,\zeta_{i_1}(x)^{\beta_1}\zeta_{i_2}(x)^{\beta_2};t)=Rf(x,\zeta_{i_1}(x)^{\beta_1};t)*Rf(x*\delta_t(\zeta_{i_1}(x)),\zeta_{i_2}(x)^{\beta_2};t).
\label{eq:n5}
\end{equation}
By (\hyperlink{unif}{ii}) we immediately infer that $\lim_{t\to0}Rf(x,\zeta_{i_1}^{\beta_1};t)=Df(x,\zeta_{i_1}(x))$. This implies in particular that in order to conclude the proof of the proposition, we just need to show that:
\begin{equation}
    \lim_{r\to 0}Rf(x*\delta_t(\zeta_{i_1}(x))^{\beta_1},\zeta_{i_2}(x)^{\beta_2};t)=Df(x,\zeta_{i_2}(x)^{\beta_2}).
    \label{eq:finalew}
\end{equation}
Thanks to Proposition \ref{prop:diff_impl_direct}, for $\mu$-almost every $x\in K_1$ we can find a map $x(t)$ taking values in $K_1$ for which:
\begin{equation}
\lim_{t\to 0}\frac{d_c(x(t),x\delta_t(\zeta_{i_1}(x)^{\beta_1}))}{t}=0.
    \label{eq:n3}
\end{equation}
With the aid of the map $x(t)$, we can rewrite $Rf(x\delta_t(\zeta_{i_1}(x))^{\beta_1},\zeta_{i_2}(x)^{\beta_2};t)$ as follows
\begin{equation}
\begin{split}
   &\qquad\qquad\qquad\qquad\qquad\qquad\qquad\qquad Rf(x*\delta_t(\zeta_{i_1}(x))^{\beta_1},\zeta_{i_2}(x)^{\beta_2};t)\\
    =&\underbrace{\delta_{1/t}\Big(f\big(x*\delta_t(\zeta_{i_1}(x)^{\beta_1})\big)^{-1}f(x(t))\Big)}_{\text{(I)}}* \underbrace{Rf(x(t),\zeta_{i_2}(x)^{\beta_2};t)}_{\text{(II)}}* \underbrace{\delta_{1/t}\Big(f\big(x(t)*\delta_t(\zeta_{i_2}(x)^{\beta_1}\big)^{-1}f(x*\delta_{t}(\zeta_{i_1}(x)^{\beta_1}*\zeta_{i_2}(x)^{\beta_2})\Big)}_{\text{(III)}}.
    \label{ingrassia}
\end{split}
\end{equation}
Let us separately estimate the norm of the terms (I), (II) and (III). Thanks to the Lipschitzianity of $f$, we deduce that:
\begin{equation}
    \frac{\lim_{t\to 0}\lVert \text{(I)}\rVert}{\text{Lip}(f)}\leq \lim_{t\to 0} \frac{d_c(x*\delta_t(\zeta_{i_1}(x)^{\beta_1}),x(t))}{t}=0.\label{eq:n1}
\end{equation}
Furthermore, thanks to \eqref{eq:n3} we infer that:
\begin{equation}
    \begin{split}
         \frac{\lim_{t\to 0}\lVert \text{(III)}\rVert}{\text{Lip}(f)}\leq& \lim_{t\to 0} \frac{d\big(x(t)*\delta_t(\zeta_{i_2}(x)^{\beta_2}),x*\delta_{t}(\zeta_{i_1}(x)^{\beta_1}*\zeta_{i_2}(x)^{\beta_2})\big)}{t}\\
     =&\lim_{t\to 0} \lVert \zeta_{i_2}(x)^{-\beta_2}\delta_{1/t}\big(x(t)^{-1}*x*\delta_t(\zeta_{i_1}(x)^{\beta_1})\big)*\zeta_{i_2}(x)^{\beta_2}\rVert=0.
     \label{eq:n2}
    \end{split}
\end{equation}
Finally, we can rewrite (II) in the following convenient way:
\begin{equation}
        \text{(II)}=Rf(x(t),\zeta_{i_2}(x(t))^{\beta_2};t)*\underbrace{\delta_{1/t}(f(x(t)*\delta_t(\zeta_{i_2}(x(t))^{\beta_2})^{-1}*f(x(t)*\delta_t(\zeta_{i_2}(x)^{\beta_2}))}_{\text{(IV)}}.
    \label{ciccio}
\end{equation}
Thanks to (\hyperlink{unif}{ii}) and the fact that $x(t)\in K_1$, for any $\varepsilon>0$ there exists a $t_\varepsilon>0$ such that:
$$\lVert Df(x(t),\zeta_{i_2}(x(t))^{\beta_2})^{-1} Rf(x(t),\zeta_{i_2}(x(t))^{\beta_2};t)\rVert\leq \varepsilon,$$
for any $\lvert t\rvert\leq t_\varepsilon$. Finally, the Lipschitzianity of $f$ and (\hyperlink{continuity}{iii}) imply that:
\begin{equation}
    \lim_{t\to 0}\lVert \text{(IV)}\rVert\leq \text{Lip}(f)\lim_{t\to 0}\lVert\zeta_{i_2}(x(t))^{-\beta_2}\zeta_{i_2}(x)^{\beta_2}\rVert=0.\label{pippofranco}
\end{equation}
    Putting together the information we gathered, we infer that:
    \begin{equation}
    \begin{split}
         &\qquad\qquad\qquad \qquad\limsup_{t\to 0}\lVert Df(x,\zeta_{i_2}(x)^{\beta_2})^{-1}*Rf(x*\delta_t(\zeta_{i_1}(x)),\zeta_{i_2}(x)^{\beta_2};t)\rVert\\
        &\underset{\eqref{ingrassia},\eqref{ciccio}}{=}\limsup_{t\to 0}\lVert Df(x,\zeta_{i_2}(x)^{\beta_2})^{-1}*\mathrm{(I)}*Df(x,\zeta_{i_2}(x)^{\beta_2})*Df(x,\zeta_{i_2}(x)^{\beta_2})^{-1}*Rf(x(t),\zeta_{i_2}(x(t))^{\beta_2};t)*\mathrm{(IV)}*\mathrm{(III)}\rVert\\
        &\qquad\qquad\qquad\qquad\underset{\eqref{eq:n1},\eqref{eq:n2},\eqref{pippofranco}}{=}\limsup_{t\to 0}\lVert Df(x,\zeta_{i_2}(x)^{\beta_2})^{-1}*Rf(x(t),\zeta_{i_2}(x(t))^{\beta_2};t)\rVert\\
        &\leq\limsup_{t\to 0}\lVert Df(x,\zeta_{i_2}(x)^{\beta_2})^{-1}*Df(x(t),\zeta_{i_2}(x(t))^{\beta_2})\rVert+\lVert Df(x(t),\zeta_{i_2}(x(t))^{\beta_2})^{-1}*Rf(x(t),\zeta_{i_2}(x(t))^{\beta_2};t)\rVert\leq \epsilon,
            \nonumber
    \end{split}
    \end{equation}
where in the last identity we also used Lemma \ref{lem:EstimateOnConjugate} and where the last inequality above comes from (\hyperlink{continuity}{iii}). The arbitrariness of $\varepsilon$ concludes the proof.
\end{proof}

\begin{teorema}\label{th:main2.1}
Suppose $\mathscr{D}$ is a finite family of Borel maps $\zeta:\GG\to \GG$ such that any $f\in \mathrm{Lip}(\GG,\HH)$ is differentiable at $\mu$-almost every $x\in\GG$ along $\zeta(x)$. Then, every Lipschitz map $f\in \mathrm{Lip}(\GG,\HH)$ is Pansu differentiable with respect to the subgroup $\mathfrak{S}(\{\zeta(x):\zeta\in\mathscr{D}\})$ for $\mu$-almost any $x\in\GG$ .
\end{teorema}

\begin{proof}
Let $v:\GG\to\GG$ be a map for which there exists an $N\in\N$, $\rho_i\in\mathbb{Q}$ and $v_i\in\mathscr{D}$ with $i=1,\ldots,N$ such that
\begin{equation}
    v(x)=\delta_{\rho_1}(v_1(x))*\cdots*\delta_{\rho_N}(v_N(x)).
    \label{eq:vf}
\end{equation}
Let $\tilde{\mathfrak{S}}$ be the countable family of maps that satisfy identity \eqref{eq:vf} for some choice of $N$, $\rho_i$ and $v_i$ and let
$$\tilde{\mathfrak{S}}(x):=\{w\in\GG:\text{there exists a }v\in\tilde{\mathfrak{S}}\text{ such that }v(x)=w\}.$$
Proposition \ref{composition} and Remark \ref{oss:hom} immediately imply that for $\mu$-almost every $x\in \GG$ every Lipschitz map is differentiable along $v(x)$ whenever $v\in\tilde{\mathfrak{S}}$ and
\begin{equation}
    Df(x,u(x)v(x))=Df(x,u(x))Df(x,v(x))\qquad\text{for $\mu$-almost every $x\in\GG$ and any $u,v\in\tilde{\mathfrak{S}}$}.
    \label{id:composiz}
\end{equation}
In particular this can be rephrased as follows. For $\mu$-almost every $x\in \GG$, every Lipschitz map is differentiable along any $v\in\tilde{\mathfrak{S}}(x)$ and
$ Df(x,u*v)=Df(x,u)Df(x,v)$ for $\mu$-almost every $x\in\GG$ and any $u,v\in\tilde{\mathfrak{S}}(x)$.

The next step in the proof is to show that for $\mu$-almost every $x\in\GG$ and any $w\in\mathrm{cl}(\tilde{\mathfrak{S}}(x))$ every Lipschitz function is differentiable along $w$ at $x$.
Thanks to the choice of $w$ there exists a Cauchy sequence $\{w_i\}_{i\in\N}\subseteq \tilde{\mathfrak{S}}(x)$ such that for any $k\in\N$ there exists an $M\in \N$ such that for any $i,j\geq M$ we have $d_\GG(w_j,w_i)\leq 1/k$. Since $w_i^{-1}w_j\in\tilde{\mathfrak{S}}(x)$, thanks to \eqref{id:composiz} we infer that 
\begin{equation}
\begin{split}
     d_\HH(Df(x,w_i),Df(x,w_j))=&d_\HH(Df(x,w_i^{-1}*w_j),0)=\lim_{t\to 0} \frac{d_\HH\big(f(x)^{-1}*f(x*\delta_t(w_i^{-1}*w_j)),0\big)}{t}\\
     \leq& \mathrm{Lip}(f)d_\GG(w_j,w_i)\leq \mathrm{Lip}(f)/k.
     \label{cauchyder}
\end{split}
\end{equation}
for any $i,j\geq M$. The bound \eqref{cauchyder} shows that the sequence $\{Df(x,w_i)\}_{i\in\N}$ is Cauchy in $\HH$ and thus there exists an element of $\HH$, that we denote by $\mathfrak{d}f(x,w)$, such that $\lim_{i\to\infty}Df(x,v_i)=\mathfrak{d}f(x,v)$. However, for any $i\in\N$ we have
\begin{equation}
    \begin{split}
        &\qquad\qquad\qquad\limsup_{r\to 0}\frac{\lVert\delta_t(\mathfrak{d}f(x,w))^{-1}*f(x)^{-1}*f(x*\delta_t(w))\rVert_\HH}{\lvert t\rvert}\\
        \leq& \lVert \mathfrak{d}f(x,w_i)^{-1}*\mathfrak{d}f(x,w)\rVert_\HH+\limsup_{r\to 0}\frac{\lVert\delta_t(\mathfrak{d}f(x,w_i))^{-1}*f(x)^{-1}*f(x*\delta_t(w_i))\rVert_\HH}{\lvert t\rvert}\\
        &\qquad\qquad\qquad\qquad\qquad\quad\,\,+\limsup_{r\to 0}\frac{\lVert f(x*\delta_t(w_i))^{-1}*f(x*\delta_t(w)) \rVert_\HH}{\lvert t\rvert}\\
        \leq &\lVert \mathfrak{d}f(x,w_i)^{-1}*\mathfrak{d}f(x,w)\rVert_\HH+\mathrm{Lip}(f)d_\GG(w_i,w).
    \end{split}
\end{equation}
The arbitrariness of $i$ implies that
$$\mathfrak{d}f(x,w)=\lim_{r\to 0}\delta_{1/r}(f(x)^{-1}* f(x*\delta_r(w))=Df(x,w),$$
and this shows that $f$ is differentiable at $x$ along $w$. Note in particular that the above computations also prove that the function $Df(x,\cdot):\mathrm{cl}(\tilde{\mathfrak{S}}(x))\to\HH$ is continuous.

Since it can be easily seen that $\mathfrak{S}(\{v(x):v\in\mathscr{D}\})=\mathrm{cl}(\tilde{\mathfrak{S}}(x))$ for any $x\in\GG$, the only thing left to prove is that the map $v\mapsto Df(x,v)$ is a homogeneous homomorphism on $\mathrm{cl}(\tilde{\mathfrak{S}}(x))$. 
To do to this, let $v,w\in \mathrm{cl}(\tilde{\mathfrak{S}}(x))$ and let $v_i,w_i\subseteq \tilde{\mathfrak{S}}(x)$ be two sequence converging to $v$ and $w$ respectively. 
Since the sequence $v_i*w_i\in \text{cl}(\tilde{\mathfrak{S}}(x))$ converges to $v*w$, by the continuity of $Df(x,\cdot)$, we infer:
$$Df(x,v*w)=\lim_{i\to\infty}Df(x,v_i*w_i)=\lim_{i\to\infty}Df(x,v_i)*Df(x,w_i)=Df(x,v)*Df(x,w).$$
This concludes the proof, since the homogeneity of $Df(x,\cdot)$ is guaranteed by Remark \ref{oss:hom}.
\end{proof}

\begin{teorema}\label{th.differentiability}
Let $\mu$ be a Radon measure on $\GG$. Then, for any Carnot group $\HH$ and for $\mu$-almost every $x\in\GG$ every Lipschitz map $f\in\mathrm{Lip}(\GG,\HH)$ is differentiable along the subgroup $V(\mu,x)\in\Gr_\mathfrak{C}(\GG)$, the decomposability bundle of $\mu$ defined in Definition \ref{Dec-bundle}, for $\mu$-almost every $x\in \GG$.
\end{teorema}

\begin{proof}
The Theorem follows immediately from Lemma \ref{LEMMAvector fieldS}, which guarantees that every Lipschitz function admits directional derivatives along a family of Borel vector fields $\zeta_1,\dots,\zeta_{n_1}:\GG\to\GG$ generating the decomposability bundle $V(\mu,x)$ at $\mu$-almost every  $x$, and Theorem \ref{th:main2.1} guarantees that these directional derivatives give rise to the Pansu differentiability with respect to the decomposability bundle.
\end{proof}

\begin{osservazione}
Here below we list some observation on Theorem \ref{th.differentiability} and its proof.
\begin{enumerate}
    \item With few modifications to the proofs, the statement of Theorem \ref{th.differentiability} can be strengthened to the following form.
    \begin{itemize}
        \item[]\emph{Let $\mu$ be a Radon measure on $\GG$ and $B\subseteq \GG$ be a Borel set. Then, for any \textbf{homogeneous} group $\HH$ and for $\mu$-almost every $x\in\GG$ every Lipschitz map $f\in\mathrm{Lip}(B,\HH)$ is differentiable along  $V(\mu,x)$, for $\mu$-almost every $x\in B$, i.e.
        $$\lim_{B\ni y\to x} \frac{\lVert df(x)[x^{-1}y]^{-1}f(x)^{-1}f(y)\rVert_\mathbb{H}}{d_c(x,y)}=0,$$
        for some homogeneous homomorphism $df(x):V\to \HH$.
        }
    \end{itemize}
    where here $\mathrm{Lip}(B,\HH)$ denotes the family of Lipschitz maps $f:B\subseteq \GG\to\HH$. This is a non-trivial extension as it is well known that maps between general Carnot groups do not enjoy any extension property, see for instance  \cite[Theorem 1]{MR3494181}.
    \item 
    % {\color{red}By using the same techniques employed in \cite{AlbMar} one can show that there exists a Lipschitz function $f\in\mathrm{Lip}(\GG,\R)$ such that $f$ is non differentiable along any $v\in \GG\setminus (V(\mu,x)\cup \textrm{exp}( V_2\oplus\ldots\oplus V_s))$ for $\mu$-almost every $x\in\GG$. It is a simple observation to note, see for instance \cite[Remark 1.2]{julia2021lipschitz}, that there are measures $\mu$ for which $V(\mu,x)$ is the largest subspace of differentiability for Lipschitz functions, in the following sense: if $V:\GG\to\Gr(\GG)$ is a Borel map such that for every Carnot group $\HH$ every $f\in\Lip(\GG,\HH)$ is differentiable $\mu$-a.e. along $V(x)$, then we have $V(x)\subseteq V(\mu,x)$ for $\mu$-almost every $x\in\GG$. 
    % } 
At this stage it is not clear whether the decomposability bundle constructed here is sharp in the sense that on the directions $v$ on $\mathbb{G}$ not contained in $V(\mu,x)$ there are Lipschitz function $f:\GG\to \R$ which are non-differentiable along $v$ at $x$, compare with \cite[Theorem 1.1(ii)]{AlbMar}. It seems however plausible that the same techniques employed in \cite{AlbMar} might yield the existence of a Lipschitz function $f\in\mathrm{Lip}(\GG,\R)$ such that $f$ is non differentiable along any $v\in \GG\setminus (V(\mu,x)\cup \textrm{exp}( V_2\oplus\ldots\oplus V_s))$ for $\mu$-almost every $x\in\GG$. This will be subject to further investigation.

Finally, it is a simple observation to note, see for instance \cite[Remark 1.2]{julia2021lipschitz}, that there are measures $\mu$ for which $V(\mu,x)$ is the largest subspace of differentiability for Lipschitz functions, in the following sense: if $V:\GG\to\Gr(\GG)$ is a Borel map such that for every Carnot group $\HH$ every $f\in\Lip(\GG,\HH)$ is differentiable $\mu$-a.e. along $V(x)$, then we have $V(x)\subseteq V(\mu,x)$ for $\mu$-almost every $x\in\GG$. 
\end{enumerate}
\end{osservazione}

\section{The reverse of Pansu's theorem}
\label{s7}

\subsection{Decompositions of a measure satisfying Pansu's Theorem}

\begin{definizione}\label{pansuproperty}
We say that a Carnot group $\GG$ endowed with a Radon measure $\mu$ has the \emph{Pansu property} with respect to a Carnot group $\mathbb{H}$ if for any Lipschitz function $f:\GG\to \HH$ and for $\mu$-almost every $x_0\in \GG$ there exists a homogeneous homomorphism $\text{d}f(x_0):\GG\to \HH$ such that
\begin{equation}
    \limsup_{x\to x_0}\frac{d_\mathbb{H}(f(x),f(x_0)*\text{d}f(x_0)[x_0^{-1}x])}{d_c(x,x_0)}=0.
    \label{def:panprop}
\end{equation}
\end{definizione}

\begin{osservazione}\label{rk:homreali}
If $D:\GG\to\R$ is a group homomorphism thanks to \cite[Proposition 2.5]{step2} for any $g\in\GG$ we have $D[g]=D[\pi_1(g)]$. 
\end{osservazione}

\begin{osservazione}\label{rk:pansuvsreale}
Suppose $\mu$ is a Radon measure on $\mathbb{G}$ with the Pansu property with respect to some Carnot group $\mathbb{H}$.
Let $g:\mathbb{G}\to \R$ be a Lipschitz map and $e$ be an element of the first layer $V_1$ of $\mathbb{G}$. It is easily seen that the map $f:\mathbb{G}\to \mathbb{H}$ defined as 
$f(x):=\delta_{g(x)}(e)$,
is Lipschitz and 
\begin{equation}
\begin{split}
       0=\limsup_{x\to x_0}\frac{\lVert\text{d}f(x_0)[x_0^{-1}x]^{-1}f(x_0)^{-1}f(x)\rVert_\mathbb{H}}{d_c(x,x_0)}=\limsup_{x\to x_0}\frac{\lVert\text{d}f(x_0)[x_0^{-1}x]^{-1}\delta_{g(x)-g(x_0)}(e)\rVert_\mathbb{H}}{d_c(x,x_0)}.
\end{split}
\end{equation}
This shows in particular that for any $v\in\GG$ we have
$$\lim_{r\to 0}\frac{g(x_0\delta_r(v))-g(x_0)}{r} e=\mathrm{d}f(x_0)[v].$$
Therefore, the image of the homogeneous homomorphism $\mathrm{d}f(x_0)$ is contained in the $1$-parameter subgroup generated by $e$. This immediately shows together with Remark \ref{rk:homreali} that $\mathrm{d}f(x_0)=\delta_{\langle L(x_0),\pi_1(x_0^{-1}x)\rangle}( e) $, where $L(x_0)$ is a suitable element of $V_1$. It is thus immediate to see that defined $\mathrm{d}g(x_0):=\langle L(x_0),\pi_1(x_0^{-1}x)\rangle$ we have
\begin{equation}
    \limsup_{x\to x_0}\frac{\lvert g(x)-g(x_0)-\mathrm{d}g(x_0)[x_0^{-1}x]\rvert}{d_c(x,x_0)}=0.
    \nonumber
\end{equation}
This shows that in order to prove Theorem \ref{t:main} it is sufficient to restrict ourselves to the case where the real line and that the definition of the Pansu property with real-valued functions is the weakest possible.
\end{osservazione}

\begin{proposizione}\label{prop:lipdiffspace}
Suppose the Carnot group $\GG$ endowed with the measure $\mu$ has the Pansu property. Then, $(\GG,d,\mu)$ is a Lipschitz differentiability space with the global chart $\pi_1:\GG\to V_1$.
\end{proposizione}

\begin{proof}
Thanks to Remark \ref{rk:homreali}, for any Lipschitz function $f:\GG\to\R$ and $\mu$-almost any $x\in\GG$ we have:
\begin{equation}
    \begin{split}
        0=\limsup_{x\to x_0}\frac{\lvert f(x)-f(x_0)-\mathrm{d}f(x_0)[x_0^{-1}x]\rvert}{d_c(x,x_0)}=&\limsup_{x\to x_0}\frac{\lvert f(x)-f(x_0)-\mathrm{d}f(x_0)[\pi_1(x_0^{-1}x)]\rvert}{d_c(x,x_0)}\\
        =&\limsup_{x\to x_0}\frac{\lvert f(x)-f(x_0)-\mathrm{d}f(x_0)[\pi_1(x)-\pi_1(x_0)]\rvert}{d_c(x,x_0)}.
        \nonumber
    \end{split}
\end{equation}
The above computation shows that the hypothesis of the axioms of Lipschitz differentiability space are satisfied by $(\GG,d,\mu)$ with the global chart $\pi_1:\GG\to V_1$.
\end{proof}

\begin{osservazione}\label{rkstruct}
The chart $(\pi_1,\GG)$ is $1$-structured in the sense of \cite[Definition 3.6]{Bate}. Indeed, if for any $x\in \GG$ and any $R>0$, we let $x_i:=x*\delta_Re_i$, then:
$$\max_{1\leq i\leq n_1}\frac{\lvert \langle\pi_1(x_i)-\pi_1(x), e_i\rangle\rvert}{d_c(x_i,x)}=\max_{1\leq i\leq n_1}\frac{\lvert \langle x_1+Re_i-x_1, e_i\rangle\rvert}{d_c(\delta_Re_i,0)}=1.$$
\end{osservazione}

\begin{definizione}
Let $e_j\in \mathbb{S}^{n_1-1}$ and $\sigma_j\in (0,1)$ for any $1\leq j\leq n_1$. We say that the cones $C(e_j,\sigma_j)$ are \emph{independent} if for any choice of $v_j\in C(e_j,\sigma_j)\setminus\{0\}$, the $v_i$ are linearly independent.
\end{definizione}

\begin{osservazione}\label{rk:conistaccati}
Note that if the cones $C(e_j,\sigma_j)$ for $j=1,\ldots,n_1$ are independent, since they are by construction closed sets, there exists a $\Xi>0$ such that for any $\lambda_1,\ldots,\lambda_{n_1}\in\R\setminus \{0\}$ we have: $$\lVert \sum_{i=1}^{n_1}\lambda_i v_i\rVert>\Xi\max_{1\leq i \leq n_1}\lVert \lambda_i v_i\rVert.$$
\end{osservazione}

\begin{proposizione}\label{prop:decBate}
Assume $\mu$ is a Radon measure on $\GG$ with the Pansu property. Then, there are countably many disjoint Borel subsets $U_i$ that cover $\mu$-almost all $\GG$ such that for any $e\in \mathbb{S}^{n_1-1}$ and any $i\in\N$ there is a family of measures $t\mapsto \mu_t^i$ satisfying the hypothesis (a) and (b) of Definition \ref{s-measint} and such that:
\begin{itemize}
\item[(\hypertarget{ibate}{i})] for almost every $t\in I$ there exists a bi-Lipschitz curve $\gamma_t^i$ defined on a compact set $K_t^i$ of $\R$ such that $\mu_t^i\ll\Haus^1\trace \im(\gamma_t^i)$ and:
$$\mu\trace U_i=\int \mu_t^i\,dt.$$
    \item[(\hypertarget{iibate}{ii})]for almost every $t\in I$ and almost every $s\in K_t^i$ we have $D\gamma_{t}^i(s)=(\pi_1\circ \gamma_t^i)^\prime(s)\in C(e,\varepsilon)$.
\end{itemize}
\end{proposizione}

\begin{proof}
Proposition \ref{prop:lipdiffspace} shows that $(\GG,d_c,\mu)$ is a Lipschitz differentiability space and thus thanks to \cite[Theorem 6.6]{Bate} for any fixed $R\in\N$ there are countably many disjoint Borel sets $U_i$ that cover $\mu$-almost all of $B(0,R)$ such that for every $i\in\N$ there are $n_1$ independent cones $C_i^j$ and families of measures $\mathcal{A}_i^j=\{\mu_{i,t}^j:t\in I\}$ satisfying the hypothesis (a) and (b) of Definition \ref{s-measint} such that: 
\begin{itemize}
\item[($\alpha$)] for almost every $t\in I$ there exists a bi-Lipschitz curve $\gamma_{i,t}^j:K_{i,t}^j\to\GG$ defined on a compact set $K_{i,t}^j$ of $\R$ such that $\mu_{i,t}^j\ll\Haus^1\trace \im(\gamma_{i,t}^j)$ and:
\begin{equation}
    \mu\trace U_i=\int \mu_{i,t}^j\,dt.
    \label{identitdec}
\end{equation}
    \item[($\beta$)]for almost every $t\in I$ and almost every $s\in K_{i,t}^j$ we have $(\pi_1\circ\gamma_{i,t}^j(s))'\in C_i^j$.
\end{itemize}
Let us observe that since the $U_i$s are disjoint and contained in $B(0,R)$, thanks to \eqref{identitdec} we have that for any $i\in\N$, any $j=1,\ldots,n_1$ for almost every $t$ we have $\im(\gamma_{i,t}^j)\subseteq B(0,R)$.
Let us fix an $j\in\{1,\ldots,n_1\}$ and a $i\in\N$ and a let $t\in I$ be such that ($\alpha$) and ($\beta$) hold. 
Then, thanks to Lemmas \ref{lemma.monti1} and \ref{lemma.monti2}, we infer that:
\begin{equation}
    \frac{\lvert (\gamma_{i,t}^j)^\prime(s)\rvert}{\sup_{x\in B(0,R)}\lVert \mathscr{C}(x)\rVert}=\frac{\big\lvert\mathscr{C}(\gamma_{i,t}^j(s))[D\gamma_{i,t}^j(s)]\big\rvert}{\sup_{x\in B(0,R)}\lVert \mathscr{C}(x)\rVert}\leq\lvert D\gamma_{i,t}^j(s)\rvert=(\pi_1\circ\gamma_{i,t}^j(s))',
\end{equation}
where the last identity comes from Remark \ref{ossnice}.
In the language of \cite{Bate}, this means that the decompositions $\mathcal{A}_j^i$ have $\pi_1$-speed bigger than $\sup_{x\in B(0,R)}\lVert \mathscr{C}(x)\rVert^{-1}$. Therefore, the arbitrariness of $R$, Remarks \ref{rkstruct} and \ref{rk:conistaccati} together \cite[Theorem 9.5]{Bate} conclude the proof of the proposition.
\end{proof}

\begin{proposizione}\label{prop:decomposizionei}
Assume $\mu$ is a Radon measure on $\GG$ with the Pansu property. Then:
\begin{itemize}
    \item[(i)] $V(\mu,x)=\mathbb{G}$ for $\mu$-almost every $x\in\mathbb{G}$,
    \item[(ii)] there are countably many disjoint Borel sets $\{U_i\}_{i\in\N}$ of $\GG$ that cover $\mu$-almost all $\GG$ and such that for  any $j=1,\ldots,n_1$ we can find a $1$-dimensional horizontal normal current $\mathbf{T}_{i,j}=\tau_{i,j}\eta_{i,j}$ with $\partial \mathbf{T}_{i,j}=0$ and such that $\mu\trace U_i\ll\eta_{i,j}$ and 
    $$\tau_{i,j}(x)=\mathscr{C}(x)[e_j]\qquad\text{for $\mu$-almost every $x\in U_i$,}$$
    where as usual $\{e_1,\ldots,e_{n_1}\}$ denotes an orthonormal basis of $V_1$.
\end{itemize}
\end{proposizione}

\begin{proof}
% Item (i) is an immediate consequence of Proposition \ref{prop:lipdiffspace} and \cite[Corollary 10.5]{Bate}. 
In order to prove items (i) and (ii), let us note that Propositions \ref{generato}, \ref{prop:V=carnot}, \ref{prop:decBate}  together with Remark \ref{rkfbeit}
imply that $V(\mu\trace U_i,x)=\GG$ for $\mu$-almost every $x\in\GG$ and any $i$. Proposition \ref{costruzionecampi} and Theorem \ref{theoAu=Dec}
immediately imply the existence of the currents $\mathbf{T}_{i,j}$.
\end{proof}

We are now ready to prove the main result, which states that the existence of $n_1$ independent representations for a Radon measure $\mu$ in a Carnot group $\mathbb{G}$ implies that $\mu$ is diffuse. This is the analogue of \cite[Corollary 1.12]{DPR} and the proof follows the same overall strategy of \cite[Theorem 1.1]{DPR}, which was in turn inspired by the strong constancy lemma of Allard \cite{All76}. As explained in the introduction, we have however to adapt the proof to the ``hypoelliptic setting''. As an additional difficulties, we note that in this context we can not rely on a Besicovicth covering theorem and some classical Lebesgue point arguments need to be adapted. For the sake of readability we report these proofs in the appendix.

\begin{proposizione}\label{propofond}
Suppose $\mu$ is a Radon measure on $\mathbb{G}$ satisfying item (ii) of Proposition \ref{prop:decomposizionei}. Then $\mu\ll\mathcal{L}^n$. 
\end{proposizione}

An immediate consequence of the above proposition is our main result

\begin{teorema}
Let $\mathbb{G}, \mathbb{H}$ be Carnot groups. Suppose further that $\mu$ is a Radon measure on $\mathbb{G}$ with the Pansu property with respect to $\HH$. Then $\mu\ll \mathcal{L}^n$.
\end{teorema}

\begin{proof}
The claim follows immediately from Propositions \ref{propofond}, \ref{prop:decomposizionei} and Remark \ref{rk:pansuvsreale}.
\end{proof}

\begin{proof}[Proof of Proposition \ref{propofond}]
% Proposition \ref{prop:decomposizionei} tells us that $\mu$ is asymptotically doubling and that we can assume without loss of generality that for  any $j=1,\ldots,n_1$ we can find a $1$-dimensional horizontal normal current $T_{j}=\tau_{j}\eta_{j}$ with $\partial T_{j}=0$ such that $\mu\ll\eta_{j}$ and 
%     $$\tau_{j}(x)=\mathscr{C}(x)[e_j]\qquad\text{for $\mu$-almost every $x\in \GG$,}$$
%     where as usual $\{e_1,\ldots,e_{n_1}\}$ denotes an orthonormal basis of $V_1$.
It is clear that if for any $i\in\N$ we prove that $\mu\llcorner U_i$ is absolutely continuous with respect to $\mathcal{L}^n$ the conclusion follows as the $U_i$ are disjoint. Therefore, in the following we assume without loss of generality that 
\begin{itemize}
    \item[] for  any $j=1,\ldots,n_1$ we can find a $1$-dimensional horizontal normal current $\mathbf{T}_{j}=\tau_{j}\eta_{j}$ with $\partial \mathbf{T}_{j}=0$ and such that $\mu\ll\eta_{j}$ and 
    $\tau_{j}(x)=\mathscr{C}(x)[e_j]\text{ for $\mu$-almost every $x\in \GG$.}$
\end{itemize}

Thanks to Remark \eqref{remark:rappresentazione}, we can think to $\mathbf{T}_j$ as a vector-valued measure $\mathbf{T}_j\in\mathcal{M}(\mathbb{G},\R^{n_1})$ acting by duality with the scalar product of $\R^{n_1}$ on the smooth function $\omega\in C^\infty(\mathbb{G},\R^{n_1})$ and the boundary operator $\partial$ on these measures acts as shown in \eqref{rappre1} and \eqref{rappre2}. Thus, the currents $\mathbf{T}_1,\ldots,\mathbf{T}_{n_1}$ above can be written in this notation as $\mathbf{T}_j=e_j\eta_j$ for any $j=1,\ldots,n_1$.

Throughout the proof, we define on the measures $\boldsymbol\nu=(\boldsymbol\nu_1,\ldots,\boldsymbol\nu_{n_1})\in\mathcal{M}(\mathbb{G},\R^{n_1\times n_1})$, i.e. the Radon measures taking values in $\R^{n_1\times n_1}$, the differential operator 
$\mathfrak{B}\boldsymbol\nu:=(\partial \boldsymbol\nu_1,\ldots,\partial \boldsymbol\nu_{n_1})$, or more precisely for any test function $\varphi:=(\varphi_1,\ldots,\varphi_{n_1})\in C^\infty(\mathbb{G},\R^{n_1}\times\R^{n_1})$  
$$\langle \mathfrak{B}\boldsymbol\nu,\varphi\rangle=(\langle \boldsymbol \nu_1,d_H\varphi_1\rangle,\ldots,\langle \boldsymbol\nu_{n_1},d_H\varphi_{n_1}\rangle).$$
In the above notations, defined $\mathbb{T}:=(\mathbf T_1,\ldots,\mathbf T_{n_1})$ we have that $\mathfrak{B}\mathbb T=0$.
We can write $\mathbb{T}$ as
$$\mathbb{T}=(\tau_1\eta_1,\ldots, \tau_{n_1}\eta_{n_1})=\mathfrak{T}\Xi=\mathfrak{T}\Xi^a+\mathfrak{T}\Xi^s,$$
where $\mathfrak{T}\in\R^{n_1\times n_1}$ is a Borel map in $L^1(\Xi)$ such that $\lvert \mathfrak{T}\rvert=1$\footnote{Given an $n_1\times n_1$ matrix $A$ the norm $\lvert A\rvert$ is computed as follows. Denoted by $a_1,\ldots,a_{n_1}\in\R^{n_1}$ the columns of $A$ we let $\lvert A\rvert^2:=n_1^{-1}\sum_{i=1}^{n_1}\lvert a_i\rvert^2$, where $\lvert a_i\rvert$ denotes the usual Euclidean norm of the vectors $a_i$.} for $\Xi$-almost every $x\in\mathbb{G}$, $\Xi^a\ll\mathcal{L}^n$ and $\Xi^s$ is mutually singular with respect to $\mathcal{L}^n$. Note that since for any $j=1,\ldots,n_1$ we have $\Xi\geq \eta_j$ and this shows in particular that $\mu\ll\Xi$. 
% Hence, if we prove that $\Xi\ll \mathcal{L}^n$ the proof of the proposition is achieved.

\medskip

Let us assume by contradiction that there exists an $x_0\in\supp (\Xi^s)$ for which there exists an infinitesimal sequence $r_k$ such that
\begin{enumerate}
\item[(i)]$\displaystyle
\lim_{k\to\infty}\fint_{B(x_0,r_k)} \left| \mathfrak{T}(x)-\mathfrak{T}(x_0)\right|d \Xi(x) = 0 ;
$
\item[(\hypertarget{hlrpproofii}{ii})] $\displaystyle \limsup_{k\to\infty}\frac{\Xi(B(x_0,r_k/5))}{\Xi(B(x_0,r_k))} \geq \frac{1}{j_0}$ for some $j_0 \in \N$;
\item[(iii)] $\displaystyle\lim_{k\to\infty}\frac{\Xi^a(B(x_0, r_k ))}{\Xi^s(B(x_0,r_k))} = 0$;
\item[(iv)] $ P_0:=\mathfrak{T}(x_0)=\mathrm{diag}(\kappa_1,\ldots,\kappa_{n_1})$ for some $\kappa_1,\ldots,\kappa_{n_1}\in\R\setminus \{0\}$.
\end{enumerate}
First of all, let us show that the hypothesis (i) to (iii) are satisfied on a set of full $\Xi^s$-full measure. It follows immediately from Proposition~\ref{p:differentiation} that (i) holds $\Xi$-almost everywhere. Further, it is showed in the proof of  Proposition \ref{p:differentiation} that (ii) holds $\Xi$-almost everywhere. Finally, thanks to Proposition \ref{p:differentiation} we know that 
\begin{equation}
    \lim_{k\to\infty}\frac{\Xi^a(B(x_0, r_k ))}{\Xi(B(x_0,r_k))} = 0,
    \label{eq:limite0}
\end{equation}
 for $\Xi$-almost every $x\in \mathbb{G}$, and thus $\Xi^s$-almost everywhere,
with the choice $f:=\chi_{\supp(\Xi^a)}$. In the following, we will prove that this is in contradiction with (iv) at every such $x_0$.
Define the normalized blow-up sequence
\begin{equation}
   \boldsymbol \nu_k : = \frac{1}{\Xi(B(x_0,r_k))}  T_{x_0,r_k}\mathbb{T},  \qquad\text{for any } k \in \N.
\label{definizionenuk}
\end{equation}
and note that $\lVert \boldsymbol \nu_k\rVert(B(0,1) )=1$ and $\lVert\boldsymbol \nu_k\rVert(B(0,1/5) ) \geq j_0^{-1} > 0$.
Up to the extraction of a subsequence, by (ii) we can assume that 
\begin{equation}\label{eq:limitnu}
\lVert \boldsymbol \nu_k\rVert\rightharpoonup\nu \quad \text{in $\mathcal{M}_+(B(0,1))$}
\end{equation}
with $\nu(B(0,1) )\leq 1$ and $\nu(B(0,1/5) ) \geq j_0^{-1}$.
The vector fields $X_i$ are left-invariant, and thus their (formal) adjoints coincide with $-X_j$, and this implies that
\begin{equation}
\begin{split}
      \mathfrak{B}(\boldsymbol \mu_1,\ldots,\boldsymbol \mu_{n_1})=-(\sum_{i=1}^{n_1}X_i(\boldsymbol \mu_1^i),\ldots,\sum_{i=1}^{n_1}X_i(\boldsymbol \mu_{n_1}^i)),
      \label{espressionepermu}
\end{split}
\end{equation}
where $\boldsymbol\mu_j^i$ denotes the $i$th entry of $\mu_j$.
In addition, since the vector fields $X_j$ are homogeneous, we also have
$  X_j\varphi(\delta_{1/r}(x^{-1}y)) = rX_j(\varphi(\delta_{1/r}(x^{-1}\cdot)))(y)$ for all $j=1,\ldots,n_1$, all smooth functions $\varphi$ and all $r > 0$. This, together with an elementary computation shows in particular that $\mathfrak{B}[\boldsymbol \nu_k]=0$ for any $k\in\N$. 

% {\color{RoyalBlue}
% Indeed for any $\varphi\in C^\infty_c(\mathbb{G},\R^{n_1})$ we have
% \begin{equation}
%   \begin{split}
%           &-\langle X_i(\frac{1}{\Xi(B(x_0,r_k))}  T_{x_0,r_k}\mathbb{T}),\varphi\rangle= \frac{\langle T_{x_0,r_k}\mathbb{T},X_i\varphi\rangle}{\Xi(B(x_0,r_k))}  =\frac{(\langle T_{x_0,r_k}(\tau_1\eta_1), X_i\varphi\rangle,\ldots,\langle T_{x_0,r_k}(\tau_{n_1}\eta_{n_1}),X_i\varphi\rangle)}{\Xi(B(x_0,r_k))}\\
%   =&\frac{\Big(\int\langle \tau_1(y),X_i\varphi(\delta_{r_k^{-1}}(x_0^{-1}y))\rangle d\eta_1(y), \ldots,\int\langle \tau_{n_1}(y),X_i\varphi(\delta_{r_k^{-1}}(x_0^{-1}y))\rangle d\eta_{n_1}(y) \Big)}{\Xi(B(x_0,r_k))}\\
%   =&r_k\frac{\Big(\int\langle \tau_1(y),X_j(\varphi(\delta_{1/r_k}(x^{-1}\cdot)))(y)\rangle d\eta_1(y), \ldots,\int\langle \tau_{n_1}(y),X_j(\varphi(\delta_{1/r_k}(x^{-1}\cdot)))(y)\rangle d\eta_{n_1}(y) \Big)}{\Xi(B(x_0,r_k))}.
%   \end{split}
% \end{equation}
% The above identity thus implies that 
% \begin{equation}
%     \begin{split}
% &\qquad \qquad\qquad\langle\mathfrak{B}\nu_k,\varphi\rangle=-\sum_{i=1}^{n_1}\langle X_i(\frac{1}{\Xi(B(x_0,r_k))}  T_{x_0,r_k}\mathbb{T}),\varphi\rangle\\
% =&r_k\frac{\Big(\langle \partial (\tau_1\eta_1),\varphi(\delta_{1/r_k}(x^{-1}\cdot)))(y)\rangle, \ldots,\langle \partial (\tau_{n_1}\eta_{n_1}),\varphi(\delta_{1/r_k}(x^{-1}\cdot)))(y)\rangle\Big)}{\Xi(B(x_0,r_k))}=0.
%     \end{split}
% \end{equation}
% and thus $\mathfrak{B}\boldsymbol\nu_k=0$.
% }
Then $\mathfrak{B}[P_0\lVert \boldsymbol\nu_k\rVert]=-P_0^T[\nabla_\mathbb{G}\lVert \boldsymbol\nu_k\rVert]$
where $\nabla_\mathbb{G}:=(X_1,\ldots,X_{n_1})$.
On the other hand 
\begin{equation}
    \mathfrak{B}[P_0\lVert \boldsymbol\nu_k\rVert]=\mathfrak{B}[P_0\lVert \boldsymbol\nu_k\rVert-\boldsymbol\nu_k]+\mathfrak{B}\boldsymbol\nu_k=\mathfrak{B}[P_0\lVert \boldsymbol\nu_k\rVert-\boldsymbol\nu_k].
\end{equation}
Let $\Phi$ be a smooth positive function supported on $B(0,1)$ such that $\int \Phi d\mathcal{L}^n=1$ and $\Phi(\cdot)=\Phi(\cdot^{-1})=\Phi(-\cdot)$. Let $\{\varepsilon_k\}_{k\in\N}$ be an infinitesimal sequence of positive real numbers to be fixed later, let 
$\Phi_{\varepsilon_k}(\cdot):=\varepsilon_k^{-\mathcal{Q}}\Phi(\delta_{1/\varepsilon_k}(\cdot))$ and define
\begin{align*}
%   q_\alpha &:= A_\alpha P_0 = A_\alpha  \frac{d \mu}{d \lVert \mu\rVert}(x_0) \in \R^n, \\
  u_k &:= \Phi_{\varepsilon_k} * \lVert \boldsymbol\nu_k\rVert \in C^\infty(B(0,1)), 
\\  
  V_k &:= \Phi_{\varepsilon_k} * \bigl[ P_0 \lVert \boldsymbol\nu_k\rVert - \boldsymbol\nu_k \bigr] \in C^\infty(B(0,1),\R^{n_1\times n_1}),
\end{align*}
where here $*$ denotes the convolution with respect to the group law of $\mathbb{G}$, i.e. $f*g:=\int f(xy^{-1})g(y)d\mathcal{L}^n(y)$. It will be clear from the context when $*$ denotes a convolution and when it denotes the group law of $\mathbb{G}$.
Then, if we let $\chi\in\mathcal{C}^\infty(\mathbb{G},[0,1])$ be such that $\chi=1$ on $B(0,1/2)$ and $\chi=0$ on $B(0,3/4)^c$, we infer from the above discussions that
\begin{equation}
    \begin{split}
        &-P_0^T[\nabla_\mathbb{G}(\chi u_k)]=\mathfrak{B}[P_0\chi u_k]=-u_kP_0^T[\nabla_\mathbb{G}\chi]-\chi P_0^T[\nabla_\mathbb{G}u_k]=-u_kP_0^T[\nabla_\mathbb{G}\chi]-\chi P_0^T[\nabla_\mathbb{G}(\Phi_{\varepsilon_k}*\lVert\boldsymbol\nu_k\rVert)]\\
        =&-u_kP_0^T[\nabla_\mathbb{G}\chi]-\chi P_0^T[\Phi_{\varepsilon_k}*\nabla_\mathbb{G}\lVert \boldsymbol\nu_k\rVert]=-u_kP_0^T[\nabla_\mathbb{G}\chi]-\chi \Phi_{\varepsilon_k}*P_0^T[\nabla_\mathbb{G}\lVert \boldsymbol\nu_k\rVert]\\
        =&-u_kP_0^T[\nabla_\mathbb{G}\chi]+\chi \Phi_{\varepsilon_k}*\mathfrak{B}[P_0\lVert \boldsymbol\nu_k\rVert]=-u_kP_0^T[\nabla_\mathbb{G}\chi]+\chi \Phi_{\varepsilon_k}*\mathfrak{B}[P_0\lVert \boldsymbol\nu_k\rVert-\boldsymbol\nu_k]+\chi \Phi_{\varepsilon_k}*\mathfrak{B}[\boldsymbol\nu_k]\\
        =&-u_kP_0^T[\nabla_\mathbb{G}\chi]+\chi \Phi_{\varepsilon_k}*\mathfrak{B}[P_0\lVert \boldsymbol\nu_k\rVert-\boldsymbol\nu_k].
        \label{eq:identityoperator}
    \end{split}
\end{equation}
Thanks to \eqref{espressionepermu} and to the fact that for any $i=\{1,\ldots,n_1\}$ we have $\langle X \psi_1,\varphi\rangle=-\langle \psi_1, X \varphi\rangle$ and $X(\psi_1*\psi_2)=\psi_1*X\psi_2$ for any distribution $\psi_1,\psi_2$ and any test function $\varphi$. It is possible to prove that 
$$\Phi_{\varepsilon_k}*\mathfrak{B}[P_0\lVert \boldsymbol\nu_k\rVert-\boldsymbol\nu_k]=\mathfrak{B}[\Phi_{\varepsilon_k}*(P_0\lVert \boldsymbol\nu_k\rVert-\boldsymbol\nu_k)],$$
and hence \eqref{eq:identityoperator} can be rewritten as
\begin{equation}
    -P_0^T[\nabla_\mathbb{G}(\chi u_k)]=-u_kP_0^T[\nabla_\mathbb{G}\chi]+\chi \mathfrak{B}[V_k]=-u_kP_0^T[\nabla_\mathbb{G}\chi]-V_k\mathfrak{B}[\chi] +\mathfrak{B}[\chi V_k].
    \label{eq:identitadifferenziale}
\end{equation}
Define $R_k:=-u_kP_0^T[\nabla_\mathbb{G}\chi]-V_k\mathfrak{B}[\chi]$ and 
let us apply to both sides of \eqref{eq:identitadifferenziale}
the differential operator $-\nabla^T_\mathbb{G} P_0$, to obtain 
$$\nabla^T_\mathbb{G} P_0P_0^T\nabla_\mathbb{G}[\chi u_k]=-\nabla^T_\mathbb{G}P_0[\mathfrak{B}[\chi V_k]+R_k]=-\nabla^T_\mathbb{G}P_0[\mathfrak{B}[\chi V_k]]-\nabla^T_\mathbb{G}P_0[R_k].$$
The matrix $\Gamma:=P_0^TP_0=\mathrm{diag}(\kappa_1^2,\ldots,\kappa_{n_1}^2)$ is positively definite and  diagonal.
 Therefore, the operator $\nabla^T_\mathbb{G} P_0P_0^T\nabla_\mathbb{G}$ can thus be rewritten as
% $\nabla^T_\mathbb{G} P_0P_0^T\nabla_\mathbb{G}=\nabla^T_\mathbb{G} \mathbb{O}\tilde{\Delta}\mathbb{O}^T\nabla_\mathbb{G}$. Let us denote by $\mathbb{O}_j$ be for any $j=1,\ldots,n_1$ the $j$th line of the matrix $\mathbb{O}$ and let us denote with $Y_j$ the vector fields $Y_j:=\langle \mathbb{O}_j,\nabla_\mathbb{G}\rangle$.  
% The vector fields $Y_j$ are clearly horizontal and they are indipendent. 
% In particular the operator $\nabla^T_\mathbb{G} P_0P_0^T\nabla_\mathbb{G}$ can be rewritten as
$$\mathfrak{D}:=\nabla^T_\mathbb{G} P_0P_0^T\nabla_\mathbb{G}=\sum_{i=1}^{n_1}\kappa_i^2 X_i^2=\sum_{i=1}^{n_1}(\lvert\kappa_i\rvert X_i)^2.$$
% Therefore, \cite[Theorem 1.1]{MR222474}
% implies that $\mathfrak{D}$ is an homogeneous, left-infariant hypoelliptic operator of degree $2$. It is also well known that $\mathfrak{D}$ admits a fundamental solution $K_0$ which is homogeneous of degree $2-Q$, see for instance \cite[Theorem 40]{MR3154431}.
It is well known that, see for instance \cite[Proposition 5.3.2]{equivmetr} and \cite[Proposition 5.3.11]{equivmetr}, since $\mathfrak{D}$ is a sub-Laplacian, that $\mathfrak{D}$ admits a fundamental solution $K_0$  such that  $K_0\in\mathcal{C}^{\infty}(\mathbb{G}\setminus \{0\})$ and that $K_0(x)=K_0(-x)$. Let us first prove that 
\begin{equation}
    0\leq \chi u_k=-\nabla^T_\mathbb{G}P_0[\mathfrak{B}[\chi V_k]]*K_0-\nabla^T_\mathbb{G}P_0[R_k]*K_0=\mathcal{L}_1[\chi V_k]*K_0+\mathcal{L}_2[R_k]*K_0=:f_k+g_k,
    \label{equationfondamentaleinvertita}
\end{equation}
 where we note that the convolutions above must me intended in the distributional sense and that they are well defined since both $-\nabla^T_\mathbb{G}P_0[\mathfrak{B}[\chi V_k]]$ and $-\nabla^T_\mathbb{G}P_0[R_k]$ have compact support. 
% Since $K_0$ is homogeneous of degree $2-Q$ and it is smooth away from the origin, \cite[Proposition 1.15]{folland75} implies that there is a constant $C>0$ such that $$\lvert K_0(xy)-K_0(x)\rvert\leq C\lVert y\rVert\lVert x\rVert^{1-Q}\qquad \text{for any }\lVert y\rVert\leq \lVert x\rVert/2.$$

Let us check that the sequence $\chi V_k$ converges to $0$ in $L^1(B(0,1/2))$. Indeed thanks to the choice of $\chi$ we have
\begin{equation}
\begin{split}
      &\int  \chi(y)\lvert V_k(y)\rvert d\mathcal{L}^n(y)\leq  \int_{B(0,3/4)}  \lvert V_k(y)\rvert d\mathcal{L}^n(y)\\
      =&\int_{B(0,3/4)} \lvert \Phi_{\varepsilon_k}*( P_0\lVert \boldsymbol\nu_k\rVert-\boldsymbol\nu_k)\rvert(y) d\mathcal{L}^n(y)\leq  \int_{B(0,1)} \lvert P_0-\mathfrak{T}_k(y)\rvert d\lVert \boldsymbol\nu_k\rVert( y).
      \label{convV_ka0}
\end{split}
\end{equation}
where $\boldsymbol\nu_k=\mathfrak{T}_k\lVert \boldsymbol\nu_k\rVert$. Note that the arbitrariness of $\chi$ imply that even $\sqrt{\chi}V_k$ converges to $0$ in $L^1(B(0,3/4))$. On the other hand, recalling the definition of $\boldsymbol\nu_k$ in \eqref{definizionenuk}, the \eqref{convV_ka0} boils down to
\begin{equation}
\begin{split}
    \lim_{k\to \infty}\int  \chi(y)\lvert V_k(y)\rvert d\mathcal{L}^n(y)\leq& \lim_{k\to\infty}\frac{\int_{B(0,1)} \lvert P_0-\mathfrak{T}(x_0\delta_{r_k}(y))\rvert dT_{x_0,r_k}\Xi(y)}{\Xi(B(x_0,r_k))}\\
    \leq &\lim_{k\to\infty}\frac{\int_{B(x,r_k)}\lvert P_0-\mathfrak{T}(z)\rvert d\Xi(z)}{\Xi(B(x_0,r_k))}=0,
        \nonumber
\end{split}
\end{equation}
which shows that $\chi V_k\to 0$ in $L^1(\mathbb{G})$.
Let us now give a uniform upper bound on the $L^1(\mathbb{G},\R^{n_1})$ norm of the functions $R_k$. It is easy to see that 
\begin{equation}
  \begin{split}
      &\int \lvert R_k\rvert d\mathcal{L}^n \leq  \int_{B(0,3/4)} u_k\lvert P_0^T[\nabla_\mathbb{G}\chi]\rvert+\int_{B(0,3/4)} \lvert V_k\rvert \lvert \mathfrak{B}[\chi]\rvert d\mathcal{L}^n\\
      \leq\lVert P_0^T[\nabla_\mathbb{G}\chi] \rVert_\infty \int_{B(0,3/4)}& u_kd\mathcal{L}^n+ \lVert\mathfrak{B}[\chi]\rVert_\infty\int_{B(0,3/4)}\lvert V_k\rvert d\mathcal{L}^n\leq \lVert P_0^T[\nabla_\mathbb{G}\chi] \rVert_\infty + \lVert\mathfrak{B}[\chi]\rVert_\infty\int_{B(0,3/4)}\lvert V_k\rvert d\mathcal{L}^n.
      \nonumber
  \end{split}
\end{equation}
Since we showed above that $V_k\to 0$ in $L^1(B(0,3/4))$, this shows that 
\begin{equation}
    \sup_{k\in\N} \lVert R_k\rVert_{L^1(\mathbb{G},\R^{n_1})}\leq \lVert P_0^T[\nabla_\mathbb{G}\chi] \rVert_\infty + \lVert\mathfrak{B}[\chi]\rVert_\infty.
    \label{eq:boundL1onRk}
\end{equation}
Let us prove that the sequence $\mathcal{L}_2[R_k]*K_0$ is precompact in $L^1_{\mathrm{loc}}(\mathbb{G},\R^{n_1})$. As a first step, let us write the action of  $\mathcal{L}_1[R_k]*K_0$ on test functions in a more explicit way
\begin{equation}
    \langle \mathcal{L}_1[R_k]*K_0,\varphi\rangle= \langle \mathcal{L}_1[R_k],\varphi*K_0^\vee\rangle=\langle \mathcal{L}_1[R_k],\varphi*K_0\rangle=\sum_{i,j=1}^{n_1}\langle R_k^j*(X_iK_0)^\vee,\varphi\rangle=-\sum_{i,j=1}^{n_1}\langle R_k^j*(X_iK_0),\varphi\rangle
\end{equation}
where we denoted by $R_k^j$ the $j$th component of the vector $R_k$, and we used repeatedly the fact that $K_0=K_0^\vee$, where $\Psi^\vee$ denotes the distribution that acts as $\langle\Psi^\vee,\varphi\rangle =\langle\Psi, \varphi(-\cdot)\rangle$. Since $X_iK_0$ is an $(1-Q)$-homogeneous distribution, we can represent its action on test functions by integration. Let us prove that for any $\rho>0$ the functions $R_k^j*(X_iK_0)$ are equicontinuous in $L^1(B(0,\rho))$.
For any smooth function $u$ with compact support in $B(0,1)$ we have 
\begin{equation}
    \begin{split}
  \lVert u*(X_iK_0)(\cdot*h)-u*(X_iK_0)(\cdot)\rVert_{L^1(B(0,\rho))}    =&\int_{B(0,\rho)} \lvert u*(X_iK_0)(y*h)-u*(X_iK_0)(y)\rvert d\mathcal{L}^n(y) \\
  \leq &\int_{B(0,\rho)} \Big\lvert \int_{B(0,\rho)} u(z)\Big((X_iK_0)(z^{-1}y)-(X_iK_0)(z^{-1}yh)\Big) d\mathcal{L}^n(z)\Big\rvert d\mathcal{L}^n(y)\\
  \leq& \int_{B(0,\rho)} \int_{B(0,\rho)} \lvert u(z)\rvert \lvert (X_iK_0)(z^{-1}y)-(X_iK_0)(z^{-1}yh)\rvert d\mathcal{L}^n(z) d\mathcal{L}^n(y)\\
  \leq& \int_{B(0,\rho)} \lvert u(z)\rvert\Big(\int_{B(0,\rho)}  \lvert (X_iK_0)(z^{-1}y)-(X_iK_0)(z^{-1}yh)\rvert d\mathcal{L}^n(y)\Big)d\mathcal{L}^n(z) 
    \end{split}
\end{equation}
Let us study the inner integral above. Let $\eta:=\lVert h\rVert$ and note that
\begin{equation}
    \begin{split}
        &\qquad\qquad\qquad\qquad\qquad\int_{B(0,\rho)}  \lvert (X_iK_0)(z^{-1}y)-(X_iK_0)(z^{-1}yh)\rvert d\mathcal{L}^n(y)\\
        \leq &  \underbrace{\int_{\lVert z^{-1}y\rVert\leq 2\eta } \lvert (X_iK_0)(z^{-1}y)-(X_iK_0)(z^{-1}yh)\rvert d\mathcal{L}^n(y)}_{\textrm{(I)}}+\underbrace{\int_{\lVert z^{-1}y\rVert> 2\eta,\lVert y\rVert\leq \rho} \lvert (X_iK_0)(z^{-1}y)-(X_iK_0)(z^{-1}yh)\rvert d\mathcal{L}^n(y)}_{\textrm{(II)}}
        \nonumber
    \end{split}
\end{equation}
In order to estimate (I) it suffices to recall that $X_iK_0$ is $(1-Q)$-homogeneous
\begin{equation}
    \lvert \mathrm{(I)}\rvert\leq \int_{\lVert z^{-1}y\rVert\leq 2\eta} \lvert (X_iK_0)(z^{-1}y)\rvert d\mathcal{L}^n(y)+\int_{\lVert z^{-1}yh\rVert\leq 3\eta} \lvert (X_iK_0)(z^{-1}yh)\rvert d\mathcal{L}^n(y)
    \leq 20\eta\sup_{\lVert p\rVert=1}\lvert X_iK_0(p)\rvert,
    \nonumber
\end{equation}
On the other hand, smooth functions are locally Lipschitz with respect to the intrinsic distance on $\mathbb{G}$ and thanks to the fact that $X_iK_0$ is $(1-Q)$-homogeneous it is not hard to see that there exists a constant $C>0$ such that $\lvert f(a)-f(ah)\rvert\leq C\lVert a\rVert^{-Q}\eta$. 

This allows us to prove 
\begin{equation}
    \begin{split}
      \lvert \mathrm{(II)}\rvert\leq C\eta\int_{2\eta<\lVert z^{-1}y\rVert<\rho+\lVert y\rVert} \lVert z^{-1}y\rVert^{-Q} d\mathcal{L}^{n}(y)\leq C\eta(\rho+\lVert z\rVert-2\eta).
    \end{split}
\end{equation}
Summing up what we have shown above, we infer that 
$$\lVert u*(X_iK_0)(\cdot*h)-u*(X_iK_0)(\cdot)\rVert_{L^1(B(0,\rho))} \leq (20\sup_{\lVert p\rVert=1}\lvert X_iK_0(p)\rvert+2C\rho)\lVert h\rVert\lVert u\rVert_{L^1(B(0,\rho))}.$$
The above bound, together with \eqref{eq:boundL1onRk} conclude the proof of the fact, by choosing suitably diagonal subsequences, that the sequence $g_k$ is precompact by Kolmogorov-Riesz-Frechet theorem.

As a first step, let us show that $f_k$ converges to $0$ in the weak $L^1$ space $L^{1,\infty}$ and as distribution they converge to $0$ in the weak* topology. 
We can rewrite $f_k$ in the following way
$$\langle f_k,\varphi\rangle=\langle\mathcal{L}_1[\chi V_k]*K_0,\varphi\rangle=\langle\chi V_k,\mathcal{L}^*_1[\varphi*K_0]\rangle, $$
where $\mathcal{L}^*_1$ is the adjoint operator of $\mathcal{L}^1$. In addition, since $\chi V_k$ converge to $0$ in $L^1$, and thus it is immediate that $\lim_{k\to 0}\langle\chi V_k,\mathcal{L}^*_1[\varphi*K_0]\rangle=0$ for any test function $\varphi$. Hence
$\lim_{k\to \infty}\langle f_k,\varphi\rangle=0$,
and thus by definition of weak* convergence of distributions we conclude that $f_k\overset{*}{\rightharpoonup}0$. In addition we can further rewrite the action of $f_k$ on text functions as
\begin{equation}
    \begin{split}
        \langle \mathcal{L}_1[\chi V_k]*K_0,\varphi\rangle=\langle \mathcal{L}_1[\chi V_k],\varphi*K_0^\vee\rangle=\langle-\nabla^T_\mathbb{G}P_0\mathfrak{B}(\chi V_k),\varphi*K_0\rangle=\sum_{i=1}^{n_1}\sum_{j=1}^{n_1}\kappa_j\langle (\chi V_k)^j*(X_iX_jK_0), \varphi\rangle
        \nonumber
    \end{split}
\end{equation}
where $(\chi V_k)^j$ denotes the $j$th entry of the vector valued function $\chi V_k$.
It is easily checked that the distribution $X_iX_j K_0$ is $-Q$ homogeneous and it coincides with a smooth function away from $0$. In the notations of \cite[p.164]{folland75}, the distribution $X_iX_j K_0$ is said to be a \emph{Kernel of type $0$} and by \cite[Proposition 1.8]{folland75} there is a constant $C>0$ such that  
$$X_iX_jK_0=C\delta_0+PV(X_iX_jK_0),$$
where the distribution $PV(X_iX_jK_0)$ acts on test functions $\varphi$ as 
$$\langle PV(X_iX_jK_0),u\rangle=\lim_{\epsilon\to 0}\int_{\lVert x\rVert\geq \epsilon} X_iX_jK_0(x)u(x) dx.$$
In order to see that such distribution is well defined we refer to \cite[p.166]{folland75}. 
In addition, \cite[Proposition 1.9]{folland75} tells us that the operator $T:u\mapsto u*X_iX_jK_0$ is bounded in $L^p(\mathbb{G},\R^{n_1})$ for any $1<p<\infty$ and thus the operator 
$\tilde{\mathcal{L}}:u\mapsto \mathcal{L}_1[u]*K_0$,
is, or more precisely extends to a, bounded in $L^p(\mathbb{G},\R^{n_1})$ for any $1<p<\infty$. To be precise \cite[Proposition 1.9]{folland75} gives us a little more. Indeed, defined $T_\varepsilon[u]=u*(X_iX_jK_0)^\varepsilon+Cu$ where $(X_iX_jK_0)^\varepsilon$ is the function coinciding with $X_iX_jK_0$ on $B(0,\varepsilon)^c$ and $0$ otherwise, we have that $T_\varepsilon$ are \emph{uniformly} bounded in $L^p$ for any $1<p<\infty$ and 
\begin{equation}
\lim_{\varepsilon\to 0}\lVert T_\varepsilon[u]-T[u]\rVert_{L^p(\mathbb{G})}=0\qquad \text{for any test function $u$}.
\label{existenceoflimitLp}
\end{equation}

Let us now show that the operator $u\mapsto u*X_iX_j K_0$ is of weak $(1,1)$-type. The above discussion shows that 
$$T[u]-Cu=\lim_{\epsilon\to 0} T_\epsilon[u]-Cu=\lim_{\epsilon\to 0}\int_{\lVert w\rVert\geq \epsilon}u(\cdot*w^{-1})X_iX_jK_0(w)d\mathcal{L}^n(w),$$
where the limits above have to been understood in the $L^p$ sense. This in particular implies that the operator $$u\mapsto\lim_{\epsilon\to 0}\int_{\lVert w\rVert\geq \epsilon}u(\cdot*w^{-1})X_iX_jK_0(w)d\mathcal{L}^n(w),$$
defines an operator bounded on $L^p$. 
In addition, let us note that it is possible to show that for any $\varepsilon>0$, there exists a constant $C>0$ such that for any $\lVert z^{-1}w\rVert \geq \varepsilon$ we have
$$\lvert X_iX_j K_0(w^{-1}z )- X_iX_j K_0(w^{-1}\bar z)\rvert \leq C\lVert w^{-1} z\rVert^{-(Q+1)}\lVert z^{-1}\bar z\rVert\qquad \text{for any $\lVert z^{-1}\bar z\rVert \leq \varepsilon/4$,}$$
which implies that there exists a constant $A>0$ such that 
$$\int_{\lVert w^{-1}z\rVert \geq \varepsilon}\lvert X_iX_j K_0(w^{-1} z)- X_iX_j K_0(w^{-1} \bar z)\rvert\leq A\qquad \text{for any $\lVert z^{-1}\bar z\rVert \leq \varepsilon/4$,}$$
which, thanks to the fact that the topologies induced by the Euclidean metric and the sub-Riemannian one are the same together with the argument employed in the proof of  \cite[Chapter 1, \S 5 Theorem 3 ]{stein}, allows us to conclude that the operator $T[u]-u$ is of weak $(1,1)$-type. Thus clearly $T[u]$ is of weak $(1,1)$-type as well. This together with the fact that $\chi V_k$ converge to $0$ in $L^1(\mathbb{G})$ implies that
\begin{equation}
    \lim_{k\to \infty}\lVert f_k\rVert_{L^{1,\infty}}=0.
    \label{equazionea0}
\end{equation}
Thanks to \eqref{equationfondamentaleinvertita} we know that $f_k+g_k\geq 0$ and hence we must have that $f_k^{-}:=\max{0,-f_k}\leq \lvert g_k\rvert$. However, since $g_k$ is precompact in $L^1_{loc}(\mathbb{G})$, the functions $\lvert g_k\rvert$ are locally uniformly integrable, namely for any $R>0$ and any $\varepsilon>0$ there exists a $\delta>0$ such that for any Borel set $E\subseteq B(0,R)$ such that $\mathcal{L}^n(E)<\delta$ then $\int_E f_k^-\leq \int_E\lvert g_k\rvert<\varepsilon$. 
Let us pick a test function $\varphi$ and note that 
 \begin{equation}
\begin{split}
       \lim_{k\to \infty} \int \varphi \lvert f_k\rvert d\mathcal{L}^n\leq & \lim_{k\to\infty}\int \varphi f_k d\mathcal{L}^n+2\int  f_k^- d\mathcal{L}^n=\langle \chi V_k,\mathcal{L}^*(\varphi*K_0)\rangle+2\int  f_k^- d\mathcal{L}^n \\
     \leq &\lim_{k\to\infty}\langle \sqrt{\chi} V_k,\sqrt{\chi}\mathcal{L}^*_1(\varphi*K_0)\rangle+2\int_{\{\lvert f_k\rvert>\delta\}} \varphi f_k^-d\mathcal{L}^n+\delta\lVert \varphi\rVert_{L^1(\mathbb{G})}\\
     \leq& \lim_{k\to\infty}\lVert \sqrt{\chi} V_k\rVert_{L^1(\mathbb{G})} \lVert \sqrt{\chi}\mathcal{L}^*_1(\varphi*K_0)\rVert_{L^\infty(\mathbb{G})} +2\mathcal{L}^n(\{\lvert f_k\rvert>\delta\})\lVert\varphi\rVert+2\delta\lVert \varphi\rVert_{L^1(\mathbb{G})}= \delta\lVert \varphi\rVert_{L^1(\mathbb{G})},
     \nonumber
\end{split}
\end{equation}
where $\mathcal{L}_1^*$ is the adjoint of the operator $\mathcal{L}^1$, and where the last 
identity comes from the fact that $\sqrt{\chi}V_k$ is converging to $0$ in $L^1(\mathbb{G})$ and \eqref{equazionea0}. Finally thanks to the arbitrariness of $\delta$ we conclude that $\lim_{k\to \infty} \int \varphi \lvert f_k\rvert d\mathcal{L}^n=0$. This show in particular that the sequence $f_k$ is converging to $0$ in $L^1_{loc}(\mathbb{G})$ (see \cite[Lemma 2.2]{DPR} for the same type of  argument).

Since $g_k$ is precompact in $L^1_{loc}(\mathbb{G})$, thanks to the above discussion we know that the sequence $\chi u_k=f_k+g_k$ is also precompact in $L^1_{loc}(\mathbb{G})$ and thus there exists a $v\in L^1(\mathbb{G})$ supported on $B(0,3/4)$ such that $\chi u_k\to v$
in $L^1(\mathbb{G})$.

Let us show that $\Phi_{\varepsilon_k}*[\lVert\boldsymbol\nu_k\rVert-\lVert\boldsymbol\nu_k\rVert ^s]$ converges to $0$ in $L^1(B(0,1))$. By definition, we have 
\begin{equation}
    \begin{split}
\lim_{k\to \infty}\lVert u_k-\Phi_{\varepsilon_k}*\lVert\boldsymbol\nu_k\rVert ^s\rVert_{L^1(B(0,1/2))}=&\lim_{k\to \infty}\int_{B(0,1/2)} \Phi_{\varepsilon_k}*[\lVert\boldsymbol\nu_k\rVert-\lVert\boldsymbol\nu_k\rVert ^s]d\mathcal{L}^n\leq  \lim_{k\to \infty}\lVert\boldsymbol\nu_k\rVert^a(B(0,1))\\
=&\lim_{k\to \infty}\frac{\Xi^a(B(x_0,r_k))}{\Xi(B(x_0,r_k))}\overset{\eqref{equazionea0}}{=} 0, 
\label{eq.stabilesingolare}
    \end{split}
\end{equation}
and this implies in particular that the sequence $\Phi_{\varepsilon_k}*\lVert\boldsymbol\nu_k\rVert ^s$ is precompact in $L^1_{loc}(\mathbb{G})$. In addition, we also have that $\Phi_{\varepsilon_k}*\lVert \boldsymbol\nu_k\rVert^s\rightharpoonup \nu$. Indeed, for any test function $\varphi$ supported in $B(0,1/2)$ we have
\begin{equation}
  \lim_{k\to\infty} \langle \Phi_{\varepsilon_k}*\lVert \boldsymbol\nu_k\rVert^s, \varphi \rangle\overset{\eqref{eq.stabilesingolare} }{=}\lim_{k\to\infty} \langle \Phi_{\varepsilon_k}*\lVert \boldsymbol\nu_k\rVert, \varphi \rangle= \lim_{k\to\infty}\langle \lVert \boldsymbol\nu_k\rVert, \Phi_{\varepsilon_k}*\varphi \rangle=\lim_{k\to\infty}\int (\Phi_{\varepsilon_k}*\varphi) d\lVert \boldsymbol\nu_k\rVert=\langle \nu,\varphi\rangle ,
\end{equation}
where the last identity comes from the fact that the sequence of functions $\Phi_{\varepsilon_k}*\varphi$ converges uniformly to $\varphi$. The above chain of identities also proves that 
$$\langle v,\varphi\rangle=\lim_{k\to\infty}\langle \chi u_k,\varphi\rangle=\lim_{k\to\infty} \langle \Phi_{\varepsilon_k}*\lVert \boldsymbol\nu_k\rVert, \varphi \rangle=\langle \nu,\varphi\rangle,$$
which means that on $B(0,1/2)$ the measure $\nu$ is (represented by) the $L^1(B(0,1/2))$ function $v$. 

It is now the moment to choose the sequence $\{\varepsilon_k\}_{k\in\N}$. Thanks to the lower semicontinuity of the total variation we know that 
$$\lvert \lVert\boldsymbol\nu_k\rVert^s -\nu\rvert(B(0,1/2))\leq \liminf_{\varepsilon\to 0}\lvert \Phi_\varepsilon*\lVert\boldsymbol\nu_k\rVert^s -\nu\rvert(B(0,1/2)),$$
this means that for any $k\in\N$ we can choose an $\varepsilon_k$ such that 
\begin{equation}
    \lvert \lVert\boldsymbol\nu_k\rVert^s -\nu\rvert(B(0,1/2))\leq \lvert \Phi_{\varepsilon_k}*\lVert\boldsymbol\nu_k\rVert^s -\nu\rvert +k^{-1}.
    \label{eq.simtalsc}
\end{equation}
Let $E$ be a Borel set of $\mathbb{G}$ such that $\mathcal{L}^n(E)=0$, $\Xi^s(\mathbb{G}\setminus E)=0$. Thanks to \eqref{eq:limite0} and to (\hyperlink{hlrpproofii}{ii}) we know that if $k$ is sufficiently big, we have
\begin{equation}
\begin{split}
      1/j_0\leq& \frac{\Xi^s(B(x_0,r_k/2))}{\Xi(B(x_0,r_k))}=\frac{\Xi^s(B(x_0,r_k/2)\cap E)}{\Xi(B(x_0,r_k))}=\lVert \boldsymbol\nu_k \rVert^s(B(0,1/2)\cap \delta_{1/r_k}(x_0^{-1}E))\\
      \leq& \lvert\lVert \boldsymbol\nu_k \rVert^s-  \nu\rvert(B(0,1/2)\cap \delta_{1/r_k}(x_0^{-1}E))) + \nu(B(0,1/2)\cap \delta_{1/r_k}(x_0^{-1}E))\\
      =&\lvert\lVert \boldsymbol\nu_k \rVert^s-  \nu\rvert(B(0,1/2)\cap \delta_{1/r_k}(x_0^{-1}E)))\overset{\eqref{eq.simtalsc}}{\leq}\lvert \Phi_{\varepsilon_k}*\lVert\boldsymbol\nu_k\rVert^s -\nu\rvert(B(0,1/2)) +k^{-1}
      \label{eq.contraddizionefinale}
\end{split}
\end{equation}
Since $\lvert \Phi_{\varepsilon_k}*\lVert\boldsymbol\nu_k\rVert^s -\nu\rvert(B(0,1/2))=\lVert \Phi_{\varepsilon_k}*\lVert\boldsymbol\nu_k\rVert^s-v\rVert_{L^1(B(0,1/2))}$, we see that if $k$ is chosen small enough the inequality 
$$1/j_0\leq\lVert \Phi_{\varepsilon_k}*\lVert\boldsymbol\nu_k\rVert^s-v\rVert_{L^1(B(0,1/2))} +k^{-1},$$
cannot be satisfied thanks to the fact that $u_k\to v$ in $L^1(B(0,1/2))$ and to \eqref{eq.stabilesingolare}. This shows that the points where (i), (ii), (iii) and (iv) hold together form a $\Xi^s$-null set.

\medskip

Thanks to Radon-Nykodim decomposition, we can write $\mu$ as $\mu=\mu^a+\mu^s$, where $\mu^a\ll\mathcal{L}^n$ and $\mu^s\perp \mathcal{L}^n$ and it is elementary to see that $\mu^s\ll\Xi^s$ since $\mu\ll\Xi$. Since $\eta_i\ll \Xi$ there are $\alpha_i\in L^1(\Xi)$ such that $\eta_i=\alpha_i\Xi$. Hence, it is easy to see that 
\begin{equation}
\begin{split}
     \mathfrak{T}(x)=&\frac{d\mathbb{T}(x)}{d\Xi}=\Big(\frac{d(\tau^1\eta^1)}{d\Xi}(x),\ldots,\frac{d(\tau^{n_1}\eta^{n_1})}{d\Xi}(x)\Big)=\Big(\frac{d(\tau^1\alpha_1\Xi)}{d\Xi}(x),\ldots,\frac{d(\tau^{n_1}\alpha_{n_1}\Xi)}{d\Xi}(x)\Big)\\
     =&(\alpha_1(x)\tau_1(x),\ldots,\alpha_{n_1}(x)\tau_{n_1}(x))=(\alpha_1(x) e_1,\ldots,\alpha_{n_1}(x) e_{n_1}),
\end{split}
    \nonumber
\end{equation}
for $\Xi$-almost every $x\in \cap_{i=1}^{n_1}\{0<\lvert\alpha_i\rvert<\infty\}$. It is thus immediate to see that this implies
$$\mathfrak{T}(x)=(\alpha_1(x) e_1,\ldots,\alpha_{n_1}(x) e_{n_1})\qquad\text{for  $\sum_{i=1}^{n_1}\eta_i$-almost every $x\in \mathbb{G}$},$$
and $\cap_{i=1}^{n_1}\{0<\lvert\alpha_i\rvert<\infty\}$ is a set of full $\sum_{i=1}^{n_1}\eta_i$-measure. Finally, since $\mu\ll \sum_{i=1}^{n_1}\eta_i$, we conclude that 
$$\mathfrak{T}(x)=\mathrm{diag}(\alpha_1(x),\ldots,\alpha_{n_1}(x))\qquad\text{for $\mu$-almost every $x\in\GG$},$$
and $\alpha_1(x),\ldots,\alpha_{n_1}(x)\in \R\setminus\{0\}$ for $\mu$-almost every $x\in\GG$. However, thanks to the discussion above we know that 
$$\Xi^s(\{x\in\mathbb{G}:\mathfrak{T}(x)=\mathrm{diag}(\alpha_1(x),\ldots,\alpha_{n_1}(x))\text{ and }0<\lvert\alpha_i\rvert<\infty\text{ for any }i=1,\ldots,n_1\})=0.$$
This, however, concludes the proof of the fact that $\mu^s=0$. 
\end{proof}

\appendix

\section{Differentiation properties for non-doubling Radon measures}

In this section we prove that for any Radon measure $\nu$ and $\nu$-almost everywhere there exists a sequence of scales on which $\nu$ behaves like a doubling measure and along such sequence of scales a form of Lebesgue differentiation theorem holds. 
 
\begin{lemma}\label{l:weakdoubl}
Let $\nu $ be a non-negative Radon measure and let $t \in (0,1)$. There exists a $\nu$-measurable set $E^t_\nu$ with $\nu(E^t_\nu)=0$ such that for all $x \in \supp (\nu)\setminus E^t_\nu$ it holds that
\begin{equation}\label{eq:weakdoubl}
\limsup_{r\to 0}\frac{\nu(B(x,tr))}{\nu(B(x,r))} >0.
\end{equation}
\end{lemma}

\begin{proof}
Since the map $x\mapsto\nu(B(x,t^k))$ is upper semicontinuous, we infer that $\mathfrak{r}(x):=\limsup_{k\to\infty}\nu(B(x,t^{k+1})/\nu(B(x,t^k)))$ is Borel. Hence, the set $E^t_\nu:=\supp(\nu)\setminus\mathfrak{r}^{-1}(0)$ is Borel as well. Clearly, the set of points in $\supp(\nu)$ where \eqref{eq:weakdoubl} is not satisfied is contained in $E^t_\nu$.  

Assume by contradiction that $\nu(E^t_\nu)>0$. Since as seen above, the fractions $x\mapsto\nu(B(\cdot,t^{k+1}))/\nu(B(\cdot,t^k))$ are Borel functions, by Severini-Egorov's Theorem there exists a compact set $K\subset E_\nu^t$ with $\nu(K)>0$ where 
for every  $ \eps>0$ there exists $k_0\in\N$ such that 
 \[
 {\nu(B(x,t^{k+1}))} \leq \eps {\nu(B(x,t^k))}\qquad\text{for every } x\in K \text{ and every } k\geq k_0.
 \]
For any such $x\in K$ it therefore holds that 
\[
\nu(B(x,t^{k})) \leq \nu(B(x,t^{k_0})) \eps^{k-k_0}. 
\]
Since $K$ is compact, $\nu$ is Radon, and $t<1$, the right hand-side of the above inequality is uniformly bounded by some constant $C>0$. This implies by Remark \ref{rk.norm} that  $\nu(U(x,t^{sk})) \leq C \eps^{k-k_0}$, where here $U(x,r)$ denotes the closed euclidean ball with centre $x$ and radius $r$. 
Choosing $\delta<1$ and $\eps=\delta t^{sd}$, we obtain a sequence 
$r_k := t^{ks}$ such that 
\[
\nu(U(x,r_k)) \leq C\delta^{k-k_0} r_k^d. 
\] 
This last inequality can hold however only for a $\nu$-null set of points $x$, since the classical Vitali-Besicovitch theorem, applied to the covering $$\mathcal{U}:=\{U(x,r_k):x\in K\text{ and }k\geq k_0+2\},$$ provides a disjoint subcover $\{U(x_i,r_{k_i})\}_{i\in\N}$ of $\mathcal{U}$ such that 
\[
\nu(K) \leq C\delta^{k_1-k_0}\sum_{i\in\N} r_{k_i}^n \leq C\delta^{k_1-k_0}\mathcal{L}^n(U(K,1)) 
\]
where $U(K,1)$ denotes the closed Euclidean neighbourhood of radius $1$ of the compact set $K$. The arbitrariness of $\delta$ concludes that  $\nu(E^t_\nu)=0$. 
\end{proof}

\begin{proposizione}\label{p:differentiation}
Let $\nu$ be a Radon measure on $\mathbb{G}$ and let $f\in L^1(\nu)$. Then, for $\nu$-almost every $x\in\mathbb{G}$ there exists an infinitesimal sequence $r_k^x$ such that
\begin{equation}\label{eq:lebpt}
\lim_{k\to\infty}\fint_{B(x,r^x_k)} |f(y)-f(x)| d\nu(y) =0.
\end{equation}
\end{proposizione}

\begin{proof}Let us fix $t:=1/5$.
As a first step, let us define
$$E_j:=\{x\in E^t_\nu:\limsup_{k\to \infty}\frac{\nu(B(x,t^{k+1}))}{\nu(B(x,t^k))} >j^{-1}\}.$$
The set $E_j$ is easily seen to be Borel, thanks to the discussion contained in the proof of Lemma \ref{l:weakdoubl} and the arguments therein contained also prove that $\nu(\mathbb{G}\setminus \bigcup_{j\in\N}E_j)=0$. 

For any $\varphi\in L^1(\nu)$ us define the function $M\varphi:\cup_{j\in\N}E_j\to \R$ such that
$$M\varphi(x):=\sup_{k\in\N}\fint_{B(x,t^{k+1})}\lvert \varphi(y)\rvert d\nu(y),$$
Let us see that the function $M\varphi$ is $\nu$-measurable. It is easy to see that the map 
$x\mapsto \fint_{B(x,t^{k+1})}\lvert \varphi(y)\rvert d\nu(y)$,
is $\nu$-measurable and thus  $M\varphi$ is as well since it is the countable supremum of $\nu$-measurable maps. Let us now show that 
$$\nu(\{M\varphi>\lambda\}\cap E_j)\leq \frac{j}{\lambda}\int\lvert \varphi(y)\rvert d\nu(y).$$
For any $x\in\{M\varphi>\lambda\}\cap E_j$, there exists a $\mathfrak{k}=\mathfrak{k}(x,\lambda)\in\N$ such that
$$\fint_{B(x,t^\mathfrak{k+1})}\lvert \varphi(y)\rvert d\nu(y)>\lambda\qquad \text{and}\qquad \frac{\nu(B(x,t^{\mathfrak{k}+1}))}{\nu(B(x,t^\mathfrak{k}))}>j^{-1}.$$
This implies that $E_j$ is covered by the balls $\mathcal{C}:=\{B(x,t^{\mathfrak{k}(x)}):x\in E_j\}$ and thus by the classical Vitali covering lemma applied to the family of balls $\mathcal{C}$ we infer we can find countably many $y_i\in E_j$ such that the $B(y_i,t^{\mathfrak{k}(y_i)})$ cover $E_j$ and the balls $B(y_i,t^{\mathfrak{k}(y_i)+1})$ are disjoint. This implies 
\begin{equation}
\begin{split}
\nu(\{M\varphi>\lambda\}\cap E_j)\leq\sum_{i=1}^\infty \nu(B(y_i,t^{\mathfrak{k}(y_i)}))
    \leq j\sum_{i=1}^\infty \nu(B(y_i,t^{\mathfrak{k}(y_i)+1}))\leq \frac{j}{\lambda}\int\lvert\varphi(y)\rvert d\nu(y).
    \nonumber
\end{split}
\end{equation}
Let us define 
$$A_{\ell,j}:=\Big\{x\in E_j:\limsup_{k\to\infty}\fint_{B(x,t^k)}\lvert f(y)-f(x)\rvert d\nu(y)>\ell^{-1}\Big\},$$
observe that for    every continuous function $g$ we have 
\begin{equation}
    \begin{split}
        A_{\ell,j}\subseteq \{x\in E_j:M\lvert f-g\rvert>(2\ell)^{-1}\}\cup\{x\in E_j:\lvert f(x)-g(x)\rvert>(2\ell)^{-1}\},
    \end{split}
\end{equation}
This implies thanks to the above discussion that
\begin{equation}
    \nu(A_{\ell,j})\leq 2j\ell \int\lvert (f-g)(y)\rvert d\nu(y)+2\ell \int_{E_j}\lvert (f-g)(y)\rvert d\nu(y).
\end{equation}
Thanks to the arbitrariness of the choice of $g$, we finally conclude that $\nu(A_{\ell,j})=0$. This concludes the proof since the set of points where \eqref{eq:lebpt} holds in contained in the complement of  $\cup_{\ell,j\in\N} A_{\ell,j}$.
\end{proof}

%%%%%%%%%%%%%%%%%%%%%%%%%%%%%%%%%%%%%%%%%%%%%%%%%%%%%%%%%%%%%%%%%
%
%	BIBLIOGRAPHY
%
%%%%%%%%%%%%%%%%%%%%%%%%%%%%%%%%%%%%%%%%%%%%%%%%%%%%%%%%%%%%%%%%%

\printbibliography

\end{document}